\documentclass{imamat}

\gridframe{N}

\jno{dri017}

\usepackage{modamsthm,modnatbib}

\usepackage{latexsym,amssymb,lastpage}
\usepackage{graphicx,amsfonts}
\usepackage{times,mathptmx,bm,amsmath}
\usepackage{dcolumn}

\renewenvironment{proof}[1][Proof]{\noindent\textit{#1. } }{\hfill$\square$}

 \newtheoremstyle{theorem}{6pt}{6pt}{\rm}{}{\sffamily}{ }{ }{}
 \theoremstyle{theorem}
\newtheorem{theorem}{\sc Theorem}[section]

  \newtheoremstyle{thm}{6pt}{6pt}{\rm}{}{\sffamily}{ }{ }{}
 \theoremstyle{thm}

 \newtheoremstyle{lemma}{6pt}{6pt}{\rm}{}{\sffamily}{ }{ }{}
 \theoremstyle{lemma}
 \newtheorem{lemma}{\sc Lemma}[section]
 \newtheoremstyle{lem}{6pt}{6pt}{\rm}{}{\sffamily}{ }{ }{}
 \theoremstyle{lem}

\newtheoremstyle{case}{6pt}{6pt}{\rm}{}{}{. }{ }{}
 \theoremstyle{case}

 \newtheoremstyle{statement}{6pt}{6pt}{\rm}{}{\sffamily}{ }{ }{}
\theoremstyle{statement}

 \newtheoremstyle{corollary}{6pt}{6pt}{\rm}{}{\sffamily}{ }{ }{}
 \theoremstyle{corollary}

  \newtheoremstyle{defi}{6pt}{6pt}{\rm}{}{\sffamily}{ }{ }{}
 \theoremstyle{defi}

  \newtheoremstyle{cor}{6pt}{6pt}{\rm}{}{\sffamily}{ }{ }{}
 \theoremstyle{cor}

\newtheoremstyle{example}{6pt}{6pt}{\rm}{}{\sffamily}{ }{ }{}
\theoremstyle{example}

\newtheorem{notation}[theorem]{Notation}
\newtheorem{problem}{Problem}
\newtheorem{proposition}[theorem]{\sc Proposition}

\newtheoremstyle{remark}{6pt}{6pt}{\rm}{}{\sffamily}{ }{ }{}
\theoremstyle{remark}
\newtheorem{remark}{\sc Remark}[section]

\newtheoremstyle{approximation}{6pt}{6pt}{\rm}{}{\sffamily}{ }{ }{}
\theoremstyle{approximation}

\newtheoremstyle{scheme}{6pt}{6pt}{\rm}{}{\sffamily}{ }{ }{}
\theoremstyle{scheme}

\newtheoremstyle{Algorithm}{6pt}{6pt}{\rm}{}{\sffamily}{ }{ }{}
\theoremstyle{Algorithm}

 \newtheoremstyle{Remark}{6pt}{6pt}{\rm}{}{\sffamily}{ }{ }{}
 \theoremstyle{Remark}

\newtheoremstyle{Lemma}{6pt}{6pt}{\rm}{}{\sffamily}{ }{ }{}
\theoremstyle{Lemma}

\newtheoremstyle{Assumption}{6pt}{6pt}{\rm}{}{\sffamily}{ }{ }{}
\theoremstyle{Assumption}

\newtheoremstyle{Proposition}{6pt}{6pt}{\rm}{}{\sffamily}{ }{ }{}
\theoremstyle{Proposition}

\newtheoremstyle{prop}{6pt}{6pt}{\rm}{}{\sffamily}{ }{ }{}
\theoremstyle{prop}

\newtheoremstyle{rem}{6pt}{6pt}{\rm}{}{\sffamily}{ }{ }{}
 \theoremstyle{rem}

\newtheoremstyle{hypo}{6pt}{6pt}{\rm}{}{\sffamily}{ }{ }{}
 \theoremstyle{hypo}

  \newtheoremstyle{Step}{6pt}{6pt}{\rm}{}{}{ }{ }{}
 \theoremstyle{Step}

 \newtheoremstyle{lema}{6pt}{6pt}{\rm}{}{\sffamily}{ }{ }{}
 \theoremstyle{lema}

\renewcommand{\theequation}{\thesection.\arabic{equation}}
\numberwithin{equation}{section}

\def\d{\hbox{d}}

\begin{document}
\title{Degenerate Dirichlet Problems Related to the Ergodic Property of an Elasto-Plastic Oscillator Excited by a Filtered White Noise \thanks{This research was partially supported by a grant from CEA, Commissariat \`a l'\'energie atomique and by the National Science Foundation under grant [DMS-0705247]}}
\author{{\sc Alain Bensoussan}\\[2pt]
International Center for Decision and Risk Analysis, School of Management, University of Texas at Dallas, Box 830688, Richardson, Texas 75083-0688,\\[6pt] 
Graduate School of Business, the Hong Kong Polytechnic University,\\[6pt] 
Graduate Department of Financial Engineering, Ajou University.\\[6pt]
alain.bensoussan@utdallas.edu
\footnote{This research in the paper was supported by WCU (World Class University) program through the National Research Foundation of Korea funded by the Ministry of Education, Science and Technology [R31 - 20007].}\\[6pt]
{\sc Laurent Mertz}\\[2pt]
Laboratoire Jacques-Louis Lions,Universit\'e Pierre et Marie Curie,\\[6pt]
4 place jussieu, Paris, 75005 France.\\[6pt]
mertz@ann.jussieu.fr\\[6pt]
{\rm [Received on september 2011]}}
\pagestyle{headings}
\markboth{A. Bensoussan, L. Mertz}{\rm Degenerate Dirichlet problems Related to the Ergodic property of the elasto-plastic oscillator excited by a filtered white noise}
\maketitle

\begin{abstract}
{A stochastic variational inequality is proposed to model an elasto-plastic oscillator excited by a filtered white noise. We prove the ergodic properties of the process and characterize the corresponding invariant measure. This extends Bensoussan-Turi's method (Degenerate Dirichlet Problems Related to the Invariant Measure of Elasto-Plastic Oscillators, AMO, 2008) with a significant additional difficulty of increasing the dimension. 
Two points boundary value problem in dimension 1 is replaced by elliptic equations in dimension 2. In the present context, Khasminskii's method (Stochastic Stability of Differential Equations, Sijthoff and Noordhof,1980) leads to the study of degenerate Dirichlet problems with partial differential equations and nonlocal boundary conditions.}
\end{abstract}

\section{Introduction}
Nonlinear oscillators subjected to vibrations represent useful models for predicting the response of mechanical structures when stressed beyond the elastic limit. When the excitation is a white noise, it has received considerable interest over the past decades. In a previous work \cite{BenTuri1}, the authors have considered the response of a white noise excited elasto-plastic oscillator using a stochastic variational inequality formulation. The results in \cite{BenTuri1} provide a framework to assess the accuracy of calculations made in the literature (see e.g., \cite{Feau1,Feau2,KarSchar}, and the references therein). In this paper, instead of considering white noise input signal whose power spectral density (PSD) is constant, we consider an excitation with a non-constant PSD which could be a more realistic framework.  We consider the excitation as the velocity of a ``reflected" Ornstein-Uhlenbeck process. Therefore, comparing with the elasto plastic oscillator excited by white noise, a third process occurs in the variational inequality. 
Consider $w(t)$ and $\tilde{w}(t)$, two independent Wiener processes and $x(t)$ an Ornstein-Ulhenbeck reflected process
\begin{equation}\label{chap6:excitation}
\d {x}(t) =-\alpha x(t) \d t + \d w(t) + 1_{\lbrace x(t)=-L \rbrace} \d \xi_t^1 - 1_{\lbrace x(t)=L \rbrace}\d \xi_t^2. 
\end{equation}
We have $-L \leq x(t) \leq L$. In this model, the stochastic excitation is given by 
\begin{equation}\nonumber
-\beta x(t) \d t +  \d \tilde{w}(t).
\end{equation}
The stochastic variational inequality model is given by:
 \begin{equation}
\left\{
\begin{array}{lll}
\d x(t) =-\alpha x(t) \d t + \d w(t) + 1_{\lbrace x(t)=-L \rbrace} \d \xi_t^1 - 1_{\lbrace x(t)=L \rbrace}\d \xi_t^2,\\[2mm]
\d y(t) = -( \beta x(t) + c_0y(t) + kz(t))\d t + \d \tilde{w}(t),\\[2mm]
(\d z(t)-y(t) \d t)(\zeta - z(t))  \geq   0,\\[2mm]
|\zeta|   \leq  Y,\\[2mm]
|z(t)|  \leq  Y.
\end{array}
\right.
\label{chap6:vi}
\end{equation}
When $\beta \neq 0$, $x(t)$ is involved in the dynamic of $y(t)$ and then this model will be referred as \underline{the 2d case}.  Whereas if $\beta = 0$, $x(t)$ is not involved in the dynamic of $y(t)$ and then $(y(t),z(t))$ satisfy the elasto-plastic oscillator problem of \cite{BenTuri1} which will be referred as \underline{the 1d case}.\\
\begin{notation}
Introduce the operators
\begin{eqnarray*}&&
\cr&&
Au :=   \dfrac{1}{2} u_{yy}  + \dfrac{1}{2} u_{xx}  - \alpha x u_{x} - (\beta x + c_0y+kz)  u_{y} + y u_z,\\[2mm]
\cr&&
B_+u :=   \dfrac{1}{2} u_{yy} +  \dfrac{1}{2} u_{xx}  - \alpha x u_{x} - (\beta x+c_0y + kY)  u_{y},\\[2mm]
\cr&&
B_-u :=   \dfrac{1}{2} u_{yy} +  \dfrac{1}{2} u_{xx}  - \alpha x u_{x} - (\beta x+c_0y - kY)  u_{y}.\\[2mm]
\end{eqnarray*}
\end{notation}
The infinitesimal generator of the process $(x(t),y(t),z(t))$, denoted by $\Lambda$ is given by:
\begin{equation}\nonumber
\Lambda : \begin{array}{ccl} \phi & \mapsto & \left\{
    \begin{array}{lll}
 A   \phi  & & \mbox{ if } z \in ]-Y,Y[,\\
 B_{\pm}  \phi & & \mbox{ if } z=\pm Y,\pm  y>0.
    \end{array}
\right.
 \end{array}
\label{chap6:ig}
\end{equation}

\begin{notation}
\[
\mathcal{O} :=  (-L,L) \times \mathbb{R} \times (-Y,Y) ; \quad \mathcal{O}^+ := (-L,L) \times (0, +\infty) \times \lbrace Y \rbrace ; \quad \mathcal{O}^- :=  (-L,L) \times (-\infty,0) \times \lbrace -Y \rbrace.
\]
\end{notation}

As the main result of the paper we prove the following:
\begin{theorem}
There exists one and only one probability measure $\nu$ on $\mathcal{O} \cup \mathcal{O}^- \cup \mathcal{O}^+$ satisfying
\[
\int_{\mathcal{O}} A \phi \d \nu +  \int_{\mathcal{O}^-}  B_- \phi \d \nu + \int_{\mathcal{O}^+}  B_+ \phi \d \nu = 0, \quad \forall \phi \quad \textup{smooth}.
\]
Moreover, $\nu$ has a probability density function $m$ such that
\[
\int_{\mathcal{O}} m(x,y,z) \d x \d y \d z+\int_{\mathcal{O}^+} m(x,y,Y) \d x \d y+\int_{\mathcal{O}^-} m(x,y,-Y)\d x \d y = 1,
\]
where
\begin{itemize}
\item
$\{m(x,y,z), \quad (x,y,z) \in \mathcal{O}\}$ is the elastic component,
\item
$\{m(x,y,Y), \quad (x,y) \in \mathcal{O}^+\}$ is the positive plastic component,
\item
and $\{m(x,y,-Y), \quad (x,y) \in \mathcal{O}^-\}$ is the negative plastic component.
\end{itemize}
In addition, $m$ satisfies in the sense of distributions the following equation in $\mathcal{O}$,
\[
 \alpha \frac{\partial}{\partial x}[xm] + \frac{\partial}{\partial y}[(\beta x + c_0 y + kz)m] -y \frac{\partial m }{\partial z} + \frac{1}{2}\frac{\partial^2 m}{\partial x^2} + \frac{1}{2}\frac{\partial^2 m}{\partial y^2} = 0 \quad \mbox{ in } \mathcal{O}, 
\]
and on the boundary
\begin{eqnarray*}
ym+ \frac{\partial}{ \partial x}[ xm ] + \frac{\partial}{ \partial y}[ (\beta x + c_0y+kY) m ] + \dfrac{1}{2} \dfrac{\partial^2 m}{\partial x^2 } + \dfrac{1}{2} \dfrac{\partial^2 m}{\partial y^2 }  = 0 , \quad  \mbox{ in } \mathcal{O}^+\\[3mm]
-ym + \frac{\partial}{ \partial x}[ xm ] + \frac{\partial}{ \partial y}[ (\beta x + c_0y-kY) m ] + \dfrac{1}{2} \dfrac{\partial^2 m}{\partial x^2 } + \dfrac{1}{2} \dfrac{\partial^2 m}{\partial y^2 }  =  0, \quad  \mbox{ in }\mathcal{O}^-\\[3mm]
 m  =  0, \quad  \mbox{ in } (-L,L) \times (-\infty,0) \times \{ Y \} \cup (-L,L) \times (0,\infty) \times \{ -Y \}.
\end{eqnarray*}

\label{chap6:thm}
\end{theorem}
The proof will be based on solving a sequence of interior and exterior Dirichlet problems, which are interesting in themselves. We will put in parallel the 1d and 2d cases, in order to facilitate the reader's work.
In the 1d case, the variable $x$ disappears ($\beta=0$), we will still use the notation $A,B_+,B_-$ for the operators defined above without $x$.\\

Let us mention that our study presents a mathematical interest because it generalizes the method proposed by the first author and J. Turi \cite{BenTuri1} in the case of higher dimension. Non-local boundary conditions expressed in the form of differential equations in dimension 1 are replaced by elliptic partial differential equations (PDEs) in dimension 2. In the first case, there are two semi-explicit solutions. Thus, the non-local boundary conditions are reduced to two unknown numbers. In the second case, we do not know explicit formulas for solutions of the both elliptic on the boundary. In this context, these two solutions depend on two unknown functions respectively defined on the set $(-L,L)$.\\

In addition, the choice of the excitation \eqref{chap6:excitation} is also motivated by two technical considerations:
\begin{itemize}
\item
the first is to force $x(t)$ (through the processes $\xi^1,\xi^2$) to evolve in the compact set $[-L,L]$. Thus, as part of our proof, a compactness argument allows to show the ergodic property of the triple $(x(t),y(t),z(t))$. Note that this is not a problem in terms of applications, because if we choose $L$ large enough, then the process $x(t)$ is similar to an Ornstein-Ulhenbeck process.
\item
the second is the uncorrelation of $w(t)$ and $\tilde{w}(t)$. In our approach, based on PDEs associated to the triple $(x(t),y(t),z(t))$, we avoid the appearance of cross-derivative terms in the infinitesimal generator $\Lambda$. In the case where $w(t)$ and $\tilde{w}(t)$ are correlated, these cross-derivative terms yield more technical difficulties.
\end{itemize}
\section{The interior Dirichlet problem}
In this section, we prove existence and uniqueness to the homogeneous interior Dirichlet problem.
\subsection{Some background on the interior Dirichlet problem in the 1d case}
Let us recall the interior Dirichlet problem from \cite{BenTuri1}. Let $\bar{y}_1>0$,
\begin{notation}
\[
D_1:= (-\bar{y}_1,\bar{y}_1) \times (-Y,Y); \quad  D_{1}^+ := (0,\bar{y}_1) \times \lbrace Y \rbrace; \quad  D_{1}^- := (-\bar{y}_1,0) \times \lbrace -Y \rbrace
\]
and
\[
D_1^\epsilon := (-\bar{y}_1+ \epsilon,\bar{y}_1- \epsilon) \times (-Y,Y), \epsilon >0.
\]
\end{notation}
Denote $\bar{\tau}_1:= \inf \{ t > 0,  |y(t)| =  \bar{y}_1 \}$ and consider $\phi \in L^{\infty}(-Y,Y)$. We will use the following notation $\mathbb{E}_{p}(\cdot):= \mathbb{E}\{\cdot\ | (y(0),z(0)) = p\}$. It is shown that $\mathbb{E}_{(y,z)}(\phi(z(\bar{\tau}_1)) )$ solves a nonlocal Dirichlet problem:
Find $\eta \in L^\infty(D_1) \cap \mathcal{C}^0(D_1^\epsilon), \forall \epsilon >0$ such that
\begin{equation}\label{chap6:p1}
A \eta  = 0  \mbox{ in }  D_1, \quad  B_+ \eta =0   \mbox{ in }  D_1^+, \quad B_- \eta =0 \mbox{ in }  D_1^-
\end{equation}
with
\begin{equation*}
\eta(\bar{y}_1,z)  = \phi(z)  , \quad \eta(-\bar{y}_1,z)   =  0 , \quad  \mbox{ if }  z \in (-Y,Y). 
\end{equation*}
Since  $\eta(\bar{y}_1,Y) = \phi(Y)$ and $\eta(-\bar{y}_1,-Y) = 0$, there are semi-explicit solutions by solving ordinary differential equations for $\eta$ on the boundary at $z=Y$, and $z=-Y$ respectively,
\begin{equation}\nonumber
\eta(y,Y) = \eta_Y I(y,\bar{y}_1) + \phi(Y) I(0,y), \quad 0 < y \leq \bar{y}_1 ; \quad  \eta(y,-Y) = \eta_{-Y} I(-\bar{y}_1,y), \quad -\bar{y}_1 \leq y < 0,
\end{equation}
where $\eta_Y$ and $\eta_{-Y}$ are constants, and 
\[
I(a,b):= \frac{\int_a^b \exp(c_0 \lambda^2 + 2 k Y \lambda) \d \lambda}{\int_0^{\bar{y}_1} \exp(c_0 \lambda^2 + 2 k Y \lambda) \d \lambda}.
\]
The nonlocal condition is restricted to the value of these two constants. Based on these semi-explicit expressions a subset $K$ of $H^1(D_1)$ is defined for proving existence of the solution to (\ref{chap6:p1}) , by
\[
K:=\left\{
\begin{array}{l}
v \in H^1(D_1),\\[2mm]
v(-\bar{y}_1,z) = \phi(z) , \quad v(-\bar{y}_1,z)  =  0,\\[2mm]
v(y,Y)   =  v_Y I(y,\bar{y}_1) + \phi(Y) I(0,y), \quad 0 \leq y \leq \bar{y}_1\\[2mm]
v(y,-Y)  =  v_{-Y} I(-\bar{y}_1,y), \quad -\bar{y}_1 \leq y \leq 0\\[2mm]
v_Y , v_{-Y} \quad \mbox{are constant with} \quad  |v_{\pm Y} | \leq \| \phi \|_{L^\infty}.  
\end{array}
\right.
\]
The set $K$ is convex and not empty if $\phi(z) \in H^1(-Y,Y)$. We take $v(y,z) = \phi(z)I(0,y) \mathbf{1}_{\{y>0\}}$.
\subsection{The interior Dirichlet problem in the 2d case}
\begin{notation}
\[
\Delta_1 := (-L,L)  \times (-\bar{y}_1,\bar{y}_1) \times (-Y,Y),
\]
\[
\Delta_1^+ := (-L,L)  \times (0,\bar{y}_1) \times \lbrace Y \rbrace, \quad \Delta_1^- := (-L,L)  \times (-\bar{y}_1,0) \times \lbrace -Y \rbrace
\]
and
\[
 \Delta_1^\epsilon := (-L,L)  \times (-\bar{y}_1 + \epsilon,\bar{y}_1-\epsilon) \times (-Y,Y), \forall \epsilon >0.
\]
\end{notation}
Denote $\bar{\tau}_1:= \inf \{ t > 0,  |y(t)| =  \bar{y}_1 \}$ and consider $\phi \in L^{\infty}((-L,L) \times (-Y,Y))$. Similarly as before we use the notation $\mathbb{E}_{p}(\cdot):= \mathbb{E}\{\cdot\ | (x(0),y(0),z(0)) = p\}$. We want to define $\mathbb{E}_{(x,y,z)}(\phi(x(\bar{\tau}_1), z(\bar{\tau}_1)))$ as the solution of the interior Dirichlet problem stated below.
\begin{problem}\label{chap6:pb1}
Find $\eta \in L^\infty(\Delta_1) \cap \mathcal{C}^0(\Delta_1^\epsilon), \forall \epsilon >0$ such that
\begin{equation}\nonumber
A \eta  = 0  \mbox{ in }  \Delta_1, \quad  B_+ \eta = 0   \mbox{ in }  \Delta_1^+, \quad B_- \eta = 0 \mbox{ in }  \Delta_1^-
\end{equation}
and
\begin{eqnarray*}
\begin{array}{rclcl}
\eta_x(\pm L,y,z) & =&  0 & \mbox{ in } & (y,z) \in (-\bar{y}_1,\bar{y}_1) \times (-Y,Y),\\ 
\eta(x,\bar{y}_1,z) & = & \phi(x,z) & \mbox{ in } &  (x,z) \in (-L,L) \times (-Y,Y),\\ 
\eta(x,-\bar{y}_1,z)&  = & 0 & \mbox{ in } &  (x,z) \in (-L,L) \times (-Y,Y). 
\end{array} 
\end{eqnarray*}
\end{problem}
This is formal. We should consider the case of $\phi$ smooth first and precise the functional space, then proceed with the regularization.

As in the 1d-case, this problem is a nonlocal problem but the boundary condition are in two dimensions. Thus we need to solve partial differential equations for $\eta$ on the boundary at $z=Y$, and $z=-Y$ respectively. Here we do not have semi-explicit solution, indeed 
$\eta(x,y,Y)$ solves
\begin{equation}\label{chap6:etaplus}
 B_+ \eta =0 \mbox{ on } (-L,L) \times (0, \bar{y}_1) \mbox{ with } \eta_x(\pm L,y,Y) = 0, \quad \eta(x,\bar{y}_1,Y) = \phi(x,z)
 \end{equation}
with $\eta(x,0,Y)=\eta_Y(x)$,
and $\eta(x,y,-Y)$ solves 
\begin{equation}\label{chap6:etamoins}
B_- \eta =0 \mbox{ on } (-L,L) \times (-\bar{y}_1, 0) \mbox{ with } \eta_x(\pm L,y,-Y) = 0, \quad \eta(x,-\bar{y}_1,-Y) = 0
\end{equation}
with $\eta(x,0,-Y)=\eta_{-Y}(x)$ where $\eta_Y(x)$ and $\eta_Y(x)$ are unknown function with $\| \eta_{\pm Y} \|_{L^\infty} \leq \| \phi \|_{L^\infty}$.  Next, we give a convenient formulation of the boundary condition (\ref{chap6:etaplus})-(\ref{chap6:etamoins}).

\subsubsection{Boundary conditions}
In order to reformulate Problem \ref{chap6:pb1}, we consider first the equation (\ref{chap6:etaplus}) on the boundary $\Delta_1^+$.
Define $\beta^+(x,y)$ the solution of the mixed Dirichlet-Neuman problem
\begin{equation}\label{chap6:betaplus}
\left\{
\begin{array}{rrll}
B_+ \beta^+               & = & 0, &  \mbox{ in } (-L,L) \times (0,\bar{y}_1),\\[3mm]
\beta_x^+( \pm L,y)  & = & 0, &  \mbox{ in } (0,\bar{y}_1),\\[3mm]
\beta^+(x,\bar{y}_1) & = & \phi(x,Y), &  \mbox{ in } (-L,L),\\[3mm]
\beta^+(x,0)               & = & 0, &  \mbox{ in } (-L,L).
\end{array}
\right. 
\end{equation}

\begin{proposition}\label{chap6:prop:betaplus}
If $\phi(x,Y) \in H^1(-L,L)$ then there exists a unique solution to the equation \eqref{chap6:betaplus}
\[
\beta^+ \in H^1((-L,L) \times (0, \bar{y}_1))  \quad \mbox{satisfying} \quad \| \beta^+ \|_{L^\infty} \leq \| \phi \|_{L^\infty}. 
\]
\end{proposition}
\begin{proof}
Define $D_Y^+ := (-L,L) \times (0,\bar{y}_1)$ and consider on $H^1(D_Y^+)$ the bilinear form
\[
b_Y(\xi,\chi)= \frac{1}{2} \int_{D_Y^+} \big{(}  \xi_x \chi_x + \xi_y \chi_y \big{)} \d x \d y + \int_{D_Y^+} \big{(} \alpha x \xi_x + (c_0y+kY+\beta x) \xi_y \big{)} \chi \d x \d y. 
\]
For $\lambda$ sufficiently large $b_Y(\xi,\chi) + \lambda (\xi,\chi)$ is coercive. Now, define the convex set
\[
K_Y = \lbrace \xi \in H^1(D_Y^+), \xi(x,\bar{y}_1) = \phi(x,Y) \mbox{ and } \xi(x,0) = 0 \rbrace
\] 
which is not empty since $\phi(x,Y) \in H^{\frac{1}{2}}(-L,L)$. Take $z \in K_Y$ with $\| z \|_{L^\infty} \leq \| \phi \|_{L^\infty}$, we define $\xi_\lambda$ to be the unique solution of 
\begin{equation}
b_Y(\xi_\lambda, \chi-\xi_\lambda) + \lambda (\xi_\lambda, \chi - \xi_\lambda) \geq \lambda (z,\chi-\xi_\lambda), \qquad \xi_\lambda \in K_Y, \forall \chi \in K_Y
\label{chap6:p4}
\end{equation}
We have defined a map $T_\lambda(z)=\xi_\lambda$ from $K_Y \to K_Y$. Let us check that 
\begin{equation}
\| \xi_\lambda \|_{L^\infty} \leq \| \phi \|_{L^\infty} = C.
\label{chap6:p5}
\end{equation}
Indeed, in (\ref{chap6:p4}) we take $\chi = \xi_\lambda - (\xi_\lambda-C)^+ \in K_Y$. Hence
\[
-b_Y(\xi_\lambda, (\xi_\lambda - C)^+) - \lambda (\xi_\lambda, (\xi_\lambda-C)^+) \geq \lambda (z,(\xi_\lambda - C)^+)\\
\]
and
\[
b_Y((\xi_\lambda - C)^+, (\xi_\lambda - C)^+) + \lambda |(\xi_\lambda - C)^+|_{L^2} \leq -\lambda (z+C,(\xi_\lambda - C)^+).
\]
Since $z+C \geq 0$ it follows that $(\xi_\lambda - C)^+=0$, hence $\xi_\lambda \leq C$. Similarly, we check that $(-\xi_\lambda -C)^+=0$, hence we have (\ref{chap6:p5}). Consider then the sequence $\xi_\lambda^n$ defined by 
\begin{equation}
b_Y(\xi_\lambda^{n+1}, \chi-\xi_\lambda^{n+1}) + \lambda (\xi_\lambda^{n+1}, \chi- \xi_\lambda^{n+1}) \geq \lambda (\xi_\lambda^{n}, \chi-\xi_\lambda^{n+1})
\label{chap6:p6}
\end{equation}
with 
\[ 
\xi_\lambda^{0} \in K_Y, \| \xi_\lambda^0 \|_{L^\infty} \leq \| \phi \|_{L^\infty}.
\]
We can take $\xi_\lambda^{0}(x,y) = \frac{y}{\bar{y}_1} \phi(x,Y)$. From (\ref{chap6:p5}), we have $\| \xi_\lambda^n \|_{L^\infty(D_Y^+)} \leq C$ and from (\ref{chap6:p6}), we have $\| \xi_\lambda^n \|_{H^1(D_Y^+)} \leq C'$.
Then, we can consider a subsequence, also denoted by $\xi_\lambda^n$ such that 
\[
\xi_\lambda^n \to \xi \qquad \mbox{ in } H^1(D_Y^+) \mbox{ weakly and in } L^\infty(D_Y^+) \mbox{ weakly } \star
\]
also,
\[
\xi_\lambda^n \to \xi \qquad \mbox{ in } L^2(D_Y^+) \mbox{ strongly. }
\]
From (\ref{chap6:p6}) we obtain
\begin{equation}\nonumber
b_Y(\xi, \chi - \xi) \geq 0, \qquad \xi \in K_Y, \forall \chi \in K_Y.
\end{equation}
We conclude easily that $\xi$ is a solution of (\ref{chap6:betaplus}) also $\| \xi \|_{L^\infty} \leq \| \phi \|_{L^\infty}$.
Now, in order to prove uniqueness, we must prove that a solution $\xi \in H^1(D_Y^+)$ satisfying
\begin{equation}\label{chap6:p7}
\begin{array}{c}
B_+ \xi = 0 \mbox{ in } (-L,L) \times (0,\bar{y}_1)\\[2mm]
\xi_x(-L,y) = \xi_x(L,y)=0\\[2mm]
\xi(x,\bar{y}_1) = \xi(x,0)=0
\end{array}
\end{equation}
is identically zero.\\
Consider $\chi:=\xi_x$ then we have 
\begin{equation}\nonumber
\begin{array}{c}
B_+ \chi + \alpha \chi + \beta \xi_y = 0 \mbox{ in } D_Y^+,\\[2mm]
\chi(-L,y) = \chi(L,y)=0,\\[2mm]
\chi(x,\bar{y}_1) = \chi(x,0)=0,
\end{array}
\end{equation}
hence $\chi \in H_0^1(D_Y^+)$. This implies 
\[
\xi_{xx} \in L^2(D_Y^+), \qquad \xi_{yx} \in L^2(D_Y^+).
\] 
From equation (\ref{chap6:p7}) we deduce $\xi_{yy} \in L^2(D_Y^+)$. Hence $\xi \in H^2(D_Y^+)$. In particular, $\xi$ is continuous on $\bar{D_Y^+}$. We have $\| \xi \|_{L^\infty} = 0$, so $\xi =0$.
\end{proof}

Now, consider the following convex sets,
\begin{equation}
K_{Y,M} := \lbrace \chi^+ \in H^1((-L,L) \times (0,\bar{y}_1)) \mbox{ satisfying } (\ref{chap6:pKYM}) ; \quad  \chi^+(x,\bar{y}_1)=0 ; \quad \| \chi^+ \|_{L^\infty} \leq M \rbrace\\[2mm]
\label{chap6:KYM}
\end{equation}
where
\begin{equation}
\begin{array}{l}
\int_{-L}^{L} \int_{0}^{\bar{y}_1} \lbrace \frac{1}{2}(\chi_x^+ \psi_x + \chi_y^+ \psi_y) + (\alpha x \chi_x^+ + (\beta x + c_0 y + kY) \chi_y^+ )\psi \rbrace \d x \d y=0,\\[3mm]
 \forall \psi \in H^1((-L,L)\times(0,\bar{y}_1)) \mbox{ with } \psi(x,0)=\psi(x,\bar{y}_1)=0.\\[3mm]
\end{array}
\label{chap6:pKYM}
\end{equation}
and
\begin{equation}
K_{-Y,M} := \lbrace \chi^- \in H^1((-L,L) \times (-\bar{y}_1,0)) \mbox{ satisfying } (\ref{chap6:pKmYM})  ; \quad  \chi^-(x,-\bar{y}_1)=0 ; \quad \| \chi^- \|_{L^\infty} \leq M \rbrace\\[3mm]
\label{chap6:KmYM}
\end{equation}
where
\begin{equation}
\begin{array}{l}
\int_{-L}^{L} \int_{-\bar{y}_1}^{0} \lbrace \frac{1}{2}(\chi_x^- \psi_x + \chi_y^- \psi_y) + (\alpha x \chi_x^- + (\beta x + c_0 y - kY) \chi_y^- )\psi \rbrace \d x \d y=0,\\[3mm]
 \forall \psi \in H^1((-L,L)\times(-\bar{y}_1,0)) \mbox{ with } \psi(x,0)=\psi(x,-\bar{y}_1)=0.\\[3mm]
\end{array}
\label{chap6:pKmYM}
\end{equation}
These sets are not empty since they contain $0$.

\begin{remark}
\begin{itemize}
\item
\label{chap6:rem:r1}
$\forall \pi(y)$ function of $y$ such that $\pi(0)=0$ ,
\[
  \lbrace \chi  \pi  ; \quad \chi \in K_{Y,M} \rbrace \quad \mbox{ and } \quad 
  \lbrace \chi  \pi  ; \quad \chi \in K_{-Y,M} \rbrace
  \quad \mbox{ are } H^1-\mbox{bounded}.
\]
\item
If $\chi \in K_{Y,\| \phi \|_{L^\infty}}$, denoting $\omega :=\beta^+ + \chi$, we have
\[
B_+ \omega = 0 \qquad \mbox{ and } \qquad \max (\vert \omega(x,0) \vert, \vert \omega(x,\bar{y}_1)\vert ) \leq \| \phi \|_{L^\infty}, 
\]
so a maximum principle implies $\| \omega \|_{L^\infty} \leq \| \phi \|_{L^\infty}$.
\end{itemize}
\end{remark}

Using the sets $K_{Y,\| \phi \|}$ and $K_{-Y,\| \phi \|}$, the following result gives a convenient formulation of the boundary conditions in Problem \ref{chap6:pb1}.

\begin{proposition}
The Problem \ref{chap6:pb1} can be reformulated in the following way: find $\eta \in L^\infty(\Delta_1) \cap \mathcal{C}^0(\Delta_1^\epsilon), \forall \epsilon >0$ such that
\begin{equation}\nonumber
A \eta = 0  \mbox{ in } \Delta_1, \quad \eta(x,y,Y) - \beta^+(x,y) \in K_{Y,\| \phi \|}, \quad \eta(x,y,-Y)  \in K_{-Y,\| \phi \|}
\end{equation}
and
\begin{eqnarray*}
\begin{array}{rclcl}
\eta_x(\pm L,y,z)     &  =  &  0,             &   \mbox{ in }  &  (y,z) \in (-\bar{y}_1,\bar{y}_1) \times (-Y,Y),\\ 
\eta(x,\bar{y}_1,z)   &  =  & \phi(x,z),  &   \mbox{ in }  &  (x,z) \in (-L,L) \times (-Y,Y),\\ 
\eta(x,-\bar{y}_1,z)  &  =  & 0,              &   \mbox{ in }  &  (x,z) \in (-L,L) \times (-Y,Y). 
\end{array} 
\end{eqnarray*}
\label{chap6:formulation2}
\end{proposition}

\begin{proof}
First, we can obtain the generic solution of (\ref{chap6:etaplus}) by considering any function $\chi^+$ which satisfies 
\begin{equation*}
\chi^+ \in H^1((-L,L) \times (0,\bar{y}_1)) 
\end{equation*}
and
\begin{equation}\label{chap6:p8} 
B_+ \chi^+  =  0 , \quad \chi_x^+(\pm L,y)  = 0 , \quad \chi^+(x,\bar{y}_1)  =  0.
\end{equation}
Note that we have not defined the value of $\chi^+$ for $y=0$, hence $\chi^+$ is certainly not unique. 
We add the condition that $\chi^+$ is bounded by $\| \phi \|_{L^\infty}$. We define in a similar way the function $\chi^-$ such that 
\begin{equation*}
\chi^- \in H^1((-L,L) \times (-\bar{y}_1,0))
\end{equation*}
and 
\begin{equation}\label{chap6:p14}
B_- \chi^- =  0 , \quad  \chi_x^-( \pm L,y) = 0 , \quad \chi^-(x,-\bar{y}_1) = 0. 
\end{equation}
Then, interpreting (\ref{chap6:p8}) as (\ref{chap6:pKYM}) and (\ref{chap6:p14}) as (\ref{chap6:pKmYM}) repectively, we obtain $\chi^+  \in K_{Y,\| \phi \|}$ and $\chi^-  \in K_{-Y,\| \phi \|}$.
Hence, the set of solutions of (\ref{chap6:etaplus}) and (\ref{chap6:etamoins})  can be written as follows
\begin{equation}\nonumber
\eta(x,y,Y)-\beta^+(x,y) \in K_{Y,\| \phi \|}, y >0 ; \quad \eta(x,y,-Y)  \in K_{-Y,\| \phi \|}, \quad  y<0.
\end{equation} 
\end{proof}

\subsubsection{Approximation (part 1)}
We study Problem \ref{chap6:pb1} by a regularization method in the next proposition. Define $A^\epsilon := A + \frac{\epsilon}{2} \frac{\partial^2}{\partial z^2}$.  

\begin{proposition}\label{chap6:prop:p26}
The following problem: find $\eta^\epsilon \in L^{\infty}(\Delta_1) \cap H^1(\Delta_1)$ such that $\| \eta^\epsilon \|_{L^\infty} \leq \| \phi \|_{L^{\infty}}$,
\begin{equation}\label{chap6:EPSILON}
A^\epsilon \eta^\epsilon  = 0  \mbox{ in }   \Delta_1, \quad  \eta^\epsilon(x,y,Y) - \beta_{+}(x,y) \in K_{Y,\| \phi \|}, \quad \eta^\epsilon(x,y,-Y) \in K_{-Y,\| \phi \| }
\end{equation}
and
\begin{eqnarray*}
\begin{array}{lclcl}
\eta_x^\epsilon(\pm L,y,z)     & = & 0,                   & \mbox{ in } & (y,z) \in (-\bar{y}_1,\bar{y}_1) \times (-Y,Y),\\ 
\eta_z^\epsilon(x,y,\pm Y)     & =&  0,                  & \mbox{ in } & (x, \mp y) \in (-L,L) \times (0,\bar{y}_1),\\ 
\eta^\epsilon(x,\bar{y}_1,z)   & = & \phi(x,z),      & \mbox{ in } &  (x,z) \in (-L,L) \times (-Y,Y),\\ 
\eta^\epsilon(x,-\bar{y}_1,z)  & = & 0,                  & \mbox{ in } &  (x,z) \in (-L,L) \times (-Y,Y), 
\end{array} 
\end{eqnarray*} 
has a unique solution.
\end{proposition}
As in the 1d case, we formulate a variational inequality to prove existence of solutions. In the present context, we consider a convex subset of $H^1(\Delta_1)$ which is adapted to the two dimensional boundary condition by
\begin{equation}\nonumber
K= \left\{ 
\begin{array}{ll}
\psi \in H^1(\Delta_1), \quad \| \psi \|_{\infty} \leq \| \phi \|_{\infty}, \\
\psi(x,\bar{y}_1,z)  =   \phi(x,z), \quad \mbox{ for }  (x,z) \in (-L,L) \times (-Y,Y),\\
\psi(x,-\bar{y}_1,z)  =  0,  \quad \mbox{ for } (x,z) \in (-L,L) \times (-Y,Y),\\
\psi(.,.,Y) -\beta^+(x,y) \in K_{Y,\| \phi \|_{L^\infty}},\\
\psi(.,.,-Y) \in K_{-Y,\| \phi \|_{L^\infty}}.\\
\end{array}
\right.
\end{equation}
\begin{proposition}
The set $K$ is a closed non-empty subset of $H^1(\Delta_1)$.
\end{proposition}
\begin{proof}
The fact that $K$ is closed follows from the continuity of the trace operator. Now, pick the function,
\begin{equation}
\psi(x,y,z)= \frac{y}{\bar{y}_1} \big{[} \phi(x,z) - \frac{z}{2 Y}\phi(x,Y) - \frac{1}{2} \phi(x,Y)  \big{]} 1_{\lbrace y \geq 0 \rbrace} + (\frac{z}{2 Y} + \frac{1}{2}) \beta_{+}(x,y)1_{\lbrace y \geq 0 \rbrace}. 
\label{chap6:p19}
\end{equation}
We have 
\begin{eqnarray*}
\psi(x,\bar{y}_1,z) & = & \phi (x,z),\\
\psi(x,-\bar{y}_1,z) & = & 0,\\
\psi(x,y,Y) & = & \beta_+(x,y), \quad y > 0,\\
\psi(x,y,-Y) & = & 0, \quad y < 0.
\end{eqnarray*}
So, if $\phi(x,z) \in H^1(\partial \Delta_1)$, $\phi(x,Y) \in H^1(-L,L)$, then the function $\psi$ defined by (\ref{chap6:p14}) belongs to $K$.
\end{proof}
%

\begin{remark}
If $u \in K$ and $w \in H^1(\Delta_1)$ with $w(x,\pm \bar{y}_1,z) = 0$ and $w(x,y,\pm Y) = 0$, for $0< \pm y < \bar{y}_1$ then $u+w \in K$. 
\end{remark}

Consider the bilinear form

\begin{eqnarray*}
a(u,v) & = &  \frac{1}{2} \int_{\Delta_1} \lbrace \epsilon u_z v_z + u_yv_y + u_xv_x \rbrace \d x \d y \d z\\[2mm]
& & + \int_{\Delta_1} (\beta x+c_0y+kz)u_y v \d x \d y \d z - \int_{\Delta_1} y u_z v   \d x \d y \d z  +  \int_{\Delta_1} \alpha x u_x v   \d x \d y \d z.
\end{eqnarray*}

Equation \eqref{chap6:EPSILON} is formulated as follows $a(u,v-u) \geq 0, \quad \forall v \in K, u \in K$.

\begin{proof}[Proof of Proposition \ref{chap6:prop:p26}]
First, existence is proved by variational argument.
For $\lambda$ sufficiently large $a(u,v) + \lambda (u,v)$ is coercive on $H^1(\Delta_1)$ and for $f \in L^2(\Delta_1)$ we can solve the variational inequality
\[
a(u,v-u) + \lambda (u,v-u)\geq (f,v-u),  \qquad \forall v \in K, u \in K.
 \]
We define the map $u = T_\lambda w$ where $u$ is the unique solution of   
\[
a(u,v-u) + \lambda (u,v-u)\geq \lambda (w,v-u),  \qquad \forall v \in K, u \in K.
 \]
 The following lemma shows that $T_\lambda$ is a contracting map. (see proof in Appendix)
\begin{lemma}\label{chap6:lem:l1}
If $\| w \|_{L^\infty} \leq \| \phi \|_{L^\infty}$ then $\| u \|_{L^\infty} \leq \| \phi \|_{L^\infty}$.
\end{lemma}
Moreover, taking $v=u_0 \in K$ we deduce
\[
a(u,u) + \lambda |u|_{L^2}^2 \leq a(u,u_0) + \lambda (u,u_0) + \lambda  \gamma |u_0 - u |_{L^1}.
\]
That implies 
\[
\|u\|_{H^1(\Delta_1)} \leq M,  
\]
for a constant $M$ which depends only of $\lambda, \epsilon$, $H^1$ norm of $u_0$ and $\gamma$.\\  
Now, we define
\[
 \bar{K}= \lbrace w \in K : \| w \|_{L^\infty} \leq  \| \phi \|_{L^\infty}, \| w \|_{H^1} \leq  M  \rbrace.
\] 
We have $T_\lambda:L^2(\Delta_1) \to L^2(\Delta_1)$ is continuous, maps $\bar{K}$ into itself and $\bar{K}$ is a compact subset of $L^2$.
Then Schauder's theorem implies $T_\lambda$ has a fixed point $u \in \bar{K}$ which satisfies
\[
a(u,v-u) \geq 0, \qquad \forall v \in K.
\]
Let us check that $u$ is solution of (\ref{chap6:EPSILON}). As $u \in K$, we have
\[
\begin{array}{rclcl}
B_{+}u & = &0 & \mbox{ in } & \Delta^{+}\\
B_- u & = & 0 & \mbox{ in }& \Delta^{-}\\
u(x,\bar{y}_1,z)& = & \phi(x,z)& \mbox{ in }  & (-L,L) \times (-Y,Y)\\
u(x,-\bar{y}_1,z)& = & 0 & \mbox{ in } & (-L,L) \times (-Y,Y)\\
\end{array}
\]
Moreover,
\[
\forall v \in \mathcal{H}:=\left \{  v \in H^1(\Delta_1) \mbox{ such that } 
\left\{
 \begin{array}{l} 
v(x,\pm \bar{y}_1,z) = 0\\
v(x,y,\pm Y) = 0, \quad 0< \pm y < \bar{y}_1
\end{array}
\right. 
\right \},
\]
we have $u+v \in K \mbox{ and }  a(u,v) \geq 0$. Then, 
\[
A^\epsilon u = 0 \mbox{ in the sense } \mathcal{H}'
\]
and integration by parts gives
\begin{eqnarray*}
& & \frac{\epsilon}{2} \int_{-L}^L \int_{-\bar{y}_1}^0 u_z(x,y,Y) v(x,y,Y) \d x \d y - \frac{\epsilon}{2}  \int_{-L}^L \int_0^{\bar{y}_1} u_z(x,y,-Y) v(x,y,-Y) \d x \d y\\ 
& & + \int_{-\bar{y}_1}^{\bar{y}_1} \int_{-Y}^{Y} u_x(L,y,z) v(L,y,z) \d y \d z -  \int_{-\bar{y}_1}^{\bar{y}_1} \int_{-Y}^{Y}  u_x(-L,y,z) v(-L,y,z) \d y \d z =0.\\
\end{eqnarray*}
Uniqueness of the solution to problem \eqref{chap6:EPSILON} comes from $\| u \|_{L^\infty} \leq \| \phi \|_{L^\infty}$.
\end{proof}
\subsubsection{Approximation (part 2)}
When $\phi$ is smooth, we can exhibit a solution to the problem (\ref{chap6:EPSILON_LIMIT}) by extracting a converging subsequence of $\eta^\epsilon$.
Let $\theta$ a smooth fonction such that $\theta(\pm \bar{y}_1)=0$. Denote $\pi(y):=y^p \theta(y)^q$ for some $p,q$.

\begin{proposition}\label{chap6:prop:etaepsilon}
The following problem: find $\eta \in L^{\infty}(\Delta_1)$ such that $\eta \pi \in H^1(\Delta_1)$,
\begin{equation}\label{chap6:EPSILON_LIMIT}
A \eta  = 0  \mbox{ in }   \Delta_1, \quad  \eta(x,y,Y) - \beta_{+}(x,y) \in K_{Y,\| \phi \|}, \quad \eta(x,y,-Y) \in K_{-Y,\| \phi \| }
\end{equation}
and
\begin{eqnarray*}
\begin{array}{lclcl}
\eta_x(\pm L,y,z) & = & 0 & \mbox{ in } & (y,z) \in (-\bar{y}_1,\bar{y}_1) \times (-Y,Y),\\ 
\eta(x,\bar{y}_1,z) & = & \phi(x,z) & \mbox{ in } &  (x,z) \in (-L,L) \times (-Y,Y),\\ 
\eta(x,-\bar{y}_1,z) & =&  0 & \mbox{ in } &  (x,z) \in (-L,L) \times (-Y,Y), 
\end{array} 
\end{eqnarray*} 
has a unique solution.
\end{proposition}
As in the 1d case, the key ingredient of the proof is to bound uniformly the norm of first derivative w.r.t. $z$ using the auxiliary function $\pi$.

\begin{proof} 
From the previous section, we have  $\| \eta^\epsilon \|_{\infty} \leq \| \phi \|_{\infty}$, hence $\eta^\epsilon \to \eta$ in $L^{\infty} \star$. We also have $a(\eta^\epsilon,u_0-\eta^\epsilon) \geq 0$, for some $u_0 \in K$. So, we deduce estimates in the following lemma: (see proof in Appendix)
 
\begin{lemma}\label{chap6:lem:l2}
We have
\begin{equation}\nonumber
\epsilon \int_{\Delta_1} (\eta_z^\epsilon)^2 \d x \d y \d z \leq C ; \quad \int_{\Delta_1} (\eta_x^\epsilon)^2 \d x \d y \d z \leq C   ; \quad \int_{\Delta_1} (\eta_y^\epsilon)^2 \d x \d y \d z \leq C. 
\end{equation}
\end{lemma}

It is licit to test (\ref{chap6:EPSILON}) with $\eta^\epsilon_z y^{2p-1} \theta^{2q}$. We have
\[
\int_{\Delta_1} \big ( \frac{\epsilon}{2} \eta_{zz}^\epsilon + \frac{1}{2} \eta_{yy}^\epsilon + \frac{1}{2}\eta_{xx}^\epsilon + y \eta_z^\epsilon - \alpha x \eta^\epsilon_x - (\beta x + c_0 y + k z)\eta_y^\epsilon \big ) \eta_z^\epsilon y^{2p-1} \theta^{2q} = 0.
\]
So, we obtain
\begin{equation}\label{chap6:uzboundEPSILON}
 \begin{array}{rcl}
  & &  \int_{\Delta_1} (\eta_z^\epsilon \pi )^2  \leq\\[3mm] 
  &  & \frac{1}{4} \int_0^L \int_0^{\bar{y}_1} \lbrace (\eta_y^\epsilon(x,y,Y))^2 + (\eta_x^\epsilon(x,y,Y))^2 \rbrace y^{2p-1} \theta^{2q} \d x \d y \\[3mm]
  & + & \frac{1}{4} \int_0^L \int_{-\bar{y}_1}^0 \lbrace (\eta_y^\epsilon(x,y,-Y))^2 + (\eta_x^\epsilon(x,y,-Y))^2 \rbrace | y |^{2p-1} \theta^{2q} \d x \d y\\[3mm]
  & + & \frac{1}{2} \int_{\Delta_1}\eta_y^\epsilon \eta_z^\epsilon \big{(} ((2p-1)y^{2p-2} \theta^{2q} + y^{2p-1} 2q \theta' \theta^{2q-1} ) + 2 (\beta x + c_0 y + kz)y^{2p-1} \theta^{2q} \big{)}   \d x \d y \d z\\[3mm]
  & + &  \int_{\Delta_1} \eta_x^\epsilon \eta_z^\epsilon (\alpha x y^{2p-1} \theta^{2q}) \d x \d y \d z.\\[3mm]
\end{array}
\end{equation}
Moreover, Remark \ref{chap6:rem:r1} allows to bound the two integrals on the boundary that yields the following estimate:
\begin{equation}\nonumber
\begin{array}{l}
\| \eta_z^\epsilon \pi \|_{L^2} \leq C_{\pi}.
\end{array}
\end{equation}

Denoting $v^\epsilon:=\eta^\epsilon \pi$, we have $\| v^\epsilon \|_{H^1(\Delta_1)} \leq \tilde{C}_{\pi}$ so we can extract a weakly converging subsequence $v^\epsilon \to v$ in $H^1(\Delta_1)$ and $v=\eta \pi \in H^1(\Delta_1)$. We can check that $\eta$ satisfies the boundary condition of Problem \ref{chap6:pb1}. First, let us check that
\[
\pi(y) A \eta  = 0  \mbox{ in } H^{-1}(\Delta_1) , \quad  \pi(y) \eta_x(\pm L,y,z)   =  0 \mbox{ in } (H^{\frac{1}{2}}((-L,L)\times(-\bar{y}_1,\bar{y}_1)))'.
\]
As $v^\epsilon \in H^2(\Delta_1)$ and $\pi A^\epsilon \eta^\epsilon = 0$, we have
\[
-A^\epsilon v^\epsilon = f(\eta,\eta_y) \quad \mbox{ in strong sense }
\]
with $f(\eta,\eta_y):=  -\frac{1}{2} \lbrace \pi'' \eta^\epsilon + 2 \pi' \eta_y^\epsilon \rbrace + (c_0y+kz+\alpha x) \pi' \eta^\epsilon $.\\[2mm] 
We obtain that $\forall \phi \in H^1(\Delta_1), \phi(x,y,\pm Y) = 0$ and $\phi(x,\pm \bar{y}_1,z)=0$, 
\[
\frac{\epsilon}{2} \int_{\Delta_1} v^\epsilon_z \phi_z + \frac{1}{2} \int_{\Delta_1} v^\epsilon_y \phi_y  +\frac{1}{2} \int_{\Delta_1} v^\epsilon_x \phi_x +
\int_{\Delta_1} \big ( \alpha x v^\epsilon_x  + (\beta x + c_0y + kz) v^\epsilon_y -yv^\epsilon_z \big ) \phi =  \int_{\Delta_1}  f(\eta^\epsilon,\eta_y^\epsilon) \phi.
\]
Now, when $\epsilon$ goes to $0$, we have
\[
 \frac{1}{2} \int_{\Delta_1} v_y \phi_y  +\frac{1}{2} \int_{\Delta_1} v_x \phi_x +
\int_{\Delta_1} \big ( \alpha x v_x  + (\beta x + c_0y + kz) v_y -yv_z\big ) \phi = \int_{\Delta_1}  f(\eta,\eta_y)  \phi.
\]
We deduce we have in $H^{-1}(\Delta_1)$, firstly $-A v = f(\eta,\eta_y)$ which is equivalent to $\pi A \eta = 0$ and secondly that choice of test function implies 
$\pi \eta_x(\pm L,y,z)=0$ in $(H^{\frac{1}{2}}((-L,L)\times(-\bar{y}_1,\bar{y}_1))'$. \\
Then, we check that 
\[
\eta(x,y,Y)-\beta^+(x,y) \in K_{Y,\| \phi \|} ; \quad \eta(x,y,-Y) \in K_{-Y,\| \phi \|}.
\]
We know that $\eta^\epsilon \in H^{1}(\Delta_1)$, its trace is well defined and satisfies
\[
\begin{array}{l}
\gamma(\eta^\epsilon)(x,y,Y)=\chi^{+,\epsilon} + \beta^+ ; \quad \chi^{+,\epsilon} \in K_{Y, \| \phi \|} ; \quad y>0\\[2mm]
\gamma(\eta^\epsilon)(x,y,-Y)=\chi^{-,\epsilon} ; \quad \chi^{-,\epsilon} \in K_{-Y, \| \phi \|} ; \quad y<0\\[2mm]
\end{array}
\]
with
\begin{equation}
\label{chap6:chipmEPSILON}
\| \chi^{\pm,\epsilon} \|_{L^\infty} \leq \| \phi \|_{L^\infty}.
\end{equation}
We also have $\chi^{-,\epsilon},\chi^{+,\epsilon}$ satisfy respectively (\ref{chap6:p11EPSILON}) and (\ref{chap6:p15EPSILON}) 
\begin{equation}
\begin{array}{l}
\int_{-L}^{L} \int_{0}^{\bar{y}_1} \lbrace \frac{1}{2}(\chi_x^{+,\epsilon} \psi_x + \chi_y^{+,\epsilon} \psi_y) + (\alpha x \chi_x^{+,\epsilon} + (\beta x + c_0 y + kY) \chi_y^{+,\epsilon} )\psi \rbrace \d x \d y=0,\\[3mm]
 \forall \psi \in H^1((-L,L)\times(0,\bar{y}_1)) \mbox{ with } \psi(x,0)=\psi(x,\bar{y}_1)=0.\\[3mm]
\end{array}
\label{chap6:p11EPSILON}
\end{equation}
and
\begin{equation}
\begin{array}{l}
\int_{-L}^{L} \int_{-\bar{y}_1}^{0} \lbrace \frac{1}{2}(\chi_x^{-,\epsilon} \psi_x + \chi_y^{-,\epsilon} \psi_y) + (\alpha x \chi_x^{-,\epsilon} + (\beta x + c_0 y - kY) \chi_y^{-,\epsilon} )\psi \rbrace \d x \d y=0,\\[3mm]
 \forall \psi \in H^1((-L,L)\times(-\bar{y}_1,0)) \mbox{ with } \psi(x,0)=\psi(x,-\bar{y}_1)=0.\\[3mm]
\end{array}
\label{chap6:p15EPSILON}
\end{equation}
First, we study convergence of the sequence $\chi^{\pm,\epsilon}$ and we deduce the PDEs satisfied by $\lim_{\epsilon \to 0 } \chi^{\pm,\epsilon}$.\\
In particular (\ref{chap6:chipmEPSILON}) implies  
\[
\begin{array}{c}
\chi^{+,\epsilon} \to \chi^{+} \mbox{ in } L^2((-L,L)\times(0,\bar{y}_1)) \mbox{ weakly, }\\[2mm]
\chi^{-,\epsilon} \to \chi^{-} \mbox{ in } L^2((-L,L)\times(-\bar{y}_1,0)) \mbox{ weakly. }\\[2mm]
\end{array}
\]
And (\ref{chap6:p11EPSILON}) and (\ref{chap6:p15EPSILON}) imply
\[
\begin{array}{c}
\|  \chi^{+,\epsilon} \pi \|_{H^1} \leq C \quad \mbox{ and } \quad  \chi^{+,\epsilon} \pi  \to  \chi^{+}\pi \mbox{ in } H^1 \mbox{ weakly, }\\[2mm]
\|  \chi^{-,\epsilon} \pi \|_{H^1} \leq C \quad \mbox{ and } \quad  \chi^{-,\epsilon} \pi \to  \chi^{-} \pi \mbox{ in } H^1 \mbox{ weakly. }\\[2mm]
\end{array}
\]
Denote $\xi^{\pm, \epsilon}:= \chi^{\pm,\epsilon} y^2$. From (\ref{chap6:p11EPSILON}), we obtain $B_+ \xi^{+,\epsilon} = 0$ in $(-L,L) \times (-\bar{y}_1,\bar{y}_1)$, in $H^{-1}((-L,L) \times (0,\bar{y}_1))$. Since the operator $B_+$ is strictly elliptic then $\xi^{+,\epsilon} \in H^2((-L,L)\times (0, \bar{y}_1))$. We also have $\xi_x^{+,\epsilon}(\pm L, y)=0$ in $(H^{\frac{1}{2}}(0,\bar{y}_1))'$.
As $\xi^{+,\epsilon} \in H^2(\Delta_1)$ and  $\pi B_+ \chi^{+,\epsilon} =0$, we obtain
\[
-B_+ \xi^{+,\epsilon}  = g(\chi^{+,\epsilon},\chi^{+,\epsilon}_y) \quad \mbox{ in a strong sense }
\]
with $g(\chi^{+,\epsilon},\chi^{+,\epsilon}_y):=  -\chi^{+,\epsilon} \lbrace 1 - 2y(\alpha x + c_0y +kY)  \rbrace - 2 y \chi^{+,\epsilon}_y $. We obtain that $\forall \psi \in H^1((-L,L)\times (0,\bar{y}_1)), \psi(x,0) = \psi(x,\bar{y}_1) =  0$, 
\[
 \int_{-L}^{L} \int_{0}^{\bar{y}_1} \frac{1}{2} ( \xi^{+,\epsilon}_x \psi_x + \xi^{+,\epsilon}_y \psi_y)  + \big ( \alpha x  \xi^{+,\epsilon}_x  + (\beta x + c_0y + kY)  \xi^{+,\epsilon}_y -y \xi^{+,\epsilon}_z \big ) \psi =  \int_{-L}^{L} \int_{0}^{\bar{y}_1}  g(\chi^\epsilon,\chi_y^\epsilon) \psi.
\]
Now, when $\epsilon$ goes to $0$, we have
\[
 \int_{-L}^{L} \int_{0}^{\bar{y}_1} \frac{1}{2} ( \xi^{+}_x \psi_x + \xi^{+}_y \psi_y)  + \big ( \alpha x  \xi^{+}_x  + (\beta x + c_0y + kY)  \xi^{+}_y -y \xi^{+}_z \big ) \psi =  \int_{-L}^{L} \int_{0}^{\bar{y}_1}  g(\chi,\chi_y) \psi.
\]
We deduce that in $H^{-1}((-L,L)\times (0, \bar{y}_1))$, we firstly have $-B_+  \xi^{+} = g( \chi^{+},\chi^+_y)$, which is equivalent to $y^2B_+ \chi^+ = 0$ and secondly that choice of test functions implies 
$y^2 \chi^+_x(\pm L,y)=0$ in $(H^{\frac{1}{2}}((0, \bar{y}_1)))'$. 
To summarize, we have
\[
 \xi^{+}   \in H^1((-L,L)\times(0,\bar{y}_1)) 
 \]
 and
 \[
 \left\{
 \begin{array}{rclll}
 -B_+ \xi^{+} & = & g(\chi^{+},\chi_y^{+}) & \mbox{ in } &  (-L,L)\times(0,\bar{y}_1),\\[2mm] 
 \xi^{+}_x(\pm L,y) & = & 0 & \mbox{ in } & (0,\bar{y}_1), \\[2mm]
 \xi^{+}(x,\bar{y}_1) & = & 0 & \mbox{ in } & (-L,L).\\[2mm]
 \end{array}
 \right.
\]
Hence
\[
\pi \chi^{+}   \in H^1((-L,L)\times(0,\bar{y}_1)), \quad \| \chi^{+} \|_{L^\infty} \leq \| \phi \|_{L^\infty}
\]
and
\begin{equation}
 \left\{
 \begin{array}{rclll}
 \pi B_+ \chi^{+}             & = & 0             &  \mbox{ in } &  (-L,L)\times(0,\bar{y}_1), \\[2mm] 
 \pi \chi^{+}_x(\pm L,y) & = & 0  & \mbox{ in } & (0,\bar{y}_1),\\[2mm]
 \chi^{+}(x,\bar{y}_1)    & = & 0     & \mbox{ in } & (-L,L).\\[2mm]
 \end{array}
 \right.
 \label{chap6:xik+}
\end{equation}
Similarly, we have 
\[
\pi \chi^{-}   \in H^1((-L,L)\times(-\bar{y}_1,0)), \quad \| \chi^{-} \|_{L^\infty} \leq \| \phi \|_{L^\infty}
\]
and
\begin{equation}
 \left\{
 \begin{array}{rclll}
 \pi B_+ \chi^{-} & = & 0 & \mbox{ in } &  (-L,L)\times(-\bar{y}_1,0), \\[2mm] 
 \pi \chi^-_x(\pm  L,y)& = & 0 & \mbox{ in } & (-\bar{y}_1,0),\\[2mm]
 \pi \chi^{-}(x,-\bar{y}_1)& = & 0 & \mbox{ in } & (-L,L).\\[2mm]
\end{array}
 \right.
 \label{chap6:xik-}
\end{equation}
First, $\gamma(\pi \eta^\epsilon) \to \pi (\chi^{+} + \beta^+)$ in $H^1((-L,L)\times (0,\bar{y}_1))$ weakly. Secondly, the weak convergence of $\pi \eta^\epsilon \to \pi \eta$ in $H^1(\Delta_1)$ implies the weak convergence of $\gamma(\pi \eta^\epsilon) \to \gamma(\pi \eta)$ in $H^{\frac{1}{2}}( \partial \Delta_1)$. By uniqueness of the limit, we deduce $\gamma(\pi \eta) = \pi (\chi^{+} + \beta^+)$.
Finally, we verify that
\[
\eta(x,\bar{y}_1,z)=\phi(x,z) ; \quad \eta(x,\bar{y}_1,z)=0.
\]
Using Green formula, we obtain
\[
\forall \psi \in H^1(\Delta_1), \quad \int_{\Delta_1} \eta_y^\epsilon \psi + \int_{\Delta_1} \eta^\epsilon \psi_y = \int_{\partial \Delta_1} \phi \psi \vec{n}(y) \d \sigma. 
\] 
Now, we can let $\epsilon$ tend to $0$, we obtain
\[
\forall \psi \in H^1(\Delta_1), \quad \int_{\Delta_1} \eta_y \psi + \int_{\Delta_1} \eta \psi_y = \int_{\partial \Delta_1} \phi \psi \vec{n}(y) \d \sigma. 
\]
\end{proof}
\subsubsection{Approximation (part 3)}
Now, $\phi \in L^\infty(\partial \Delta_1)$, we introduce a sequence of function $ \lbrace \phi^k , k \geq 0 \rbrace \subset H^1(\partial \Delta_1)$ such that $\phi^k \to \phi$ in $L^2(\partial \Delta_1)$. We denote $\eta^k$ the solution of the Problem \ref{chap6:EPSILON_LIMIT} with $\phi^k$ as boundary condition. From the previous section we have $\eta^k \in  L^{\infty}(\Delta_1)$ satisfies $\pi(y)\eta^k \in H^1(\Delta_1)$,
\begin{equation}\nonumber
A \eta^k  = 0  \mbox{ in } \Delta_1, \quad  \eta^k(x,y,Y) - \beta^{k,+}(x,y) \in K_{Y,\| \phi \|}, \quad \eta^k(x,y,-Y) \in K_{-Y,\| \phi \| }
\end{equation}
and
\begin{eqnarray*}
\begin{array}{rclcl}
\eta_x^k(\pm L,y,z) & = & 0 & \mbox{ in } & (y,z) \in (-\bar{y}_1,\bar{y}_1) \times (-Y,Y),\\ 
\eta^k(x,\bar{y}_1,z) & =& \phi^k(x,z) & \mbox{ in } &  (x,z) \in (-L,L) \times (-Y,Y),\\ 
\eta^k(x,-\bar{y}_1,z) & =&  0 & \mbox{ in } &  (x,z) \in (-L,L) \times (-Y,Y). 
\end{array} 
\end{eqnarray*} 
where $\beta^{k,+}(x,y) \in H^1((-L,L) \times (0,\bar{y}_1))$ solves the problem \eqref{chap6:betaplus} with $\phi^k(x,Y)$ as boundary condition.
Moreover  $\| \eta^k \|_{\infty} \leq \| \phi \|_{\infty}$ and $\| \beta^k \|_{\infty} \leq \| \phi \|_{\infty}$. Let us check that the sequence $\beta^k$ has a limit.
\begin{proposition}
\label{chap6:prop:betaplusk}
We have
\begin{align*}
& \beta^{k,+}(x,y) \to \beta^+ \mbox{ in } L^2((-L,L) \times (0,\bar{y}_1))\mbox{ weakly},\\
& \pi \beta^{k,+}(x,y) \to \pi \beta^+ \mbox{ in } H^1((-L,L) \times (0,\bar{y}_1)) \mbox{weakly}
\end{align*}
and the limit $\beta^+$ solves the problem \eqref{chap6:betaplus} with $\phi(x,Y)$ as boundary condition.
\end{proposition}

\begin{proof}
We have $\forall \psi \in H^1((-L,L) \times (0,\bar{y}_1))$ with $\psi(x,0)=\psi(x,\bar{y}_1)=0$,
\begin{equation}\label{chap6:betakconv}
 \frac{1}{2} \iint \big{(}  \beta^{k,+}_x \psi_x + \beta^{k,+}_y \psi_y \big{)} + \iint \big{(} \alpha x \beta^{k,+}_x + (c_0y+kY+\beta x) \beta^{k,+}_y \big{)}\psi  =0. 
\end{equation}
In particular, the choice of $\psi = \pi(y) \beta^{k,+}$ with $\pi(0)=\pi(\bar{y}_1)=0$ gives 
\[
 \frac{1}{2} \iint \big{(}  (\pi \beta^{k,+}_x)^2 + (\pi \beta^{k,+}_y)^2 \big{)} + \iint \pi(y) \pi'(y) \beta^{k,+}_y \beta^{k,+} +   \iint \big{(} \alpha x \beta^{k,+}_x + (c_0y+kY+\beta x ) \beta^{k,+}_y \big{)} \pi^2 \beta^{k,+} =0. 
\]
This implies 
\[
\iint (\pi \beta^{k,+}_x)^2 \leq C_{\pi} \quad  \mbox{ and } \quad \iint (\pi \beta^{k,+}_y)^2 \leq C_{\pi}.
\]
We deduce that we can extract a subsequence such that we have
\[
 \beta^{k,+} \to  \beta^{+} \mbox{ in } L^2 \mbox{ weakly } \quad \mbox{ and } \quad   \pi \beta^{k,+} \to  \pi \beta^{+} \mbox{ in } H^1  \mbox{ weakly}.
\]
Now, denote $\gamma^k:= \pi(y) \beta^{k,+}$. We have $B_+ \gamma^k = - \beta^{k,+} (\frac{\pi''}{2}- (\beta x + c_0 y + k Y) \pi' ) - \beta_y^{k,+} \pi'$.
Also $\forall \psi \in H^1((-L,L) \times (0,\bar{y}_1))$ with $\psi(x,0)=\psi(x,\bar{y}_1)=0$,
\begin{align}\label{chap6:gammakconv2}
 \frac{1}{2} \iint \big{(}  \gamma^k_x \psi_x + \gamma^k_y \psi_y \big{)} & + \iint \big{(} \alpha x \gamma^k_x + (c_0y+kY+\beta x) \gamma^k_y \big{)}\psi\\
  & = \iint \lbrace - \beta^{k,+} (\frac{\pi''}{2}- (\beta x + c_0 y + k Y) \pi' ) - \beta_y^{k,+} \pi' \rbrace \psi. \nonumber
\end{align}
When $k$ goes to $+\infty$ in (\ref{chap6:gammakconv2}), we obtain
\[
B_+ \beta^+ = 0 ; \quad \beta_x^+(\pm L, y) = 0.  
\]
Then, from (\ref{chap6:limy}),
\begin{equation}\label{chap6:limy}
\forall \psi \in H^2\cap H_0^1, \quad \iint \beta^{k,+} \psi_{yy} = -\iint \beta_y^{k,+} \psi_y + \int \phi^k \psi_y \vec{n}(x) \d \sigma,  
\end{equation}
we deduce taking limit when $k$ goes to $+\infty$
\[
\forall \psi \in H^2\cap H_0^1, \quad \iint \beta^{+} \psi_{yy} = -\iint \beta_y^{+} \psi_y + \int \phi \psi_y \vec{n}(x) \d \sigma,  
\]
and we obtain the Dirichlet boundary condition
\[
 \beta^{+}(x,\bar{y}_1) = \phi(x,Y) , \quad \beta^{+}(x,0)= 0.
\] 
\end{proof}
\begin{theorem}\label{chap6:prop:etaklimit}
The Problem \ref{chap6:pb1} has a unique solution.
\end{theorem}

\begin{proof}
Testing $A \eta^k$ with $\eta \theta^2$, we obtain
\begin{align*}
\frac{1}{2} \iint (\eta_{y}^k \theta)^2 + \frac{1}{2} \iint (\eta_{x}^k \theta)^2 & \leq   \frac{1}{2} \iint (\eta^k)^2 (\theta \theta')' + \frac{1}{2} \iint (c_0 y + kz + \beta x) (\eta^k)^2 (\theta^2)' + \frac{1}{2} \iint \alpha \theta^2 (\eta^k)^2\\
&   - \frac{1}{2} \int \alpha x (\eta^k \theta)^2 \vec{n}(x) d \sigma  + \frac{1}{2} \int y \theta^2 (\eta^k)^2 \vec{n}(z) \d \sigma.
\end{align*}
So, we have 
\[
\iint (\eta_x^k \theta)^2 \leq C ; \quad \iint (\eta_y^k \theta)^2 \leq C.
\]
Testing $(A \eta^k)$ with $\eta^k_z y^{2p-1} \theta^{2q}$, we deduce the following lemma
\begin{lemma}\label{chap6:lem:l3}
We have
\begin{equation}
 \begin{array}{rcl}
\int_{\Delta_1} (\eta_z^k \pi )^2  \leq &  & \frac{1}{4} \int_0^L \int_0^{\bar{y}_1} \lbrace (\eta_y^k(x,y,Y))^2 + (\eta_x^k(x,y,Y))^2 \rbrace y^{2p-1} \theta^{2q} d x d y \\[3mm]
& + &  \frac{1}{4} \int_0^L \int_{-\bar{y}_1}^0 \lbrace (\eta_y^k(x,y,-Y))^2 + (\eta_x^k(x,y,-Y))^2 \rbrace | y |^{2p-1} \theta^{2q} d x d y\\[3mm]
& + &  (p-\frac{1}{2}) \int_{\Delta_1} (\eta^k_y \theta) (\eta^k_z \pi) y^{p-2} \theta^{q-1} + q \int_{\Delta_1} (\eta^k_y \theta)  (\eta^k_z \pi) y^{p-1} \theta^{q-2} \theta'\\[3mm]
& + &  \int_{\Delta_1} (c_0 y + kz + \beta x) (\eta^k_y \theta) (\eta^k_z \pi) y^{p-1} \theta^{q-1}\\[3mm]
& + &  \int_{\Delta_1} \alpha x (\eta^k_x \theta) (\eta^k_z \pi) y^{p-1} \theta^{q}.
\end{array}
  \label{chap6:uzboundk}
 \end{equation}
Moreover, $\| \eta_z^k \pi \|_{L^2} \leq C_{\pi}$.
\end{lemma}

Denoting $v^k:=\eta^k \pi$, we have $\| v^k \|_{H^1(\Delta_1)} \leq \tilde{C}_{\pi}$ so we can extract a weakly converging subsequence $v^k \to v$ in $H^1(\Delta_1)$ and $v=\eta \pi \in H^1(\Delta_1)$. Similarly as before, we can check that $\eta$ satisfies the boundary condition of Problem \ref{chap6:pb1} which is summarized in the following lemma. (see proof in Appendix).
\begin{lemma} \label{chap6:lem:l4}
We have $\pi(y) A \eta  = 0  \mbox{ in } H^{-1}(\Delta_1), \quad \pi(y) \eta_x(\pm L,y,z)   =  0 \mbox{ in } (H^{\frac{1}{2}}((-L,L)\times(-\bar{y}_1,\bar{y}_1)))'$,
$\eta(x,y,Y)-\beta^+(x,y) \in K_{Y,\| \phi \|}, \quad \eta(x,y,-Y) \in K_{-Y,\| \phi \|}, \quad \mbox{and} \quad \eta(x,\bar{y}_1,z)=\phi(x,z) ; \quad \eta(x,\bar{y}_1,z)=0.$
\end{lemma}
\end{proof}

\subsubsection{Local regularity in the interior Dirichlet problem}

In this section, we derive local regularity properties of the function $v$ related to the interior problem. We recall that $v:= \pi \eta \in H^1(\Delta_1)$ and we have 
\begin{equation}\label{chap6:regu1}
-\frac{1}{2}v_{xx} - \frac{1}{2}v_{yy} + \rho_0(x,y,z) v_y + \alpha x v_x - y v_z = \eta \rho_1(x,y,z) + \eta_y \rho_2(y) 
\end{equation}
where we denote
\[
\rho_0(x,y,z) := \beta x + c_0 y + k z, \quad \rho_1(x,y,z) := \frac{\pi''}{2}-\rho_0(x,y,z) \pi', \quad \rho_2(y) = \pi'(y).
\]
Let $\bar{y} < \bar{y}_1$,

\begin{notation}
\begin{eqnarray*}
\Delta_1(\delta) & : = &  \lbrace (x,y,z) \in \Delta_1, \quad |y-\bar{y}_1| > \delta, \quad |y+\bar{y}_1| > \delta \rbrace\\[2mm]
\Delta_1(\delta, \gamma) & := &   \lbrace (x,y,z) \in \Delta_1(\delta), \quad |z-Y| > \gamma \rbrace\\[2mm]
H_1(\delta) & := & H^1(-Y,Y;H^1(-L,L;H^1(-\bar{y} + \delta, \bar{y} - \delta ))) 
\end{eqnarray*}
\end{notation}

Recall that $\eta \in H_1(\delta)$ means $\eta,\eta_x,\eta_y,\eta_z,\eta_{xy},\eta_{xz},\eta_{zy},\eta_{xyz} \in L^2(\Delta_1(\delta))$.

\begin{proposition}
We have
\begin{equation}\label{chap6:reg_int1}
\forall \delta, \gamma >0, \quad  \eta \in H_1(\delta)  ; \quad  \eta \in \mathcal{C}^0(\Delta_1(\delta))  ; \quad  \eta \in H^2(\Delta_1(\delta,\gamma)),
\end{equation}
and
\begin{equation}\label{chap6:reg_int2}
\| \eta \|_{H_1(\delta)} \leq M_1 ; \quad \| \eta \|_{H^2(\Delta_1(\delta,\gamma))} \leq M_2.\\[2mm]
\end{equation}
Moreover, the trace of $\eta$ at $y=\bar{y}$, denoted by $h(x,z) := \eta(x,\bar{y},z)$  satisfies 
\begin{equation}\label{chap6:reg_int3}
h \in H^1((-L,L) \times (-Y,Y) ) \cap \mathcal{C}^0 ((-L,L) \times (-Y,Y) ).
\end{equation}
\end{proposition}

\begin{proof}
 
The proof relies on the following estimates (see proof in Appendix). 
We have
\[
 \eta_{xx} \pi^2 \in L^2(\Delta_1) ; \quad \eta_{xy}\pi^2 \in L^2(\Delta_1)  ; \quad  \eta_{yy}\pi^2 \in L^2(\Delta_1), 
\]
\[
 \eta_{xz} y^2 \pi^4 \in L^2(\Delta_1) ; \quad \eta_{yz} y^2 \pi^3 \in L^2(\Delta_1).
\]
For $p$ and $q$ large enough, we have
\[
 \eta_{xyz} y^p \pi^q \in L^2(\Delta_1) ; \quad \eta_{yyz } y^p \pi^q \in L^2(\Delta_1); \quad  \rho(z) \eta_{xxz} y^p \pi^q \in L^2(\Delta_1) ; \quad \rho(z) \eta_{zz } y^p \pi^q \in L^2(\Delta_1).
\]
\end{proof}

\section{The exterior Dirichlet problem}
In this section, we prove existence and uniqueness to the homogeneous exterior Dirichlet problem.
\setcounter{equation}{0} 
\subsection{Some background on the exterior Dirichlet problem in the 1d case}
\begin{notation}
\begin{equation*}
 D_d   :=   (- \infty, -\bar{y}) \times (-Y,Y) , \quad D_u  :=  (\bar{y}, +\infty ) \times (-Y,Y),
\end{equation*}
\begin{equation*}
 D   :=   D_d \cup D_u , \quad D^+  :=  (\bar{y}, +\infty ) \times \lbrace Y \rbrace, \quad D^-  :=  (-\infty, -\bar{y}) \times \lbrace -Y \rbrace
\end{equation*}
and
\[
D^\epsilon := D \cap \{ |y| > \bar{y} + \epsilon\}, \epsilon >0.
\]
\end{notation}
Let us recall the exterior Dirichlet problem from \cite{BenTuri1}. Denote $\bar{\tau}:= \inf \{ t > 0,  |y(t)| =  \bar{y} \}$ and consider $h_{\pm} \in L^{\infty}(-Y,Y)$. It is shown that $\mathbb{E}_{(y,z)}(h_{\pm}(z(\bar{\tau})) )$ solves a Dirichlet problem:
Find $\zeta \in L^\infty(D) \cap \mathcal{C}^0(D^\epsilon), \forall \epsilon >0$ such that
\begin{equation}
\label{chap6:p2}
A \zeta  = 0 \mbox{ in } D, \quad  B_{\pm} \zeta = 0, \mbox{ in } D^{\pm} 
\end{equation}
with
\begin{equation*}
\zeta(\pm \bar{y},z)  = h_{\pm}(z) , \quad  -Y < z <Y. 
\end{equation*}
By solving the ordinary differential equation on the boundary, there is 
\[
\zeta_y(y, \pm Y)=K \exp(c_0 y^2 \pm  2 kYy).
\]
As a bounded solution is sought, $K$ must be $0$. Hence, $\zeta(y, \pm Y) = h_{\pm}( \pm Y)$ and then problem \ref{chap6:p2} was recast in \cite{BenTuri1} by
\begin{equation}
\label{chap6:p2r}
A \zeta  = 0 \mbox{ in } D, \quad  \zeta(y, \pm Y) = h_{\pm}( \pm Y), \mbox{ in } D^{\pm} 
\end{equation}
with
\begin{equation*}
\zeta(\pm \bar{y},z)  = h_{\pm}(z) , \quad  -Y < z <Y. 
\end{equation*}

\subsection{The exterior Dirichlet problem in the 2d case} 
\begin{notation}
\begin{equation*}
 \Delta_d   :=   (-L,L) \times (- \infty, -\bar{y}) \times (-Y,Y) , \quad \Delta_u  :=  (-L,L) \times (\bar{y}, +\infty ) \times (-Y,Y)
\end{equation*}
\begin{equation*}
 \Delta   :=   \Delta_d \cup \Delta_u , \quad \Delta^+  :=  (-L,L) \times (\bar{y}, +\infty ) \times \lbrace Y \rbrace, \quad \Delta^-  :=  (-L,L) \times (-\infty, -\bar{y}) \times \lbrace -Y \rbrace
\end{equation*}
and
\[
\Delta_u^\epsilon := \Delta_u \cap \{ |y| > \bar{y} + \epsilon\}, \epsilon >0.
\]
\end{notation}
Due to the symmetry in the exterior Dirichlet problem, we can consider positive values of $y$ only. Denote $\bar{\tau} := \inf \{ t > 0,  y(t) =  \bar{y} \}$ and consider $h \in L^{\infty}((-L,L) \times (-Y,Y))$, we define $\mathbb{E}_{(x,y,z)}(h(x(\bar{\tau}), z(\bar{\tau})))$ as the solution of the exterior Dirichlet Problem \ref{chap6:pb2}:
\begin{problem}
\label{chap6:pb2}
Find $\zeta \in L^\infty(\Delta_u) \cap \mathcal{C}^0(\Delta_u^\epsilon), \forall \epsilon >0$ such that
\begin{equation}\nonumber
A \zeta  = 0  \mbox{ in }  \Delta_u, \quad  B_+ \zeta  = 0   \mbox{ in }  \Delta^+
\end{equation}
and
\begin{eqnarray*}
\begin{array}{rclcl}
\zeta_x(\pm L,y,z) &  = & 0,         &  \mbox{ in } &  (y,z)  \in (\bar{y},+\infty) \times (-Y,Y),\\ 
\zeta(x,\bar{y},z)    &  = & h(x,z), &  \mbox{ in } &  (x,z)  \in (-L,L) \times (-Y,Y).
\end{array} 
\end{eqnarray*}
\end{problem}

\subsubsection{Boundary condition}
Find $\zeta^+ \in L^\infty(\Delta^+)$ such that
\begin{equation}\label{chap6:bc_e2}
B_+ \zeta^+  = 0  \mbox{ in }  \Delta^+,
\end{equation}
and
\begin{eqnarray*}
\begin{array}{rclcl}
\zeta_x^+(\pm L,y,z) &  = & 0         &  \mbox{ in } &  (y,z)  \in (\bar{y},+\infty) \times (-Y,Y),\\ 
\zeta^+(x,\bar{y},z)    &  = & h(x,z) &  \mbox{ in } &  (x,z)  \in (-L,L) \times (-Y,Y).
\end{array} 
\end{eqnarray*}
Define 
\[
H^1_2(\Delta^+) := \left \{   u: \Delta^+ \to \mathbb{R}, \quad \int_{\Delta^+} \frac{u^2 + u_x^2 +u_y^2}{1+y^2} d xd y < \infty \right \}.
\]
\begin{proposition}
Assume that there exists $H \in H_2^1(\Delta^+)$ such that $H(x,\bar{y}) = h(x,Y)$.\\[2mm] 
Then there exists one and only one solution to the problem (\ref{chap6:bc_e2}) with
\[
\zeta^+ \in H^1_2(\Delta^+), \quad \| \zeta^+ \|_{L^\infty} \leq \| h(.,Y) \|_{L^\infty}. 
\]
\end{proposition}

\begin{proof}
First, we prove uniqueness.
It is sufficient to prove that if we have $B_+ \zeta = 0$, $\zeta_x (\pm L,y) =0$ and $\zeta(x,\bar{y}) = 0$ with $\zeta$ bounded then we obtain $\zeta=0$. Set $u(x,y):=\zeta(x,y) \exp(-\frac{c_0}{2} y^2)$ then 
\begin{eqnarray*}
\frac{1}{2}u_{xx} + \frac{1}{2}u_{yy} - \alpha xu_x - u_y(\beta x + kY) + u c_0 (-\frac{c_0}{2}y^2 - y(\beta x +kY) +\frac{1}{2}) & = & 0, \quad (x,y) \in \Delta^+,\\
u_x(\pm L,y) & = & 0, \quad y > \bar{y},\\
u(x,\bar{y}) & = & 0, \quad x \in (-L,L).
\end{eqnarray*}
We can assume that $-\frac{c_0}{2}y^2 - y (\beta x + kY) +\frac{1}{2} <0, \quad y \geq \bar{y}$ ($\bar{y}$ sufficiently large).
Let us prove that $u=0$. Indeed if $u$ has a positive maximum if cannot be at $y=\infty$. But then this contradicts maximum principle from (\ref{chap6:bc_e2}). Similarily, we cannot  have a negative minimum.\\

Now, we adress existence.
Let 
\[
W^1_2(\Delta^+) := \left \{   u \in H_2^1(\Delta^+), \quad \| u \|_{W^1_2(\Delta^+)} := \| u \|_{L^\infty} + \left ( \int_{\Delta^+} \frac{u_x^2 +u_y^2}{(1+y^2)^2} d xd y \right )^{\frac{1}{2}} < \infty \right  \}.
\]
We define a bilinear form on $H_2^1(\Delta^+) \times W_2^1(\Delta^+)$ by
\begin{eqnarray*}
a(u,v) & := & \frac{1}{2} \int_{\Delta^+} \frac{u_xv_x}{(1+y^2)^2}d xd y + \frac{1}{2} \int_{\Delta^+} \frac{u_yv_y}{(1+y^2)^2} d xd y - 2\int_{\Delta^+} \frac{u_y v y}{(1+y^2)^3} d xd y\\[2mm] 
& & + \alpha \int_{\Delta^+} \frac{x u_x v}{(1+y^2)^2} d xd y  + \int_{\Delta^+} \frac{(\beta x + c_0 y + kY) u_y v}{(1+y^2)^2} d xd y.
\end{eqnarray*}
We next define
\[
a_{\gamma}(u,v):= a(u,v) + \gamma \int_{\Delta^+} \frac{uv}{(1+y^2)^2} d xd y.
\]
We finally define a bilinear form on $H_2^1(\Delta^+) \times H_2^1(\Delta^+)$ by
\begin{eqnarray*}
a_{\gamma, \delta}(u,v) & := & \frac{1}{2} \int_{\Delta^+} \frac{u_x v_x}{(1+y^2)^2} d xd y + \frac{1}{2} \int_{\Delta^+} \frac{u_y v_y}{(1+y^2)^2} d xd y-2 \int_{\Delta^+} \frac{u_y v y }{(1+y^2)^3} d xd y\\[2mm] 
& & + \alpha \int_{\Delta^+} \frac{x u_x v}{(1+y^2)^2} d xd y + \int_{\Delta^+} (\beta x + c_0 \frac{y-\bar{y}}{1+\delta y} + c_0 \bar{y} + kY) \frac{u_y v}{(1+y^2)^2} d xd y\\[2mm]
& &+ \gamma \int_{\Delta^+} \frac{uv}{(1+y^2)^2} d xd y.
\end{eqnarray*}
If $v \in W_2^1(\Delta^+)$,
\[
a_{\gamma, \delta}(u,v) = a_\gamma(u,v) - \int_{\Delta^+} \frac{c_0(y-\bar{y})\delta y}{1+\delta y} \frac{u_y v}{(1+y^2)^2} d xd y.
\]
If $u \in W_2^1(\Delta^+)$, we can compute
\begin{eqnarray*}
a_{\gamma, \delta}(u,u) & = & \frac{1}{2} \int_{\Delta^+} \frac{u_x^2}{(1+y^2)^2} d xd y + \frac{1}{2} \int_{\Delta^+} \frac{u_y^2}{(1+y^2)^2}d xd y - 2 \int_{\Delta^+} \frac{u_yuy}{(1+y^2)^3} d xd y\\[2mm]
& & +\alpha \int_{\Delta^+} \frac{xu_xu}{(1+y^2)^2}d xd y + \int_{\Delta^+} \frac{(\beta x + c_0 \bar{y} + kY)}{(1+y^2)^2} u_y u d xd y\\[2mm]
& & -\frac{c_0}{2} \int_{\Delta^+} \frac{u^2}{(1+y^2)^2} \left [ \frac{1}{1+\delta y} - \frac{\delta(y-\bar{y})}{(1+\delta y)^2} - \frac{4(y-\bar{y})y}{(1+\delta y)(1+y^2)} \right ] d xd y\\[2mm]
& & +\gamma \int_{\Delta^+} \frac{u^2}{(1+y^2)^2} d xd y.  
\end{eqnarray*}
And see that 
\begin{eqnarray*}
a_{\gamma, \delta}(u,u) & \geq & \frac{1}{2} \int_{\Delta^+} \frac{u_x^2}{(1+y^2)^2} d xd y + \frac{1}{2} \int_{\Delta^+} \frac{u_y^2}{(1+y^2)^2}d xd y - 2 \int_{\Delta^+} \frac{u_yuy}{(1+y^2)^3} d xd y\\[2mm]
& & +\alpha \int_{\Delta^+} \frac{xu_xu}{(1+y^2)^2}d xd y + \int_{\Delta^+} \frac{(\beta x + c_0 \bar{y} + kY)}{(1+y^2)^2} u_y u d xd y\\[2mm]
& & +(\gamma-\frac{c_0}{2}) \int_{\Delta^+} \frac{u^2}{(1+y^2)^2} d xd y
\end{eqnarray*}
and the right hand does not depend on $\delta$. Moreover, we can define a constant $\bar{\gamma}$ depending only on the constants $\alpha,\beta,c_0,\bar{y},k,Y$ but not on $\delta$ such that
\begin{equation}
\label{chap6:bc_e4}
a_{\gamma, \delta}(v,v) \geq a_0 \| v \|_{H_2^1(\Delta^+)}^2,\quad a_0 > 0.
\end{equation}
The constant $a_0$ depends only on $\alpha,\beta,c_0,\bar{y},k,Y$. If $f(x,y)$ is bounded, we consider the problem
\begin{equation}
\label{chap6:bc_e5}
a_{\gamma,\delta}(u,v-u) \geq \gamma \int_{\Delta^+} \frac{f(v-u)}{(1+y^2)^2} d xd y,\quad \forall v \in K, \quad u \in K
\end{equation}
where 
\[
K:= \{ v \in H_2^1(\Delta^+), \quad v(x,\bar{y})=h(x,Y) \}
\]
which is not empty from the assumption. Then from the coercivity (\ref{chap6:bc_e4}) and results of the theory of Variational Inequalities (\ref{chap6:bc_e5}) has one and only one solution $u_{\gamma \delta}(f)$ (writing $u$ for $u_{\gamma \delta}(f)$ to simplify notation).
If $w \in H_2^1(\Delta^+)$ satisfies $w(x, \bar{y})=0$ then $u+w \in K$, hence 
\[
a_{\gamma,\delta}(u,w) = \gamma \int_{\Delta^+} \frac{fw}{(1+y^2)^2} d xd y
\]
and thus also
\begin{eqnarray*}
-\frac{1}{2} u_{xx} -\frac{1}{2}u_{yy} + \alpha xu_x + (\beta x + c_0 \frac{y-\bar{y}}{1+\delta y} + c_0 \bar{y} + kY) u_y + \gamma u & = & \gamma f, \quad (x,y) \in \Delta^+, \\[2mm]
u_x(\pm L, y) & = & 0, \quad y \in (\bar{y}, +\infty),\\[2mm]
u(x,\bar{y}) & = & h(x,Y), \quad x \in (-L, +L).
\end{eqnarray*}
Also if 
\[
M_f := \max \{ \| h(.,Y) \|_{\infty}, \| f \|_{L^\infty}\}
\]
then
\[
\| u_{\gamma \delta}(f)\|_{L^\infty} \leq M_f.
\]
Moreover,
\[
a_{\gamma \delta}(u,H-u) \geq \gamma \int_{\Delta^+} \frac{f(H-u)}{(1+y^2)^2} d xd y
\]
hence
\[
a_{\gamma \delta}(u,u) \leq a_{\gamma \delta}(u,H) - \gamma \int_{\Delta^+} \frac{f(H-u)}{(1+y^2)^2} d xd y
\]
and 
\begin{eqnarray*}
a_{\gamma \delta}(u,H) & = & a_\gamma(u,H) - \int_{\Delta^+} \frac{c_0(y-\bar{y})\delta y}{(1+\delta y)} \frac{u_y H}{(1+y^2)^2} d xd y\\[2mm]
& \leq & C \| u \|_{H_2^1(\Delta^+)} \| H \|_{W_2^1(\Delta^+)} + \gamma C M_f \| h(.,Y) \|_{\infty}
\end{eqnarray*}
where $C$ depends only on constants $\alpha,\beta,c_0 \bar{y}, k, Y$. From \eqref{chap6:bc_e4}, we deduce easily that 
\begin{equation}
\label{chap6:bc_e6}
\| u_{\gamma \delta}(f)\|_{H_2^1(\Delta^+)} \leq C_\gamma(f).
\end{equation}
Letting $\delta \to 0$, we obtain
\[
u_{\gamma \delta} \to u_\gamma(f) \quad \mbox{in} \quad H_2^1(\Delta^+) \quad \mbox{weakly and in} \quad L^\infty(\Delta^+) \quad \mbox{weakly-}\star, \quad u_\gamma(f) \in K 
\]
We deduce easily from \eqref{chap6:bc_e5} that $u_\gamma(f)$ satisfies
\begin{equation}
\label{chap6:bc_e7}
a_\gamma(u,v-u) \geq \gamma \int_{\Delta^+} \frac{f(v-u)}{(1+y^2)^2} d xd y 
\end{equation}
\[
\forall v \in W_2^1(\Delta^+), \quad v \in K, \quad u \in K, \quad \| u \|_{\infty} \leq M_f, \quad \| u \|_{H_2^1(\Delta^+)} \leq C_\gamma(f)
\]
where again we write $u$ for $u_\gamma(f)$. The solution of \eqref{chap6:bc_e7} is unique. Indeed, we first claim that 
\[
a(u,w) = \gamma \int_{\Delta^+} \frac{f w}{(1+y^2)^2} d xd y, \quad \forall w \in W_2^1(\Delta^+), \quad w(x, \bar{y}) = 0.
\]
But then if $u^1,u^2$ are two solutions 
\[
a_\gamma(u^1,u^1-u^2) = \gamma \int_{\Delta^+} f \frac{u^1-u^2}{(1+y^2)^2} d xd y , \quad a_\gamma(u^2,u^1-u^2) = \gamma \int_{\Delta^+} f \frac{u^1-u^2}{(1+y^2)^2} d xd y,
\]
hence $a_\gamma(u^1-u^2,u^1-u^2) = 0$. However, from \eqref{chap6:bc_e4} we also have
\[
a_\gamma(v,v) \geq a_0 \| v \|_{H_2^1(\Delta^+)}^2, \quad \forall v \in H_2^1(\Delta^+)
\]
Therefore $u^1 - u^2 = 0$. We next consider a sequence $\zeta^n$ with $\zeta^0 = \| h(.,Y) \|_{L^\infty}$ given by the solution of
\begin{equation}
\label{chap6:bc_e8}
a_\gamma(\zeta^{n+1}, v - \zeta^{n+1}) \geq \gamma \int_{\Delta^+} \frac{\zeta^n(v-\zeta^{n+1})}{(1+y^2)^2} d xd y, 
\end{equation}
where
\[
\zeta^{n+1} \in K, \quad \forall v \in W_2^1(\Delta^+), \quad v \in K,  \quad \| \zeta^{n+1} \|_{L^\infty} \leq \max \left ( \| h(.,Y) \|_{L^\infty} , \| \zeta^n \|_{L^\infty} \right ).
\]
Considering $\zeta^1$, we have 
\[
a_\gamma(\zeta^1, v -\zeta^1) \leq \gamma \int_{\Delta^+} \frac{\zeta^0(v-\zeta^1)}{(1+y^2)^2} d xd y.
\]
Take $v=\zeta^1 - (\zeta^1-\zeta^0)^+$, which is admissible, then
\[
a_\gamma(\zeta^1, (\zeta^1-\zeta^0)^+) \leq \gamma \int_{\Delta^+} \frac{\zeta^0(\zeta^1-\zeta^0)^+}{(1+y^2)^2}d xd y
\]
and 
\[
a_\gamma(\zeta^1-\zeta^0, (\zeta^1-\zeta^0)^+) = a_\gamma(\zeta^1,(\zeta^1-\zeta^0)^+) - \gamma \int_{\Delta^+} \frac{\zeta^0(\zeta^1-H)^+}{(1+y^2)^2} d xd y \leq 0
\]
from the assumptions. Therefore $(\zeta^1-\zeta^0)^+ = 0$ which implies $\zeta^1 \leq \zeta^0$ and also 
\[
\| \zeta^1\|_{L^\infty} \leq \| h(.,Y) \|_{L^\infty}.
\]
Then by induction if $\zeta^n \leq \zeta^{n-1}$, $\| \zeta^n \|_{L^\infty} \leq \| h(.,Y) \|_{L^\infty}$. We obtain  $\zeta^{n+1} \leq \zeta^{n}$, $\| \zeta^{n+1} \|_{L^\infty} \leq \| h(.,Y) \|_{L^\infty}$.
Also
\[
a_\gamma(\zeta^{n+1}, \zeta^1-\zeta^{n+1}) \geq \gamma \int_{\Delta^+} \frac{\zeta^n (\zeta^1-\zeta^{n+1})}{(1+y^2)^2} d xd y
\]
\begin{eqnarray*}
a_{\gamma}(\zeta^{n+1},\zeta^{n+1}) & \leq & a_\gamma(\zeta^{n+1},\zeta^{1}) - \gamma \int_{\Delta^+} \frac{\zeta^n(\zeta^{1}-\zeta^{n+1})}{(1+y^2)^2} d xd y\\[2mm]
& \leq & a(\zeta^{n+1},\zeta^{1}) + \gamma \int_{\Delta^+} \frac{\zeta^1 \zeta^{n+1}}{(1+y^2)^2} d xd y - \gamma \int_{\Delta^+} \frac{\zeta^n(\zeta^{1}-\zeta^{n+1})}{(1+y^2)^2} d xd y
\end{eqnarray*}
Hence
\[
\| \zeta^{n+1} \|_{H_2^1(\Delta^+)} \leq C_\gamma \left ( \| \zeta^1 \|_{H_2^1(\Delta^+)} + \| h(.,Y) \|_{\infty} \right ).
\]
The sequence $\zeta^n \to \zeta$ is monotone decreasing and $\zeta^n \to \zeta$ in $H_2^1(\Delta^+)$ weakly, $\zeta \in K, \quad \| \zeta \|_{L^\infty} \leq \| h(.,Y) \|_{L^\infty}$.\\
Hence in the limit,
\begin{equation}
\label{chap6:bc_e9}
a(\zeta,v-\zeta) \geq 0, \quad \forall v \in W_2^1(\Delta^+), v \in K, \quad \zeta \in K, \quad \| \zeta \|_{L^\infty} \leq \| h(.,Y) \|_{L^\infty}
\end{equation} 
and $\zeta$ is solution of equation \eqref{chap6:bc_e2}.
\end{proof}
\subsubsection{The Cauchy problem}
The exterior Dirichlet Problem \ref{chap6:pb2} is equivalent to 
\begin{equation}\label{chap6:cp_e10}
\begin{array}{rcll}
A \zeta & = & 0, &   (x,y,z) \in (-L,L) \times (\bar{y}, \infty) \times (-Y,Y),\\[2mm]
\zeta(x,\bar{y},z) & = & h(x,z), & (x,z) \in (-L,L) \times (-Y,Y),\\[2mm]
\zeta(x,y,Y) & = & \zeta^+(x,y), & (x,y) \in (-L,L) \times (\bar{y}, \infty),\\[2mm]
\zeta_x(\pm L, y , z) & = & 0, &  (y,z) \in (\bar{y}, \infty) \times (-Y,Y)\\ 
\end{array}
\end{equation}
with
\begin{equation*}
\begin{array}{rcll}
B_+ \zeta^+ & = & 0, & \quad (x,y) \in (-L,L) \times (\bar{y}, \infty),\\[2mm]
\zeta^+(x, \bar{y}) & = & h(x,Y), & \quad x \in (-L,L),\\[2mm]
\zeta_x^+(\pm L,y) & = & 0, & \quad y \in (\bar{y},\infty)
\end{array}
\end{equation*}
and 
\[
\| \zeta^+ \|_{L^\infty} \leq \| h(.,Y) \|_{L^\infty}.
\] 
This is a Cauchy problem, with $z$ taking place of time and $(x,y)$ as the space variable. We write it as
\begin{equation}\label{chap6:cp_e11}
\zeta_z + \frac{1}{y} \left ( \frac{1}{2} \zeta_{yy} + \frac{1}{2} \zeta_{xx} - \alpha x \zeta_x - (\beta x + c_0 y + kz ) \zeta_y \right ) = 0
\end{equation} 
and using the notation 
\[
\mathcal{A}(z)u(x,y) := -\frac{1}{y} \left ( \frac{1}{2} u_{yy} + \frac{1}{2} u_{xx} - \alpha x u_x - (\beta x + c_0 y + kz ) u_y \right ) 
\]
we obtain
\begin{equation}\label{chap6:cp_e12}
\begin{array}{rcll}
-\frac{\partial \zeta}{\partial z} + \mathcal{A}(z) \zeta & = & 0, &   (x,y) \in \Delta^+, z <Y,\\[2mm]
\zeta(x,y,Y) & = & \zeta^+(x,y), & (x,y) \in (-L,L) \times (\bar{y}, \infty),\\[2mm]
\zeta_x(\pm L, y , z) & = & 0, &  (y,z) \in (\bar{y}, \infty) \times (-Y,Y),\\[2mm]
\zeta(x,\bar{y},z) & = & h(x,z), & (x,z) \in (-L,L) \times (-Y,Y).\\ 
\end{array}
\end{equation}
So it is a Cauchy problem with mixed Dirichlet-Neuman boundary conditions. We consider the space $H_2^1(\Delta^+)$, with a new norm
\[
\| u \|_{\star}^2 = \int_{\Delta^+} \frac{1}{y} \frac{u_x^2 + u_y^2}{(1+y^2)^2} \d x \d y + \int_{\Delta^+} \frac{u^2}{(1+y^2)^2} \d x \d y
\]
which is not equivalent to 
\[
\| u \|^2 = \int_{\Delta^+} \frac{u_x^2 + u_y^2}{(1+y^2)^2} \d x \d y + \int_{\Delta^+} \frac{u^2}{(1+y^2)^2} \d x \d y.
\]
The norm in $L_2^2(\Delta^+)$ is defined by
\[
| u |_\star^2:= \int_{\Delta^+} \frac{u^2}{(1+y^2)^2}\d x \d y \leq \| u \|_{\star}^2.
\]
We define on $H_2^1(\Delta^+)$ the bilinear continuous form
\begin{eqnarray*}
\mathcal{A}(z)(u,v) & := & \frac{1}{2} \int_{\Delta^+} \frac{u_xv_x + u_yv_y}{y(1+y^2)^2} \d x \d y - \frac{1}{2} \int_{\Delta^+} \frac{u_y v (1+3y^2)}{y^2 (1+y^2)^3} \d x \d y\\[2mm]
& & + \int_{\Delta^+} \frac{\left ( \alpha x u_x + (\beta x + c_0 y + kz) u_y \right )v}{y (1+y^2)^2} \d x \d y.
\end{eqnarray*}
Moreover, 
\begin{eqnarray*}
\mathcal{A}(z)(u,u) & = & \frac{1}{2}\int_{\Delta^+} \frac{u_x^2 + u_y^2}{y(1+y^2)^2} \d x \d y  - \frac{1}{2} \int_{\Delta^+} \frac{u_y u(1+3y^2)}{y^2 (1+y^2)^3} \d x \d y\\[2mm]
& & + \int_{\Delta^+} \frac{\left ( \alpha x u_x + (\beta x + c_0 \bar{y} + kz) u_y \right )u}{y (1+y^2)^2} \d x \d y\\[2mm]
& & -\frac{1}{2} \int_{\Delta^+} u^2 c_0 \frac{d}{d y}(\frac{y-\bar{y}}{y(1+y^2)^2}) \d x \d y
\end{eqnarray*}
and for $\vert z \vert  <Y$,
\[
\leq a_0 \int_{\Delta^+} \frac{u_x^2 + u_y^2}{y(1+y^2)^2} \d x \d y - b_0 \int_{\Delta^+} \frac{u^2}{(1+y^2)^2} \d x \d y
\]
where the constants $a_0, b_0$ depend only on $\alpha,\beta,c_0,k,\bar{y},Y$. Therefore also,
\[
<\mathcal{A}(z) u, u > \geq a_0 \| u \|_{H_2^1(\Delta^+)}^2 - b \int_{\Delta^+} \frac{u^2}{(1+y^2)^2} \d x \d y.
\]
The problem \eqref{chap6:cp_e12} is equivalent to
\begin{equation}\label{chap6:cp_e13}
-(\frac{\partial \zeta}{\partial z},w) + \mathcal{A}(z)(\zeta,w) = 0, \quad \forall w \in H_2^1(\Delta^+), \quad w(x,\bar{y}) = 0
\end{equation}
and 
\[
\zeta(x,y,Y) = \zeta^+(x,y) \quad, \quad \zeta(x,\bar{y},z) = h(x,z).
\] 
We assume that there exists 
\begin{equation}\label{chap6:cp_e13prime}
H(z) \in H^1(-Y,Y; H_2^1(\Delta^+))\quad  \mbox{with} \quad  H(z)(x,\bar{y}) = h(x,z), \quad \forall x,z.
\end{equation}
Writing $\tilde{\zeta}(x,y,z) = \zeta(x,y,z) - H(z)(x,y)$, we deduce
\begin{equation*}
\begin{array}{rcl}
-(\frac{\partial \tilde{\zeta}}{\partial z}, w) + \mathcal{A}(z) (\tilde{\zeta},w) & = & (\frac{\partial H}{\partial z}(z)(x,y),w) - \mathcal{A}(z)(H(z),w), \quad \forall w \in H_2^1(\Delta^+), \quad w(x,\bar{y}) = 0.\\[2mm]
\tilde{\zeta}(x,y,Y) & = & \zeta^+(x,y) - H(Y)(x,y)\\[2mm]
\tilde{\zeta}(x,\bar{y},z) & = & 0
\end{array}
\end{equation*}
Under this form, we obtain one and only one solution
\[
\tilde{\zeta}(x,y,z) \in L^2 \left (-Y,Y;H_{2,0}^1(\Delta^+) \right ) , \quad \frac{\partial \tilde{\zeta}}{\partial z}(x,y,z) \in L^2 \left (-Y,Y;H_{2,0}^1(\Delta^+)' \right )
\]
where $H_{2,0}^1(\Delta^+)$ denotes the subspace of $H_{2}^1(\Delta^+)$ of function which vanish at $\bar{y}$.
We now prove that 
\begin{equation}\label{chap6:cp_e15}
\| \zeta \|_{L^\infty} \leq \| h(.,Y)\|_{L^\infty}
\end{equation}
We will consider 
\[
\Delta_R^+ := \{ -L < x <L, \quad \bar{y} < y < R \}
\]
for $R$ large. We begin with an approximation of the boundary condition $\zeta^+$ with $\zeta_R^+$ solution of (we delete +)
\begin{equation}\label{chap6:cp_e16}
\begin{array}{rcll}
\frac{1}{2}\zeta_{R,xx} +\frac{1}{2} \zeta_{R,yy} - \alpha x \zeta_{R,x} - (\beta x + c_0 y + kY) \zeta_{R,y} & = & 0, & (x,y) \in (-L,L) \times (\bar{y},Y),\\[2mm]
\zeta_{R,x}(\pm L, y) & = & 0, & y \in (\bar{y}, \infty),\\[2mm]
\zeta_R(x,\bar{y}) & = & h(x, Y), & x \in (-L,L),\\[2mm]
\zeta_R(x,R) & = & 0, & x \in (-L,L).  
\end{array}
\end{equation}
We can assume $h \geq 0$, otherwise we decompose $h = h^+-h^-$. We extend $\zeta_R(x,y)$ by $0$ for $y>R$. The sequence of functions $\zeta_R(x,y)$ is increasing and 
$\| \zeta_R \|_{L^\infty} \leq \| h(.,Y) \|_{L^\infty}$.
Let $\theta(y) = 1$ if  $0<y<\frac{1}{2}$ and $0$ if $y>1$ be a smooth function. We may assume $\bar{y} < \frac{R}{2}$. Let $v \in W_2^1(\Delta^+), v \in K$ and test \eqref{chap6:cp_e16} with $\frac{v \theta(\frac{y}{R})-\zeta_R}{(1+y^2)^2}$ which vanishes at $y=\bar{y}$ and $y=R$. Setting $\theta_R(y) = \theta(\frac{y}{R})$, we obtain,
\begin{align}\label{chap6:cp_e17}
&  \frac{1}{2} \int_{\Delta_R^+} \frac{\zeta_{R,x}(v_x \theta_R - \zeta_{R,x})}{(1+y^2)^2} \d x \d y + \frac{1}{2} \int_{\Delta_R^+} \frac{\zeta_{R,y}(v_y \theta_R - \zeta_{R,y})}{(1+y^2)^2} \d x \d y\\ 
& + \frac{1}{2 R} \int_{\Delta_R^+} \frac{\zeta_{R,y}v \theta_R'}{(1+y^2)^2} \d x \d y-  \int_{\Delta_R^+} \frac{2 \zeta_{R,y}(v \theta_R - \zeta_R)y}{(1+y^2)^3y} \d x \d y \nonumber \\
& +  \int_{\Delta_R^+} \frac{\alpha x \zeta_R (v \theta_R - \zeta_R)}{(1+y^2)^2} \d x \d y + \int_{\Delta_R^+} \frac{(\beta x + c_0 y + kY)}{(1+y^2)^2} \zeta_{R,y} (v \theta_R - \zeta_R) \d x \d y = 0.\nonumber
\end{align}
and thus also $a(\zeta_R, v\theta_R-\zeta_R) = 0$. Recalling that
\[
a(u,u) + \gamma \int_{\Delta^+} \frac{u^2}{(1+y^2)^2} d xd y \geq a_0 \| u \|_{H_2^1(\Delta^+)}^2,
\]
we get
\[
a_0 \| \zeta_R \|_{H_2^1(\Delta^+)}^2 \leq a(\zeta_R, v \theta_R) + C \| h(.,Y) \|_{L^\infty}^2,
\]
from which we deduce easily,
\[
\| \zeta_R \|_{H_2^1(\Delta^+)} \leq C.
\]
We then consider the limit 
\[
\zeta_R \to \zeta \mbox{ monotone increasing } , \quad \zeta_R \to \zeta \quad \mbox{in} \quad H_2^1(\Delta^+) \quad \mbox{weakly}
\]
Hence,
\[
a(\zeta,v) \leq a(\zeta,\zeta)
\]
which implies that $\zeta$ is the solution $\zeta^+$ of \eqref{chap6:bc_e2}. We next consider the approximation of \eqref{chap6:cp_e10} for $\bar{y}<y <R$. We write 
\begin{equation}\label{chap6:cp_e18}
\begin{array}{rcll}
A \zeta_R                     & = &  0,                        & (x,y) \in \Delta^R \times (-Y,Y),\\[2mm]
\zeta_R(x,y,Y)             & = &  \zeta_R^+(x,y), & (x,y) \in \Delta^R,\\[2mm]
\zeta_R(x,\bar{y},z)    & = &  h(x,z),                 & (x,z) \in (-L,L) \times (-Y,Y),\\[2mm]
\zeta_R(x,R,z) 	           & = &  0,                        & (x,z) \in (-L,L) \times (-Y,Y),\\[2mm]
\zeta_{R,x}(\pm L,y,z) & = &  0,                        & (y,z) \in (\bar{y},R) \times (-Y,Y).
\end{array}
\end{equation}
We write \eqref{chap6:cp_e18} as a Cauchy problem
\begin{equation}\label{chap6:cp_e19}
\begin{array}{rcll}
\frac{\partial \zeta_R}{ \partial z} + \mathcal{A}(z) \zeta_R & = & 0, & \mbox{in} \quad \Delta_R,\\[2mm]
                                                                     \zeta_R(x,y,R) & = & \zeta_R^+(x,y), & (x,y) \in \Delta_R^+,\\[2mm]
                                                                     \zeta_{R,x}(\pm L, y ,z) & = & 0,  & (y,z) \in (\bar{y},R) \times (-Y,Y),\\[2mm]
                                                                     \zeta_R(x,\bar{y},z) & = & h(x,z), & (x,z) \in (-L,L) \times (-Y,Y),\\[2mm]
                                                                     \zeta_R(x,R,z) & = & 0, & (x,z) \in (-L,L) \times (-Y,Y)
\end{array} 
\end{equation}
and extend $\zeta_R$ by $0$ for $y>R$. If $w \in H_2^1(\Delta^+), w(x,\bar{y}) = 0$, we can write 
\begin{equation}\label{chap6:cp_e20}
\begin{array}{rcll}
-(\frac{\partial \zeta_R}{\partial z}, w \theta_R) + \mathcal{A}(z)(\zeta_R,\theta_R w) & = & 0,\\[2mm]
\zeta_R(x,y,Y)          & = & \zeta_R^+(x,y), & (x,y) \in \Delta_R^+,\\[2mm]
\zeta_R(x,\bar{y},z) & = & h(x,z), &  (x,z) \in (-L,L) \times (-Y,Y).\\[2mm]
\end{array}
\end{equation}
However, from \eqref{chap6:cp_e19}, we deduce
\begin{equation}\label{chap6:cp_e21}
\| \zeta_R\|_{L^\infty} \leq \| h(.,Y) \|_{L^\infty}.
\end{equation}
We can pass to the limit in \eqref{chap6:cp_e21} and check that $\zeta_R \to \zeta$ solution of \eqref{chap6:cp_e13}. This proves \eqref{chap6:cp_e15}.
\subsection{Local regularity in the exterior Dirichlet problem} In this section, we derive local regularity properties of the solution $\zeta$ to the exterior Dirichlet problem.\\
Recall that $\bar{y} < \bar{y}_1$,
\begin{notation}
\[
\Delta_d(\delta)  : =   \lbrace (x,y,z) \in \Delta_d, \quad y < -\bar{y} - \delta \rbrace, \quad \Delta_u(\delta)  : =   \lbrace (x,y,z) \in \Delta_u, \quad y >  \bar{y} + \delta \rbrace, 
\]
\[
\Delta^-(\delta) : = \Delta_d(\delta) \cap \Delta^-, \quad \Delta^+(\delta) : = \Delta_u(\delta) \cap \Delta^+,
\]
\[
\Delta_d(\delta, \gamma) :=  \lbrace (x,y,z) \in \Delta_d(\delta), \quad -Y + \gamma < z \rbrace, \quad \Delta_u(\delta, \gamma) :=  \lbrace (x,y,z) \in \Delta_u(\delta), \quad z < Y - \gamma \rbrace,
\]
\begin{align*}
& H_d(\delta) := H^1(-Y,Y;H^1(-L,L;H^1(-\bar{y}_1 - \delta, -\bar{y}_1 + \delta ))),\\
&  \quad H_u(\delta) := H^1(-Y,Y;H^1(-L,L;H^1(\bar{y}_1 - \delta, \bar{y}_1 + \delta )))
\end{align*}
and
\[
\Delta(\delta)   : =    \Delta_d(\delta) \cup \Delta_u(\delta), \quad \Delta(\delta, \gamma) := \Delta_d(\delta, \gamma) \cup \Delta_u(\delta, \gamma), \quad  H(\delta) :=  H_d(\delta) \cap H_u(\delta).
\]
\end{notation}
Again, thanks to the symmetry in the exterior Dirichlet problem, we consider only positive values of $y$.
\begin{proposition}\label{chap6:prop:reg_ext1}[Local regularity of the exterior Dirichlet problem]
We have
\begin{equation}\label{chap6:reg_ext1}
\forall \delta, \gamma >0, \quad  \zeta \in H_u(\delta)  ; \quad  \zeta \in \mathcal{C}^0(\Delta_u(\delta))  ; \quad  \zeta \in H^2(\Delta_u(\delta,\gamma)),
\end{equation}
and
\begin{equation}\label{chap6:reg_ext2}
\| \zeta \|_{H_u(\delta)} \leq M_{1,\| \phi \|} ; \quad \| \zeta \|_{H^2(\Delta_u(\delta,\gamma))} \leq M_{2,\| \phi\|}.\\[2mm]
\end{equation}
Moreover, the trace of $\zeta$ at $y = \bar{y}_1$, denoted $g(x,z) := \zeta(x,\bar{y},z)$  satisfies 
\begin{equation}\nonumber
g \in H^1((-L,L) \times (-Y,Y) ) \cap \mathcal{C}^0 ((-L,L) \times (-Y,Y) ).
\end{equation}
\end{proposition}

\begin{proposition}\label{chap6:prop:reg_ext2}[Local regularity of the exterior Dirichlet problem on the boundary $z=Y$]
We have
\begin{equation}\label{chap6:reg_ext3}
\forall \delta >0, \quad  \zeta \in H^2(\Delta^+(\delta)) ; \quad  \| \zeta \|_{H^2(\Delta^+(\delta))} \leq M_{3, \| \phi \|}.
\end{equation}
\end{proposition}
Similar results hold for negative values of $y$. Proofs of Proposition \ref{chap6:prop:reg_ext1} and Proposition \ref{chap6:prop:reg_ext2} rely on similar estimates related to the local regularity of the interior Dirichlet problem.
\section{The ergodic operator $\mathbf{P}$}
\setcounter{equation}{0}
\begin{notation}
\[
\bar{\Gamma}_1^{\pm}:= (-L,L) \times \lbrace \pm \bar{y}_1 \rbrace \times (-Y,Y) ; \quad  \bar{\Gamma}^{\pm} := (-L,L) \times \lbrace \pm \bar{y} \rbrace \times (-Y,Y) 
\]
and
\[
\bar{\Gamma}_1 := \bar{\Gamma}_1^- \cup \bar{\Gamma}_1^+ ; \quad  \bar{\Gamma} := \bar{\Gamma}^- \cup \bar{\Gamma}^+
\]
\end{notation}
\subsection{Construction of the operator $\mathbf{P}$}
Consider $\phi : = (\phi_-, \phi_+) \in L^\infty(\bar{\Gamma}_1)$. Following the same procedure of the 1d case, we first solve the interior Dirichlet problem for $\eta$ with the boundary condition 
\[
\left\{
\begin{array}{rcll}
\eta & =  &  \phi_+  & \mbox{ in } \bar{\Gamma}_1^+,\\[2mm]
\eta & = & \phi_-     & \mbox{ in } \bar{\Gamma}_1^-.\\[2mm]
\end{array}
\right.
\]
Then
\[
\| \eta \|_{L^\infty} \leq \max(\| \phi_+ \|,\| \phi_- \|). 
\]
Then, we solve the exterior Dirichlet problem for $\zeta$ with the boundary condition
\[
\left\{
\begin{array}{rcll}
\zeta & =  &  \eta     & \mbox{ in } \bar{\Gamma}^+,\\[2mm]
\zeta & =  &  \eta     & \mbox{ in } \bar{\Gamma}^-.\\[2mm]
\end{array}
\right.
\]
Then
\[
\|\zeta \|_{L^\infty} \leq \| \eta \|_{L^\infty} \leq \| \phi \|_{L^\infty}.
\]
For $\bar{p}_1 \in \bar{\Gamma}_1$, let us define 
\[
\mathbf{P}\phi(\bar{p}_1):= \zeta(\bar{p}_1). 
\]
\subsection{Probabilistic interpretation of the operator $\mathbf{P}$}
Let us recall from Khasminskii \cite{Khasminskii} the probabilistic interpretation of the operator $\mathbf{P}$. For $\bar{p} \in \bar{\Gamma}$ and for $\bar{p}_1 \in \bar{\Gamma}_1$, we denote by $\bar{\gamma}_1(\bar{p}; . )$ the distribution of $z(\bar{\tau}_1)$ starting from $\bar{p}$ and by $\bar{\gamma}(\bar{p}_1; . )$ the distribution of $z(\bar{\tau})$ starting from $\bar{p}_1$.  We showed
\[
\eta(\bar{p}) = \int_{\bar{\Gamma}_1} \phi(u) \bar{\gamma}_1(\bar{p}; du) \quad \mbox{ and } \quad \zeta(\bar{p}_1) = \int_{\bar{\Gamma}} \eta(v) \bar{\gamma}(\bar{p}_1; dv).
\]
Also, using Fubini theorem we can write
\[
\mathbf{P}\phi(\bar{p}_1) = \int_{\bar{\Gamma}_1} \phi(u)  \bar{\gamma} \bar{\gamma}_1(\bar{p}_1,du), 
\]
where 
\[
\bar{\gamma} \bar{\gamma}_1(\bar{p}_1,du) := \int_{\bar{\Gamma}} \bar{\gamma}(\bar{p}_1 ; dv) \bar{\gamma}_1(v ; du).
\]
By construction, the operator $\mathbf{P}$ is the transition probability associated to the Markov chain 
\[
\lbrace (x(\bar{\tau}_{1,k}) , y(\bar{\tau}_{1,k}) , z(\bar{\tau}_{1,k})) \rbrace_{k \geq 0}
\]
where $\bar{\tau}_{1,0} = 0$ and 
\[
 \bar{\tau}_{k+1} := \inf \lbrace t > \bar{\tau}_{1,k}, \quad  | y(t) | = \bar{y} \rbrace ; \quad \bar{\tau}_{1,k+1} := \inf \lbrace t > \bar{\tau}_{k+1}, \quad | y(t) | = \bar{y}_1 \rbrace.
\]
Note that $\bar{\tau}_{k+1} = \bar{\tau} \circ \theta_{ \bar{\tau}_{1,k}}$ and $\bar{\tau}_{1,k+1} = \bar{\tau}_1 \circ \theta_{ \bar{\tau}_{k+1}}$ where $\theta_t$ is the shift operator.
\subsection{Ergodic property for $\mathbf{P}$}
\begin{theorem}
The operator $\mathbf{P}$ is ergodic.
\end{theorem}
\begin{proof}
A Borel subset of $\bar{\Gamma}_1$ can be written as $B:= B_- \times  \lbrace y_1 \rbrace \cup B_+ \times  \lbrace -y_1 \rbrace$ with $B_+, B_-$ are Borel subsets of $(-L,L) \times (-Y,Y)$.
We have
\[
B_+= \lbrace (x,z) : (x,\bar{y}_1,z) \in B \rbrace \mbox{ and } B_-=\lbrace (x,z) : (x,-\bar{y}_1,z) \in B \rbrace
\]
also
\[
\mathbf{1}_{B_+}(x,z)=\mathbf{1}_{B}(x,\bar{y}_1,z) ; \quad \mathbf{1}_{B_-}(x,z)=\mathbf{1}_{B}(x,-\bar{y}_1,z).
\]
Consider $\phi_+=\mathbf{1}_{B_+}$ and $\phi_-=\mathbf{1}_{B_-}$ in the interior Dirichlet problem, then
 \[
 \mathbf{P}(\mathbf{1}_{B})(x,y,z)= \left\{ \begin{array}{ll} \zeta(x,\bar{y}_1,z) & \mbox{ if } y=\bar{y}_1\\[2mm] \zeta(x,-\bar{y}_1,z) & \mbox{ if } y=-\bar{y}_1.\\[2mm] \end{array} \right.
 \]
 Let $p,\tilde{p} \in \bar{\Gamma}_1$,  define
 \[
 \lambda_{p,\tilde{p}}(B) : = \mathbf{P}(\mathbf{1}_{B})(p)- \mathbf{P}(\mathbf{1}_{B})(\tilde{p}).
 \] 
 We will prove the ergodic property of the operator $\mathbf{P}$ in the four following steps.
 \begin{enumerate}
 \item {\bf Doob's criterion from \cite{Doob}.}\\
 The operateur $\mathbf{P}$ is ergodic if we prove the following Doob's criterion
 \begin{equation}\nonumber
 \sup \lambda_{p,\tilde{p}}(B) < 1, \quad \forall p,\tilde{p} \in \bar{\Gamma}_1 \mbox{ and } \forall B.
 \end{equation}

\item {\bf Negation of Doob's criterion.}\\
Suppose Doob's criterion is not verified. Then, there exists two sequences $p_k:=(x_k,y_k,z_k)$, $\tilde{p}_k:=(\tilde{x}_k,\tilde{y}_k,\tilde{z}_k)$ in $\bar{\Gamma}_1$ and a sequence of Borel subset $B_k$ such that 
\[
\lambda_{p_k,\tilde{p}_k}(B_k) \to 1.
\] 
Denote $\eta^k$ and $\zeta^k$ the solution of the interior and exterior Dirichlet problems with $\phi_+=\mathbf{1}_{B_{+k}}$ and $\phi_-=\mathbf{1}_{B_{-k}}$,  where $B_{+k}$ and $B_{-k}$ are deduced from $B_k$ as previously. We have
\[
\lambda_{p_k,\tilde{p}_k}(B_k) =  \zeta^k(x_k,y_k,z_k) - \zeta^k(\tilde{x}_k,\tilde{y}_k,\tilde{z}_k).
\]
Now, extracting a subsequence of $p_k$ and $\tilde{p}_k$, we deduce
\[
p_k \to p^\star \quad \mbox{and} \quad \tilde{p}_k \to \tilde{p}^\star \mbox{ in } \bar{\Gamma}_1.
\]
Also,
\[
p^\star = (x^\star,\bar{y}_1,z^\star) \quad \mbox{ or } \quad (x^\star,-\bar{y}_1,z^\star) 
\]
and
\[
\tilde{p}^\star = (\tilde{x}^\star,\bar{y}_1,\tilde{z}^\star) \quad \mbox{ or } \quad (\tilde{x}^\star,-\bar{y}_1,\tilde{z}^\star). 
\]
From (\ref{chap6:reg_ext1}), we know that
\begin{equation}\nonumber
\| \zeta^k \|_{H(\delta)} \leq M_{1,1} \mbox{ and }  \| \zeta^k \|_{H^2(\Delta(\delta,\gamma))} \leq M_{2,1}
\end{equation}
then
\begin{equation}\label{chap6:regularite111}
\begin{array}{ccl}
& & \zeta^k \in H(\delta) \subset \mathcal{C}^0(\Delta(\delta)),\\[2mm]
& & \zeta^k \to \zeta^\star \quad \mbox{ in } \quad H(\delta) \quad \mbox{weakly},\\[2mm]
& & \zeta^\star \in \mathcal{C}^0(\Delta(\delta)),\\[2mm]
\mbox{ and } & & \\[2mm] 
 & & \zeta^k \to \zeta^\star \quad \mbox{ in } \quad \mathcal{C}^0(\Delta(\delta,\gamma)).\\[2mm]
\end{array}
\end{equation}
From \eqref{chap6:reg_ext2}, we know that
\begin{equation}\nonumber
\| \zeta^k \|_{H^2(\Delta^\pm(\delta))} \leq M_{3,1} ,
\end{equation}
so, we obtain
\begin{equation}\label{chap6:regularite222}
\zeta^k \to \zeta^\star \quad \mbox{ in } \quad \mathcal{C}^0(\Delta^\pm(\delta)).\\[2mm]
\end{equation}
\begin{proposition} Under the hypothesis that $P$ is not ergodic, we have  $\zeta^\star(p^\star)=1.$
\end{proposition}
\begin{proof}
We must have $\lim_{k \to \infty} \zeta^k(p_k)=1$.
\begin{enumerate}
\item
If $|z_\star| < Y$ or $z_\star y_\star = - Y \bar{y}_1$ then (\ref{chap6:regularite111}) implies $\lim_{k \to \infty} \zeta^k(p_k) = \zeta^\star(p_\star)$.
\item
If $z_\star y_\star = Y \bar{y}_1$ then (\ref{chap6:regularite222}) implies $\lim_{k \to \infty} \zeta^k(x_k, \pm \bar{y_1}, \pm Y) = \zeta^\star(p_\star)$. 
Indeed, as for all $k \geq 0, \zeta^k$ is continuous, there exists $\sigma_k \geq k$ such that $|\zeta^k(p_{\sigma_k}) - \zeta^k(x_{\sigma_k}, \pm \bar{y}_1,\pm Y) | \leq 2^{-k}$. Also, $\lim_{k \to \infty} \zeta^k(p_{\sigma_k}) = \lim_{k \to \infty} \zeta^k(x_{\sigma_k}, \pm \bar{y}_1, \pm Y)=\zeta^\star(p^\star)$. 
\end{enumerate}
\end{proof}
\item {\bf Contradiction.}\\
Suppose $p^\star=(x^\star,\bar{y}_1,z^\star)$ and set
\[
 \Xi_1:= \bar{\Gamma}_1 \cap \lbrace x = \pm L \rbrace ; \quad  \Xi_2:= \bar{\Gamma}_1 \cap \lbrace z = Y \rbrace.
\]
Maximum principle for parabolic operator applied to $\zeta^*$ implies $p^\star \in \Xi_1 \cup \Xi_2$. Then $p^\star$ cannot be in $\Xi_1$ because of the Neuman condition and $p^\star$ cannot be in $\Xi_2$ because of the boundary condition $z=Y$. A similar argument yields a contradiction with $p^\star=(x^\star,-\bar{y}_1,z^\star)$.
\item {\bf Conclusion.}\\
From ergodic theory, there exists a unique probability mesure $\gamma_\star = (\gamma_\star^-,\gamma_\star^+)$ on $\bar{\Gamma}_1$ and $K,\rho >0$ such that 
\begin{equation}\label{chap6:ERGODIQUE}
\forall n \geq 0, \quad \left | \mathbf{P}^n \phi - \iint_{\bar{\Gamma}_1^-} \phi(r,s) \d \gamma_\star^-(r,s) - \iint_{\bar{\Gamma}_1^+} \phi(r,s) \d \gamma_\star^+(r,s) \right | \leq K \| \phi \| \exp(-\rho n).
\end{equation}
\end{enumerate}
\end{proof}

\section{The nonhomogeneous interior Dirichlet problem} 

\setcounter{equation}{0}
Consider $f$ a bounded function on $\Delta_1$, we want to define $\mathbb{E}_{(x,y,z)} \big{(} \int_0^{\bar{\tau}_1} f(x(s),y(s),z(s))\d s \big{)}$ as the solution of the non-homogeneous interior Dirichlet Problem \ref{chap6:pb3}:
\begin{problem}\label{chap6:pb3}
Find $\chi \in L^{\infty}(\Delta_1) \cap \mathcal{C}^0(\Delta_1^\epsilon), \forall \epsilon >0$ such that
\begin{equation}\nonumber
\begin{array}{rcccl}
A \chi + f                              & = &  0    & \mbox{  in }   &   \Delta_1,\\[2mm]
B_{\pm}\chi + f                & = &  0     & \mbox{  in }  &   \Delta_1^{\pm},\\[2mm]
\chi(x, \pm \bar{y}_1,z)     & = &  0     & \mbox{  in }   &    (-L,L) \times (-Y,Y),\\[2mm]
\chi_x(\pm L,y,z)                & = &  0     & \mbox{ in }   &   (-\bar{y}_1,\bar{y}_1) \times (-Y,Y).
\end{array}
\end{equation}
\end{problem}
Consider
\[
\Phi(x,y,z) := \exp \lambda (c_0 k z^2 + c_0 y^2) ; \quad \lambda \geq 1. 
\]
\begin{theorem}
The Problem \ref{chap6:pb3} has a unique solution. Moreover, we have
\[
 \| \chi \|_{L^\infty} \leq  \exp \lambda (c_0 k Y^2 + c_0 \bar{y}_1^2),
\]
where $\lambda$ depends on $f$.
\end{theorem}
\begin{proof}
The uniqueness of a solution of Problem $\ref{chap6:pb3}$ is argued as the uniqueness for the homogeneous Dirichlet interior problem. The existence can be proven by the regularization technique used previously. Now, we give a  $L^\infty$ bound for $\chi$. We have
\[
\begin{array}{l}
\phi_x=0 ; \quad \phi_{xx}=0\\[2mm] 
\phi_y = 2y c_0 \lambda \phi ; \quad \phi_{yy} = 2 c_0 \lambda \phi + (2 c_0 \lambda y)^2 \phi\\[2mm]
\phi_z = 2z c_0 k \lambda \phi\\[2mm] 
\end{array}
\]
and
\begin{eqnarray*}
y \phi_z - \alpha x \phi_x - \phi_y (\beta x+c_0 y +kz) + \frac{1}{2} \phi_{yy} + \frac{1}{2} \phi_{xx} 
& = & - 2\beta c_0  \lambda x y \phi - 2 c_0^2  \lambda y^2 \phi + c_0 \lambda \phi  + 2 (c_0\lambda y)^2\\[2mm]
& = &  c_0 \lambda \phi - c_0\lambda 2\beta x y   \phi + 2  (c_0y)^2 \lambda(\lambda -1) \phi \\[2mm]
& \geq &  c_0 \lambda (1 - 2\beta L \bar{y}_1 ) \phi + 2   (c_0 y)^2 \lambda (\lambda -1) \phi \\[2mm]
& \geq & \| f \|  
\end{eqnarray*}
where $\bar{y}_1 \mbox{ can be chosen as } 2 \bar{y}_1 < \frac{1}{2 \beta L},  \mbox{ and } \lambda \geq \max(1, \frac{\| f \|}{c_0(1-2 \beta L \bar{y}_1)})$.\\
Using that $\phi(x,y,\pm Y) = \exp \lambda (c_0 k Y^2)\exp \lambda (c_0^2 y^2)$ we obtain for $z=Y, 0< y < \bar{y}_1$.
\begin{eqnarray*}
 -\alpha x \phi_x - \phi_y (\beta x+c_0 y +kY) + \frac{1}{2} \phi_{yy} + \frac{1}{2} \phi_{xx} 
& = & -\phi_y(x,y,Y)(c_0y + kY + \beta x) + \frac{1}{2}\phi_{yy}(x,y,Y)\\[2mm]
& = & \lambda c_0 \phi(x,y,Y)\big{(}1 - 2 (k Y + \beta x)y + 2 c_0(\lambda -1)y^2 \big{)}\\[2mm]
& \geq & \lambda c_0 \phi(x,y,Y)\big{(}1 - 2 (k Y + \beta L)y + 2 c_0(\lambda -1)y^2 \big{)}\\[2mm]
& \geq & \lambda c_0 \big{(} 1- \frac{(kY+\beta L)^2}{2 c_0 (\lambda -1)} \big{)}\\[2mm]
& \geq & \| f \|,  \\[2mm]
\end{eqnarray*}
if $\lambda \geq \max(2 \frac{\| f \|}{c_0},1+\frac{(kY+\beta L)^2}{c_0})$.
Similar estimates hold for $z=-Y, -\bar{y}_1 < y < 0$.
Consider $u:=\phi-\chi$, then $u$ satisfies the following inequalities
\begin{equation}
\label{chap6:uphi1}
Au  \geq  0  \mbox{ in }  \Delta_1 ; \quad B_{\pm} u  \geq  0  \mbox{ in }  \Delta_{1}^{\pm} \\[2mm]
\end{equation}
and the following boundary conditions
\begin{equation}
\label{chap6:uphi2}
\begin{array}{rclll}
u(x, \pm \bar{y}_1,z) & = & \exp(\lambda \big{(} c_0 k z^2 + c_0 \bar{y}_1^2 \big{)}) & \mbox{ in } (-L,L) \times (-Y,Y),\\[2mm]
u_x(\pm L ,y,z)          & = & 0 & \mbox{ in } ( -\bar{y}_1, \bar{y}_1) \times (-Y,Y).\\[2mm]
\end{array}
\end{equation}
The maximum of $u$ cannot be attained in $\Delta_1$ or on the boundaries $z= \pm Y,x= \pm L$. It can only be attained for $y=\bar{y}_1$ or $y=\bar{y}_1$.
Hence,
\[
u(x,y,z) \leq \exp(\lambda(c_0kY^2 + c_0\bar{y}_1^2))
\]
which implies
\[
\chi(x,y,z) \geq -\exp(\lambda(c_0kY^2 + c_0\bar{y}_1^2)).
\]
Now, consider $v=-(\phi+\chi)$, then $v$ satisfies inequalities (\ref{chap6:vphi1})
\begin{equation}
\label{chap6:vphi1}
Av  \leq  0  \mbox{ in }  \Delta_1 ; \quad B_{\pm} v  \leq  0  \mbox{ in }  \Delta_{1}^{\pm} \\[2mm]
\end{equation}
and the boundary condition (\ref{chap6:vphi2})
\begin{equation}
\label{chap6:vphi2}
\begin{array}{rclll}
v(x, \pm \bar{y}_1,z) & = & -\exp(\lambda \big{(} c_0 k z^2 + c_0 \bar{y}_1^2 \big{)}) & \mbox{ in } (-L,L) \times (-Y,Y),\\[2mm]
v_x(\pm L ,y,z)          & = & 0 & \mbox{ in } ( -\bar{y}_1, \bar{y}_1) \times (-Y,Y).\\[2mm]
\end{array}
\end{equation}
Again, the minimum of $v$ cannot be attained in $\Delta_1$ or on the boundaries $z=\pm Y,x=\pm L$. Hence
\[
v(x,y,z) \geq -\exp(\lambda(c_0kY^2 + c_0\bar{y}_1^2))
\]
which implies
\[
\chi(x,y,z) \leq \exp(\lambda(c_0kY^2 + c_0\bar{y}_1^2)).
\]
Now, let us derive some estimates on derivatives. Let $M:=\exp(\lambda(c_0kY^2 + c_0\bar{y_1}^2))$. Using the test function $m_0(x,y,z)=\exp(-c_0(y^2+kz^2))\exp(-\alpha x^2)$ we obtain now a priori estimates on the partial derivatives of $\chi$.
We test (\ref{chap6:pb3}) with $m_0 \chi$
\begin{eqnarray*}
 \int_{\Delta_1} \chi_y^2 m_0 +  \int_{\Delta_1} \chi_x^2 m_0 & = &  \int_{\partial{\Delta}_1, z=Y} y \chi^2 m_0 -  \int_{\partial {\Delta}_1, z=-Y} y \chi^2 m_0  + 2 \int_{\Delta_1} f \chi m_0 -  \int_{\Delta_1} c_0 \beta x \chi^2 m_0\\[2mm]
& \leq &  \int_{\partial {\Delta}_1, z=Y,y>0} y \chi^2 m_0 - \int_{\partial {\Delta}_1, z=-Y,y<0} y \chi^2 m_0  + 2 \int_{\Delta_1} f \chi\\[2mm]
& \leq & C(M).
\end{eqnarray*}
Hence,
\[
\int_{\Delta_1} \chi_y^2 m_0 \leq C(M) \quad; \quad  \int_{\Delta_1} \chi_x^2 m_0 \leq C(M).
\]
We have then, 
\[
\int_{\Delta_1} (y^2 \chi_z)^2 m_0 \leq C(M) \quad \mbox{ and } \quad \int_{\Delta_{1}} (y^2 \chi_{yy})^2 m_0 \leq C(M).
\]
The function $\chi$ is smooth outside a neighborhood of $y=0$.
\end{proof}
\section{The nonhomogeneous exterior Dirichlet problem}
Consider the function
\[
\Psi(y) := \log(y) + K , \mbox{ $K$ such that } \log(\bar{y}_1) + K >0
\]
and the space of functions
\[
X_\Psi^\infty := \lbrace u \psi,\quad  u \in L^{\infty}(\Delta_u)  \rbrace.
\]
We want to define $\mathbb{E}_{(x,y,z)} \big{(} \int_0^{\bar{\tau}} f(x(s),y(s),z(s)) ds \big{)}$ as the solution of the nonhomogeneous exterior Dirichlet Problem \ref{chap6:pb4}:
\begin{problem}\label{chap6:pb4}
Find $\xi \in X_\Psi^\infty \cap \mathcal{C}^0(\Delta_u^\epsilon), \forall \epsilon>0$ such as
\begin{equation}\nonumber
\begin{array}{rclll}
A \xi + f               & = &  0  & \mbox{ in } & \Delta_u,\\[2mm]
B_{+}\xi + f      & = & 0   & \mbox{ in } & \Delta^+,\\[2mm]
\xi(x,\bar{y},z)    & = & 0  & \mbox{ in } & (-L,L) \times (-Y,Y),\\[2mm]
\xi_x(\pm L,y,z)  & = & 0  & \mbox{ in } & (\bar{y}, + \infty) \times (-Y,Y).\\[2mm]
\end{array}
\end{equation}
\end{problem}
\begin{theorem}
The Problem \ref{chap6:pb4} has a unique solution. Moreover, we have $-\psi(y) \leq  \xi(x,y,z)  \leq  \psi(y).$
\end{theorem}

\begin{proof}
We justify uniqueness of the solution taking $f=0$ and setting $\xi$
\[
\xi=w \psi^\alpha, \quad \alpha>1.
\] 
We show $w=0$. In particular, $w$ satisfies 
\begin{eqnarray*}
\xi          & = & w \psi^\alpha\\[2mm]
\xi_x      & = & w_x \psi^\alpha ; \quad \xi_y=w_y \psi^\alpha + w \alpha \psi_y \psi^{\alpha-1} ; \quad \xi_z = w_z \psi^\alpha\\[2mm]
\xi_{xx} & = &  w_{xx} \psi^\alpha ; \quad \xi_{yy}= w_{yy} \psi^\alpha + 2 w_{y} \alpha \psi_y \psi^{\alpha-1} + w \alpha \big{(} \psi_{yy} \psi^{\alpha -1} + (\psi_y)^2 \psi^{\alpha-2} (\alpha-1) \big{)}\\
\end{eqnarray*}
then
\begin{eqnarray*}
& & \frac{1}{2} w_{xx} \psi^\alpha + \frac{1}{2} \lbrace w_{yy} \psi^\alpha + 2 w_{y} \alpha \psi_y \psi^{\alpha-1}  w \alpha \big{(} \psi_{yy} \psi^{\alpha -1} + (\psi_y)^2 \psi^{\alpha-2} (\alpha-1) \rbrace\\
& & - (c_0y +kz +\beta x) \lbrace   w_y \psi^\alpha + w \alpha \psi_y \psi^{\alpha-1} \rbrace -\alpha x w_x \psi^\alpha + y w_z \psi^\alpha =0.
\end{eqnarray*}
Collecting terms related to $\psi^\alpha$, we obtain
\begin{eqnarray*}
& & \frac{1}{2} w_{xx}  + \frac{1}{2} \lbrace w_{yy}  + 2 w_{y} \alpha \frac{\psi_y}{ \psi } + w \alpha \big{(} \frac{\psi_{yy}}{\psi} + \frac{(\psi_y)^2}{ \psi^2} (\alpha-1)  \rbrace\\
& &  + (-c_0y -kz -\beta x) \lbrace   w_y  + w \alpha \frac{\psi_y}{ \psi} \rbrace -\alpha x w_x + y w_z =0,
\end{eqnarray*}
which implies for $z \in (-Y,Y) \mbox{ and } y > \bar{y}$,
\begin{eqnarray*}
& &\frac{1}{2} w_{xx}  + \frac{1}{2}  w_{yy}  +  w_{y} \lbrace \alpha \frac{\psi_y}{ \psi } -c_0 -kz -\beta x \rbrace+ \frac{w \alpha}{\psi} \lbrace \frac{1}{2}\psi_{yy}  + \frac{(\alpha-1)}{2} \frac{(\psi_y)^2}{ \psi}\\
& &  -(c_0y +kz +\beta x) \alpha \psi_y \rbrace -\alpha x w_x + y w_z =0
\end{eqnarray*}
and for $z=Y ; \quad y > \bar{y}$,
\begin{align*}
& \frac{1}{2} w_{xx}  + \frac{1}{2}  w_{yy}  +  w_{y} \lbrace \alpha \frac{\psi_y}{ \psi } -c_0 -kY -\beta x \rbrace+ \frac{w \alpha}{\psi} \lbrace \frac{1}{2}\psi_{yy}  + \frac{(\alpha-1)}{2} \frac{(\psi_y)^2}{ \psi}\\
& -(c_0y +kY +\beta x) \alpha \psi_y \rbrace -\alpha x w_x + y w_z =0
\end{align*}
and
\[
w(x,\bar{y},z)=0.
\]
We can find $\alpha>1$ with $\alpha -1$ small so that
\[
\frac{1}{2} \psi_{yy} - (c_0y+kz+\beta x) +(\alpha-1) \frac{(\psi_y)^2}{\psi} \leq -\gamma \frac{c_0}{2} + \frac{ (\alpha-1) \gamma }{ y^2 (\log(y) + K)} < 0.
\]
Since $w \to 0$ as $y \to \infty$, a positive maximum can be attained only for a finite $y^\star$.  But this is impossible from the equation. So $w$ is negative,$w \leq 0$. But, $w$ is also positive, $w \geq 0$. Hence $\xi =0$ and uniqueness follows.\\
Now, existence is demonstrated by the following approximation
\[
\begin{array}{l}
A \xi^R+ f=0 , \mbox{ in } \Delta_R,\\[2mm]
B_+ \xi^R + f=0 , \mbox{ in } \Delta_R^+,\\[2mm]
\xi(x,\bar{y},z)=0 , \quad \xi(x,R,z)=0,\\[2mm]
\xi_x(\pm L,y,z)=0.
\end{array}
\]
Using $u^R=\xi^R-\psi$, we obtain
\[
\begin{array}{l}
Au^R \geq 0, \mbox{ in } \Delta_R,\\[2mm]
B_+ u^R \geq 0, \mbox{ in } \Delta_R^+,\\[2mm]
u^R(x,\bar{y},z)= - \psi(\bar{y}),\\[2mm]
u^R_x(\pm L,y,z)=0.
\end{array}
\]
Necessarily, $u^R \leq 0$.
The sequence $\xi^R$ is monotone increasing and converges towards a solution.
Let us show estimates on $u$. Suppose $f \geq 0$, we will show that $0 \leq \xi \leq \psi$.
Then,
\begin{eqnarray*}
\psi_y(-c_0y - kz - \beta x) + \frac{1}{2} \psi_{yy} &  =  & \gamma [ \frac{1}{y}(-c_0y-kz-\beta x) - \frac{1}{2 y^2}]\\[2mm]
 & \leq & \gamma [ \frac{1}{y}(-c_0y-kz) - \frac{1}{2 y^2}]\\[2mm]
 & \leq & \gamma [ \frac{1}{y}(-c_0y-kz) - \frac{1}{2 y^2}]\\[2mm]
 & \leq & - \gamma \frac{c_0}{2}, \quad \mbox{ if }\bar{y}> \frac{2kY}{c_0}.\\[2mm]
\end{eqnarray*}
Define $\gamma$ such that $\frac{\gamma c_0}{2} \geq \| f \|$.
We then have, setting $u = \xi-\psi$,
\[
\begin{array}{l}
A u \geq 0, \mbox{ in } \Delta\\[2mm]
B_+ u \geq 0, \mbox{ in } \Delta^+\\[2mm]
u(x,\bar{y},z)=-\psi(\bar{y})\\[2mm]
u_x(\pm L,y,z)=0 
\end{array}
\]
It follows that $u\leq 0$.\\
We can show similarly that $\xi+\phi \geq 0$. Hence, we have
\[
-\psi \leq \xi \leq \psi.
\]
\end{proof}
\section{The operator $\mathbf{T}$ and the invariant measure $\nu$}

\setcounter{equation}{0}
\subsection{Construction of the operator $\mathbf{T}$}
Consider $f \in L^\infty(\mathcal{O})$, following the same procedure of the 1d case, we first solve the interior non-homogeneous Dirichlet problem  \underline{for $\chi$} with $f$ at the right hand side, then we solve the exterior non-homogeneous Dirichlet problem \underline{for $\xi$} with $f$ at the right hand side and $\chi$ as a boundary condition.
Also, for any $\bar{p}_1 \in \bar{\Gamma}_1$, we define the operator
\[
\mathbf{T}f(\bar{p}_1):= \xi(\bar{p}_1).
\]
This defines a linear operator from $ L^\infty \big{(} \bar{\Gamma}_1 \big{)}$ in $L^\infty(\bar{\Gamma}_1)$.
\subsection{Probabilistic interpretation of the operator $\mathbf{T}$}
For any $A$ a Borel subset of $\mathcal{O}$, consider the two following measures of occupation of $A$ by the process $(x(t),y(t),z(t))$ starting at $p \in \mathcal{O}$, namely
\[
\left\{
\begin{array}{rcl}
\nu_{\chi}(p;A)          & := & \mathbb{E}_p\big{(} \int_0^{\bar{\tau}_1} \mathbf{1}_A(x(s),y(s),z(s)) \d s \big{)} \\[2mm]
\nu_{\xi}(p;A) & := & \mathbb{E}_p\big{(} \int_0^{\bar{\tau}} \mathbf{1}_A(x(s),y(s),z(s)) \d s \big{)}.
\end{array}
\right.
\]
We have shown
\[
\chi(p) = \int_\mathcal{O} f(q) \d \nu_{\chi}(p; \d q) \quad \mbox{ and } \quad \xi(\tilde{p}) = \int_\mathcal{O} f(\tilde{q}) \d \nu_{\xi}(\tilde{p}; \d \tilde{q}). 
\]
For any $\bar{p}_1 \in \bar{\Gamma}_1$, we have 
\[
\mathbf{T}(f)(\bar{p}_1) := \mathbb{E}_{\bar{p}_1} \big{(} \int_0^{\bar{\tau}_1} f(x(s),y(s),z(s)) \d s\big{)} + \int_{\bar{\Gamma}} \mathbb{E}_q\big{(} \int_0^{\bar{\tau}} f(x(s),y(s),z(s)) \d s \big{)} \d\gamma(\bar{p}_1; \d q)
\]
$\mathbf{T}(f)$ integrates $f$ over a time cycle of the Markov chain stopped on $\bar{\Gamma}_1$, that is $(x(\bar{\tau}_{1,k}),y(\bar{\tau}_{1,k}),z(\bar{\tau}_{1,k}))$.\\[2mm]

\subsection{Construction of the invariant measure $\mathbf{\nu}$}
Now, define
\begin{equation*}
\nu(f) := \frac{\int_{-L}^L \int_{-Y}^{Y} Tf(r,\bar{y}_1,s) \gamma_\star^+(\d r,\d s) + \int_{-L}^L \int_{-Y}^{Y} Tf(r,-\bar{y}_1,s) \gamma_\star^-(\d r,\d s) }{ \int_{-L}^L \int_{-Y}^{Y} T \mathbf{1}(r,\bar{y}_1,s) \gamma_\star^+(\d r,\d s) + \int_{-L}^L \int_{-Y}^{Y} T \mathbf{1}(r,-\bar{y}_1,s) \gamma_\star^-(\d r,\d s)}.
\end{equation*}
The denominator is well defined and is stricly positive.
Now, we want to solve  the \underline{complete problem}. 
\begin{problem} \label{chap6:pb5}
Find $u$ such that $u \psi^{-1} \mbox{  is bounded for } |y|> |\bar{y} |$ and
\begin{equation}\nonumber
\begin{array}{rcccl}
Au+f & = & 0, & \mbox{ in } &   \mathcal{O} \\[2mm]
B_+u+f & = & 0,& \mbox{ in } & \mathcal{O}^+ \\[2mm]
B_-u+f & = & 0, & \mbox{ in } &  \mathcal{O}^- \\[2mm]
u_x(\pm L,y,z) & = & 0, &  \mbox{ in } &  \lbrace x = \pm L \rbrace \cap \mathcal{O}\\[2mm]
\end{array}
\end{equation}
Considering Problems \ref{chap6:pb1}, \ref{chap6:pb2}, \ref{chap6:pb3} and 4, in the following result we use a functional analysis method to prove that $\nu$ is the unique invariant distribution associated to the solution $(x(t),y(t),z(t))$ of the stochastic variational inequality (\ref{chap6:vi}).
\end{problem}
\begin{theorem}\label{chap6:THM}
The Problem \ref{chap6:pb5} has a unique solution if and only if $\nu(f)=0$.
\end{theorem}
\begin{proof}
Similar arguments as the 1d case are considered. 
Uniqueness is guaranteed by the ergodic property of the operator $\mathbf{P}$. 
Suppose $\nu(f)=0$. We define $\chi$ the solution of the Interior non-homogeneous Dirichlet Problem \ref{chap6:pb3} and $\xi$ the solution of the Exterior non-homogeneous Dirichlet Problem \ref{chap6:pb4}.
We set $\chi^0:=\chi$ and $\xi^0:=\xi$, and for $k \geq 0$, we define $\chi^{k+1}$  by
\begin{equation}\nonumber
\begin{array}{rcccl}
A \chi^{k+1}                       & = & 0, &  \mbox{ in }   &  \Delta_1,\\[2mm]
B_+ \chi^{k+1}                   & = & 0, &  \mbox{ in }   &  \Delta_1^+,\\[2mm]
B_- \chi^{k+1}                    & = & 0, &  \mbox{ in }   &  \Delta_1^-, \\[2mm] 
\chi_x^{k+1}(\pm L,y,z)     & = & 0, &  \mbox{ in }   &  (-\bar{y}_1,\bar{y}_1)\times(-Y,Y)\\[2mm]
\end{array}
\end{equation}
with
\[
\chi^{k+1}(x,\bar{y}_1,z)= \xi^k(x,\bar{y}_1,z) ; \quad  \chi^{k+1}(x,-\bar{y}_1,z)= \xi^k(x,-\bar{y}_1,z).
\]
Afterwards, we define $\xi_{k+1}$ by 
\begin{equation}\nonumber
\begin{array}{rcccl}
A \xi^{k+1}               & = & 0, & \mbox{ in } &  \Delta_u,\\[2mm]
B_+ \xi^{k+1}        & = & 0, & \mbox{ in } &  \Delta_u^+,\\[2mm]
\xi_x^{k+1}(\pm L,y,z) & = & 0, &  \mbox{ in } &  (\bar{y},+\infty)\times(-Y,Y)\\[2mm]
\end{array}
\end{equation}
with
\[ 
\xi^{k+1}(x,\bar{y},z)= \chi^{k+1}(x,\bar{y},z).
\] 
Similarly, we define $\xi^{k+1}$ by
\begin{equation}\nonumber
\begin{array}{rcccl}
A \xi^{k+1}                    & = & 0, & \mbox{ in } &  \Delta_d,\\[2mm]
B_- \xi^{k+1}           & = & 0, & \mbox{ in } &  \Delta_d^-,\\[2mm]
\xi_x^{k+1}(\pm L,y,z) & = & 0, &  \mbox{ in } &  (-\infty,-\bar{y})\times(-Y,Y)\\[2mm]
\end{array}
\end{equation}
with
\[ 
 \xi^{k+1}(x,-\bar{y},z)= \chi^{k+1}(x,-\bar{y},z).
\]
That means
\[
\xi^0|_{\bar{\Gamma}_1} = \mathbf{T} f ; \quad \xi^{k+1} |_{\bar{\Gamma}_1} = \mathbf{P}(\xi^k|_{\bar{\Gamma}_1}).
\]
Next, we set
\[
 \tilde{\xi}^k=\xi^0 + \xi^1 + ... + \xi^{k} ; \quad \tilde{\chi}^k=\chi^0 + \chi^1 + ... + \chi^{k} 
\]
then
\[
\tilde{\xi}^k|_{\bar{\Gamma}_1} = \mathbf{T}f + \mathbf{P}(\mathbf{T}f) + ... + \mathbf{P}^k(\mathbf{T}f). 
\]
Now, let us remark that $\tilde{\chi}^k$ satisfies the following equation: 
\begin{equation}\label{chap6:tildechik}
\begin{array}{rcccl}
A \tilde{\chi}^{k} + f                         & = & 0, & \mbox{ in } & \Delta_1,\\[2mm]
B_+ \tilde{\chi}^{k} + f                 & = & 0, & \mbox{ in } & \Delta_1^+,\\[2mm]
B_- \tilde{\chi}^{k} + f                & = & 0, & \mbox{ in } & \Delta_1^-,\\[2mm] 
\tilde{\chi}^{k}_x(\pm L,y,z)           & = & 0, &  \mbox{ in } &  (-\infty,-\bar{y})\times(-Y,Y)\\[2mm]
\end{array}
\end{equation}
with
\[
\tilde{\chi}^{k}(x,\bar{y}_1,z) = \tilde{\xi}^{k-1}(x,\bar{y}_1,z) ; \quad  \tilde{\chi}^{k}(x,-\bar{y}_1,z)= \tilde{\xi}^{k-1}(x,-\bar{y}_1,z)
\]
$\tilde{\xi}^k$ satisfies the following equations:
\begin{equation}\label{chap6:tildexik1}
\begin{array}{rcccl}
A \tilde{\xi}^{k} + f                    & = & 0, & \mbox{ in } & \Delta_u,\\[2mm]
B_+ \tilde{\xi}^{k} + f            & = & 0, & \mbox{ in } & \Delta_u^+,\\[2mm]
\tilde{\xi}^{k}_x(\pm L,y,z)      & = & 0, & \mbox{ in } & (\bar{y},+ \infty) \times (-Y,Y)\\[2mm]
\end{array}
\end{equation}
\[
\tilde{\xi}^{k}(x,\bar{y},z)= \tilde{\chi}^k(x,\bar{y},z)
\]
and
\begin{equation}\label{chap6:tildexik2}
\begin{array}{rcccl}
A \tilde{\xi}^{k} + f                    & = & 0, & \mbox{ in } & \Delta_d,\\[2mm]
B_+ \tilde{\xi}^{k} + f            & = & 0, & \mbox{ in } & \Delta_d^-,\\[2mm]
\tilde{\xi}^{k}_x(\pm L,y,z)     & = & 0, & \mbox{ in } & (-\infty,-\bar{y}) \times (-Y,Y)\\[2mm]
\end{array}
\end{equation} 
\[
\tilde{\xi}^{k}(x,-\bar{y},z)= \tilde{\chi}^k(x,-\bar{y},z).
\] 
The condition $\nu(f)=0$ means 
\[
\int_{\bar{\Gamma}_1} \mathbf{T}f(x) \d \pi(x) = \int_{\bar{\Gamma}_1} \mathbf{T}f(r,-\bar{y}_1,s) \d \pi_1(r,s) + \int_{\bar{\Gamma}_1} \mathbf{T}f(r,\bar{y}_1,s) \d \pi_2(r,s)=0.
\]
From the estimate (\ref{chap6:ERGODIQUE}), we obtained 
\[
\tilde{\xi}^k \mbox{ converges in } L^\infty(\bar{\Gamma}_1).
\]
Now, we notice that $\tilde{\chi}^k -\chi$ is a solution of the interior homogeneous Dirichlet problem with $(\tilde{\chi}^k -\chi)|_{\bar{\Gamma}_1}=  \tilde{\xi}^{k-1}|_{\bar{\Gamma}_1}$ and $\tilde{\xi}^k - \xi$ is a solution of the Exterior homogeneous Dirichlet problem with $(\tilde{\xi}^k -\xi)|_{\bar{\Gamma}}=(\tilde{\chi}^k -\chi)|_{\bar{\Gamma}}$. 
Then, we obtain
\[
\| \tilde{\xi}^k - \xi \|_{L^\infty} \leq \| \tilde{\chi}^k -\chi \|_{L^\infty} \leq \| \tilde{\xi}^{k-1} \|_{L^\infty} \leq C. 
\]
We can extract a subsequence such that
\[
\tilde{\xi}^k \to \tilde{\xi} ; \quad \tilde{\chi}^k \to \tilde{\chi}. 
\] 
Moreover, $\tilde{\chi}$ satisfies equation:
 \begin{equation}\label{chap6:tildexi}
\begin{array}{rcccl}
A \tilde{\chi}  + f                 & = & 0, & \mbox{ in } & \Delta_1,\\[2mm]
B_+ \tilde{\chi} + f          & = & 0, & \mbox{ in } &  \Delta_1^+,\\[2mm]
B_- \tilde{\chi} + f         & = & 0, & \mbox{ in } &  \Delta_1^-,\\[2mm] 
\tilde{\chi}_x(\pm L,y,z)    & = & 0, & \mbox{ in } & (-\bar{y}_1,\bar{y}_1) \times (-Y,Y)\\[2mm]
\end{array}
\end{equation}
\[
\tilde{\chi}(x,\bar{y}_1,z)= \tilde{\xi}(x,\bar{y}_1,z) ; \quad  \tilde{\chi}(x,-\bar{y}_1,z)= \tilde{\xi} (x,-\bar{y}_1,z)
\]
and $\tilde{\xi}$ satisfies equations: 
\begin{equation}\label{chap6:tildechik1}
\begin{array}{rcccl}
A \tilde{\xi}  + f                     & = & 0, & \mbox{ in } & \Delta_u,\\[2mm]
B_+ \tilde{\xi} + f              & = & 0,&  \mbox{ in } & \Delta_u^+,\\[2mm]
\tilde{\xi}_x(\pm L,y,z) + f  & = & 0, & \mbox{ in } & (\bar{y},+ \infty) \times (-Y,Y)\\[2mm]
\end{array}
\end{equation}
\[
\tilde{\xi} (x,\bar{y},z) = \tilde{\chi} (x,\bar{y},z)
\]
\begin{equation}\label{chap6:tildechik2}
\begin{array}{rcccl}
A \tilde{\xi} + f                    & = & 0, & \mbox{ in } & \Delta_d,\\[2mm]
B_- \tilde{\xi} + f           & = & 0, & \mbox{ in } & \Delta_d^-,\\[2mm]
\tilde{\xi}_x(\pm L,y,z)      & = & 0, & \mbox{ in } & (-\infty,-\bar{y}) \times (-Y,Y)\\[2mm]
\end{array}
\end{equation} 
\[
\tilde{\xi} (x,-\bar{y},z)= \tilde{\chi} (x,-\bar{y},z).
\] 
Then, we must have
\[
\tilde{\chi} = \tilde{\xi} \quad \mbox{ in } \quad \bar{y} < \pm y < \bar{y}_1  
\]
Thus, the function
\[
u = \begin{cases} \tilde{\chi}  & \mbox{ in } \Delta_1, \\ \tilde{\xi} & \mbox{ in }\Delta_1^c \end{cases}
\]
is the solution of Problem \ref{chap6:pb5}.\\
Now, suppose that Problem \ref{chap6:pb5} has a unique solution $u$. Considering the same sequences $\tilde{\chi}^k$ and $\tilde{\xi}^k$, we have that $u-\tilde{\chi}^k$ is a solution of the interior homogeneous Dirichlet problem with $(u-\tilde{\chi}^k)|_{\bar{\Gamma}_1} =  (u-\tilde{\xi}^{k-1})|_{\bar{\Gamma}_1}$ and $u-\tilde{\xi}^k$ is a solution of the exterior homogeneous Dirichlet problem $(u-\tilde{\xi}^k)|_{\bar{\Gamma}}=(u-\tilde{\chi}^k)|_{\bar{\Gamma}}$. 
Hence,
\[
\| u - \tilde{\xi}^k \|_{L^\infty(\bar{\Gamma}_1)}  \leq \| u - \tilde{\chi}^k \|_{L^\infty(\bar{\Gamma}_1)} \leq \| u - \tilde{\xi}^{k-1} \|_{L^\infty(\bar{\Gamma}_1)} \leq \| u \|_{L^\infty(\bar{\Gamma}_1)} 
\]
and so, 
\[
\| \tilde{\xi}^k \|_{L^\infty(\bar{\Gamma}_1)}  \leq C.
\]
We have
\[
\tilde{\xi}^k = (k+1) \int_{\bar{\Gamma}_1} \mathbf{T}f(x)\d \pi(x) + \sum_{j=0}^{k} \mathbf{P}^j(\mathbf{T}(f-\nu(f)))
\]
and since the sum is bounded, we obtain 
\[
(k+1) \int_{\bar{\Gamma}_1} \mathbf{T}f(x)\d \pi(x) \mbox{ is bounded.}
\]
That leads to
\[
\int_{\bar{\Gamma}_1} \mathbf{T}f(x)\d \pi(x) = 0
\]
That implies $\nu(f)=0$.
\end{proof}

Consider now $\varphi$ a smooth function in $[-L,L] \times \mathbb{R} \times [-Y,Y]$ with compact support. If we take
\begin{equation*}
\begin{array}{l}
f = - A\varphi,\\[2mm]
f(x,y,Y) = - B_+\varphi,\\[2mm]
f(x,y,-Y) = - B_-\varphi,\\[2mm]
\end{array}
\end{equation*}
then $\varphi$ is solution of (\ref{chap6:pb5}) for this $f$. From Theorem \ref{chap6:THM}, we have
\begin{equation}\nonumber
\begin{array}{ccl}
 & & \int_{-L}^L \int_{-\infty}^{\infty} \int_{-Y}^{Y} \big{\lbrace} \dfrac{1}{2} \varphi_{yy}  +\dfrac{1}{2} \varphi_{xx}  - \alpha x \varphi_{x} - (\beta x + c_0y+kz)  \varphi_{y} + y \varphi_z  \big{\rbrace} \d \nu(x,y,z)\\[2mm]
& + &  \int_{-L}^L \int_{0}^{\infty} \big{\lbrace} \dfrac{1}{2} \varphi_{yy}  +\dfrac{1}{2} \varphi_{xx}  - \alpha x \varphi_{x} - (\beta x + c_0y+kY)  \varphi_{y}  \big{\rbrace} \d \nu(x,y,Y)\\[2mm] 
& + &  \int_{-L}^L \int_{-\infty}^{0} \big{\lbrace} \dfrac{1}{2} \varphi_{yy}  +\dfrac{1}{2} \varphi_{xx}  - \alpha x \varphi_{x} - (\beta x + c_0y-kY)  \varphi_{y} \big{\rbrace} \d \nu(x,y,-Y) =0.
\label{chap6:m4}
\end{array}
\end{equation}
If $\nu$ has a density $m$, then we deduce that
\begin{equation}
 \alpha \frac{\partial}{\partial x}[xm] + \frac{\partial}{\partial y}[(\beta x + c_0 y + kz)m] -y \frac{\partial m }{\partial z} + \frac{1}{2}\frac{\partial^2 m}{\partial x^2} + \frac{1}{2}\frac{\partial^2 m}{\partial y^2} = 0 \quad \mbox{ in } (0,L) \times \mathbb{R} \times (-Y,Y) 
\label{chap6:m5}
\end{equation}
in the sense of distributions. If we test (\ref{chap6:m5}) with $\varphi$ and integrate by parts, we obtain 
\begin{eqnarray*}
-\int_{-L}^L \int_{-\infty}^{+\infty} y m(x,y,Y) \varphi(x,y,Y) d xd y + \int_0^L \int_{-\infty}^{+\infty} y m(x,y,-Y) \varphi(x,y,-Y) \d x \d y\\[2mm]
 +\int_{-L}^L \int_{-\infty}^{\infty} \int_{-Y}^{Y} m(x,y,z) \big{\lbrace} \dfrac{1}{2} \varphi_{yy}  +\dfrac{1}{2} \varphi_{xx}  - \alpha x \varphi_{x} - (\beta x + c_0y+kz)  \varphi_{y} + y \varphi_z  \big{\rbrace} \d x \d y \d z =0 
\end{eqnarray*}
and comparing with (\ref{chap6:m4}) 
\begin{eqnarray*}
-\int_{-L}^L \int_{-\infty}^{0} y m(x,y,Y) \varphi(x,y,Y) d xd y + \int_0^L \int_{0}^{+\infty} y m(x,y,-Y) \varphi(x,y,-Y) \d x \d y\\[2mm]
 -\int_{-L}^L \int_{0}^{\infty} m(x,y,Y) \big{\lbrace} \dfrac{1}{2} \varphi_{yy}  +\dfrac{1}{2} \varphi_{xx}  - \alpha x \varphi_{x} - (\beta x + c_0y+kY)  \varphi_{y} + y \varphi \big{\rbrace} \d x \d y\\[2mm]
 -\int_{-L}^L \int_{-\infty}^{0} m(x,y,-Y) \big{\lbrace} \dfrac{1}{2} \varphi_{yy}  +\dfrac{1}{2} \varphi_{xx}  - \alpha x \varphi_{x} - (\beta x + c_0y-kY)  \varphi_{y}  - y \varphi \big{\rbrace} \d x \d y =0
\end{eqnarray*}
we finally deduce
\begin{eqnarray*}
ym + \frac{\partial}{ \partial x}[ xm ] + \frac{\partial}{ \partial y}[ (\beta x + c_0y+kY) m ] + \dfrac{1}{2} \dfrac{\partial^2 m}{\partial x^2 } + \dfrac{1}{2} \dfrac{\partial^2 m}{\partial y^2 }  = 0 , \quad  \mbox{ on } \mathcal{O}^+\\[3mm]
-ym + \frac{\partial}{ \partial x}[ xm ] + \frac{\partial}{ \partial y}[ (\beta x + c_0y-kY) m ] + \dfrac{1}{2} \dfrac{\partial^2 m}{\partial x^2 } + \dfrac{1}{2} \dfrac{\partial^2 m}{\partial y^2 }  =  0, \quad  \mbox{ on }\mathcal{O}^-\\[3mm]
 m  =  0, \quad  \mbox{ on } (-L,L) \times (-\infty,0) \times \{ Y \} \cup (-L,L) \times (0,\infty) \times \{ -Y \}
\end{eqnarray*}
The proof of the main result is complete.

\section{Appendix: Proofs of lemmas \ref{chap6:lem:l1},\ref{chap6:lem:l2},\ref{chap6:lem:l3},\ref{chap6:lem:l4} }
\begin{proof}[Proof of Lemma \ref{chap6:lem:l1}]
Denote $\gamma = \| \phi \|_{L^\infty}$.
Notice we have
\[
\begin{array}{rcll}
(u-\gamma)^+(x,\pm \bar{y}_1, z) & = & 0  & \mbox{ if }  x \in (-L,L), \qquad z \in (-Y,Y) \\[2mm]
(u-\gamma)^+(x,y,\pm Y)& = & 0 & \mbox{ if }  x \in (-L,L), \qquad  0< \pm y < \bar{y}_1\\[2mm] 
\end{array}
\] 
and
\[
\begin{array}{rcll}
(u+\gamma)^-(x,\pm \bar{y}_1, z) & = & 0  & \mbox{ if }  x \in (-L,L), \qquad z \in (-Y,Y) \\[2mm]
(u+\gamma)^-(x,y,\pm Y)& = & 0 & \mbox{ if }  x \in (-L,L), \qquad 0< \pm y < \bar{y}_1.\\[2mm] 
\end{array}
\] 
Choose $v=u -(u-\gamma)^+ \in K$ and we obtain
\[
-a(u,(u-\gamma)^+) - \lambda (u,(u-\gamma)^+) \geq -\lambda (w, (u-\gamma)^+)
\]
\[
a(u,(u-\gamma)^+) + \lambda (u,(u-\gamma)^+) \leq \lambda (w, (u-\gamma)^+).
\]
Switching the first argument by $(u-\gamma)$ in the previous inequality, we obtain 
\[
a((u-\gamma)^+,(u-\gamma)^+) + \lambda | (u-\gamma)^+|_{L^2}^{2} \leq \lambda (w - \gamma, (u-\gamma)^+)
\]
and $w \leq \gamma$ implies $(u-\gamma)^+=0$.
Taking $v=u +(u+\gamma)^-$ similar arguments yields $(u+\gamma)^- =0$.
\end{proof}

\begin{proof}[Proof of Lemma \ref{chap6:lem:l2}]
We have the following expressions,
\[
a(\eta^{\epsilon},\eta^\epsilon)=\int_{\Delta_1} \frac{\epsilon}{2} (\eta^\epsilon_z)^2 + \frac{1}{2}(\eta^\epsilon_y)^2 + \frac{1}{2} (\eta^\epsilon_x)^2 + (\beta x+ kz + c_0 y) \eta^\epsilon \eta^\epsilon_y - y\eta^\epsilon \eta^\epsilon_z + \alpha x \eta^\epsilon \eta^\epsilon_x 
\]
and
\[
a(\eta^{\epsilon},u_0)=\int_{\Delta_1} \frac{\epsilon}{2} \eta^\epsilon_z u_{0z} + \frac{1}{2} \eta^\epsilon_yu_{0y} + \frac{1}{2} \eta^\epsilon_x u_{0x} + (\beta x+ kz + c_0 y) u_0 \eta^\epsilon_y - y u_0 \eta^\epsilon_z + \alpha x u_0 \eta^\epsilon_x. 
\]
 Inequality $a(\eta^\epsilon,\eta^\epsilon) \leq a(\eta^\epsilon,u_0)$ means 
 \begin{align*}
 \frac{\epsilon}{2} \int_{\Delta_1} (\eta_z^\epsilon)^2 +  \frac{1}{2}\int_{\Delta_1} (\eta_x^\epsilon)^2 + \frac{1}{2}\int_{\Delta_1} (\eta_y^\epsilon)^2 & \leq   \frac{\epsilon}{2}  \int_{\Delta_1} \eta^\epsilon_z u_{0z} + \frac{1}{2}  \int_{\Delta_1} \eta^\epsilon_yu_{0y}\\
 - &   \int_{\Delta_1} (\beta x+  c_0 y + kz) \eta^\epsilon \eta^\epsilon_y + \frac{1}{2}  \int_{\Delta_1}  (\beta x+  c_0 y + kz) u_0 \eta^\epsilon_y\\
 + & \frac{1}{2}  \int_{\Delta_1} \eta^\epsilon_x u_{0x}+  \int_{\Delta_1} \alpha x \eta^\epsilon \eta^\epsilon_x -  \int_{\Delta_1} \alpha x u_0 \eta^\epsilon_x\\
 + &\int_{\Delta_1} y\eta^\epsilon \eta^\epsilon_z -  \int_{\Delta_1} y u_0 \eta^\epsilon_z
 \end{align*}
 We apply Cauchy-Schwartz inequality to the first two terms, then we apply Green formula to the last one and we get
 \begin{align*}
 \frac{\epsilon}{2} \int_{\Delta_1} (\eta_z^\epsilon)^2 +  \frac{1}{2}\int_{\Delta_1} (\eta_x^\epsilon)^2 + \frac{1}{2}\int_{\Delta_1} (\eta_y^\epsilon)^2 & \leq   \epsilon \sqrt{\int_{\Delta_1} (\eta_z^\epsilon)^2} \frac{1}{2} (\int_{\Delta_1} (u_{0z}^\epsilon)^2 )^\frac{1}{2}\\ 
 + &\sqrt{\int_{\Delta_1} (\eta_y^\epsilon)^2} \lbrace \frac{1}{2} (\int_{\Delta_1} (u_{0y}^\epsilon)^2 )^\frac{1}{2} + \| \phi \| (\int_{\Delta_1} (\beta x + c_0 y + kz)^2 )^\frac{1}{2}\rbrace\\
 + & \sqrt{ \int_{\Delta_1} (\eta_x^\epsilon)^2 } \lbrace \frac{1}{2} (\int_{\Delta_1} (u_{0x}^\epsilon)^2 )^\frac{1}{2} + \alpha \| \phi \| (\int_{\Delta_1} x^2 )^\frac{1}{2} + \alpha (\int_{\Delta_1} x^2 u_0 )^\frac{1}{2} \rbrace  \\
 + & \| \phi \|^2 \int_{-L}^L \int_{-\bar{y}_1}^{\bar{y}_1} |y| d xd y +  \| \phi \| (\int_{\Delta_1} |y u_{0z}|) + \int_{-L}^{L} \int_{-\bar{y}_1}^{\bar{y}_1} |y| |u_0(x,y,Y)|\\
 + & \int_{-L}^{L} \int_{-\bar{y}_1}^{\bar{y}_1} |y| |u_0(x,y,-Y)| 
  \end{align*}
with $c_1$,$c_2$,$c_3$ and $c_4$ are positive constant, we get 
\[
 \| \sqrt{\epsilon} \eta_z^\epsilon \|_{L^2}^2 + \| \eta_y^\epsilon \|_{L^2}^2 + \| \eta_x^\epsilon \|_{L^2}^2 \leq  \sqrt{\epsilon} c_1 \| \sqrt{\epsilon} \eta_z^\epsilon \|_{L^2} + c_2 \| \eta_y^\epsilon \|_{L^2} +c_3 \| \eta_x^\epsilon \|_{L^2} + c_4.
 \] 
\end{proof}

\begin{proof}[Proof of Lemma \ref{chap6:lem:l3}]
We have 
\begin{align*}
\int_{\Delta_1} \eta^k_{yy} \eta^k_z y^{2p-1} \theta^{2q} & =   -\int_{\Delta_1} \eta^k_y \eta^k_{zy} y^{2p-1} \theta^{2q} -\int_{\Delta_1} \eta^k_y \eta^k_z (2p-1) y^{2p-2} \theta^{2q} - \int_{\Delta_1} \eta^k_y \eta^k_z y^{2p-1} 2 q \theta' \theta^{2q-1} \\
& =  -\frac{1}{2} \int_{\Delta_1} ((\eta^k_y)^2)_z y^{2p-1} \theta^{2q} -(2p-1) \int_{\Delta_1} (\eta^k_y \theta) (\eta^k_z y^{2p-2} \theta^{2q-1})\\
& \quad - 2q \int_{\Delta_1} (\eta^k_y \theta)  (\eta^k_z y^{2p-1} \theta^{2q-2}) \theta' \\
& =  -\frac{1}{2} \int_{-L}^{L} \int_{-\bar{y}_1}^{\bar{y}_1} (\eta^k_y)^2(x,y,Y) y^{2p-1} \theta^{2q} + \frac{1}{2} \int_{-L}^{L} \int_{-\bar{y}_1}^{\bar{y}_1} (\eta^k_y)^2(x,y,-Y) y^{2p-1} \theta^{2q}\\
&  \quad -(2p-1) \int_{\Delta_1} (\eta^k_y \theta) (\eta^k_z \pi) y^{p-2} \theta^{q-1} - 2q \int_{\Delta_1} (\eta^k_y \theta)  (\eta^k_z \pi) y^{p-1} \theta^{q-2} \theta',\\[2mm]
\int_{\Delta_1} \eta^k_{xx} \eta^k_z y^{2p-1} \theta^{2q} & =   -\frac{1}{2} \int_{\Delta_1} ((\eta^k_x)^2)_z y^{2p-1} \theta^{2q}\\  
& =   -\frac{1}{2} \int_{-L}^{L} \int_{-\bar{y}_1}^{\bar{y}_1} (\eta^k_x)^2(x,y,Y) y^{2p-1} \theta^{2q} + \frac{1}{2} \int_{-L}^{L} \int_{-\bar{y}_1}^{\bar{y}_1} (\eta^k_x)^2(x,y,-Y) y^{2p-1} \theta^{2q}\\[2mm]
\end{align*}
and
\begin{eqnarray*}
\int_{\Delta_1} (c_0 y + kz + \beta x) \eta^k_y \eta^k_z y^{2p-1} \theta^{2q}  & = & \int_{\Delta_1} (c_0 y + kz + \beta x) (\eta^k_y \theta) (\eta^k_z \pi) y^{p-1} \theta^{q-1}\\
\int_{\Delta_1} \alpha x \eta^k_x \eta^k_z y^{2p-1} \theta^{2q}  & = & \int_{\Delta_1} \alpha x (\eta^k_x \theta) (\eta^k_z \pi) y^{p-1} \theta^{q}.\\     
\end{eqnarray*}
Testing $(A \eta^k)$ with $\eta^k_z y^{2p-1} \theta^{2q}$, we deduce
\begin{eqnarray*}
\int (\eta^k_z \pi)^2 & = &  \frac{1}{4} \int_{-L}^{L} \int_{-\bar{y}_1}^{\bar{y}_1} (\eta^k_y)^2(x,y,Y) y^{2p-1} \theta^{2q} - \frac{1}{4} \int_{-L}^{L} \int_{-\bar{y}_1}^{\bar{y}_1} (\eta^k_y)^2(x,y,-Y) y^{2p-1} \theta^{2q}\\ 
&  & + \frac{1}{4} \int_{-L}^{L} \int_{-\bar{y}_1}^{\bar{y}_1} (\eta^k_x)^2(x,y,Y) y^{2p-1} \theta^{2q} - \frac{1}{4} \int_{-L}^{L} \int_{-\bar{y}_1}^{\bar{y}_1} (\eta^k_x)^2(x,y,-Y) y^{2p-1} \theta^{2q}\\
&  & + (p-\frac{1}{2}) \int_{\Delta_1} (\eta^k_y \theta) (\eta^k_z \pi) y^{p-2} \theta^{q-1} + q \int_{\Delta_1} (\eta^k_y \theta)  (\eta^k_z \pi) y^{p-1} \theta^{q-2} \theta'\\
&  & + \int_{\Delta_1} (c_0 y + kz + \beta x) (\eta^k_y \theta) (\eta^k_z \pi) y^{p-1} \theta^{q-1}\\
&  & + \int_{\Delta_1} \alpha x (\eta^k_x \theta) (\eta^k_z \pi) y^{p-1} \theta^{q}.
\end{eqnarray*}
\end{proof}

\begin{proof}[Proof of Lemma \ref{chap6:lem:l4}]
\begin{enumerate}
\item
With
\[
f(\eta^k,\eta_y^k):= -\eta^k (\frac{\pi''}{2} - (c_0y+kz+\alpha x) \pi') -  \eta_y^k  \pi' ,
\]
we have $\forall \psi \in H^1(\Delta_1), \psi(x,y,\pm Y) = 0$ and $\psi(x,\pm \bar{y}_1,z)=0$, 
\[
 \frac{1}{2} \int_{\Delta_1} v^k_y \psi_y  +\frac{1}{2} \int_{\Delta_1} v^k_x \psi_x +
\int_{\Delta_1} \big ( \alpha x v^k_x  + (\beta x + c_0y + kz) v^k_y -yv^k_z\big ) \psi = \int_{\Delta_1}  f(\eta^k,\eta^k_y)  \psi.
\]
Now, when $k$ goes to $+\infty$, we get
\[
 \frac{1}{2} \int_{\Delta_1} v_y \psi_y  +\frac{1}{2} \int_{\Delta_1} v_x \psi_x +
\int_{\Delta_1} \big ( \alpha x v_x  + (\beta x + c_0y + kz) v_y -yv_z\big ) \psi = \int_{\Delta_1}  f(\eta,\eta_y)  \psi.
\]
We deduce that in $H^{-1}(\Delta_1)$ we firstly have $-A v = f(\eta,\eta_y)$ which is equivalent to $\pi A \eta = 0$ and secondly that choice of test function implies 
$\pi \eta_x(\pm L,y,z)=0$ in $(H^{\frac{1}{2}}((-L,L)\times(-\bar{y}_1,\bar{y}_1))'$. 
\item
\begin{enumerate}
\item
We know that $\pi \eta^k \in H^{1}(\Delta_1)$, its trace is well defined and satisfies,
\[
\begin{array}{l}
\gamma(\pi \eta^k)(x,y,Y)=\pi (\chi^{+,k} + \beta^{k,+}) ; \quad y>0\\[2mm]
\gamma(\pi \eta^k)(x,y,-Y)=\pi \chi^{-,k} ; \quad y<0\\[2mm]
\end{array}
\]
with
\begin{equation}
\label{chap6:rff}
\| \chi^{\pm, k } \|_{L^\infty} \leq \| \phi \|_{L^\infty}
\end{equation}
and  $\chi^{-,k},\chi^{+,k}$ satisfy respectively (\ref{chap6:p11}) and (\ref{chap6:p15}) 
\begin{equation}
\begin{array}{l}
\int_{-L}^{L} \int_{0}^{\bar{y}_1} \lbrace \frac{1}{2}(\chi_x^{+,k} \psi_x + \chi_y^{+,k} \psi_y) + (\alpha x \chi_x^{+,k} + (\beta x + c_0 y + kY) \chi_y^{+,k} )\psi \rbrace \d x \d y=0,\\[3mm]
 \forall \psi \in H^1((-L,L)\times(0,\bar{y}_1)) \mbox{ with } \psi(x,0)=\psi(x,\bar{y}_1)=0\\[3mm]
\end{array}
\label{chap6:p11}
\end{equation}
and
\begin{equation}
\begin{array}{l}
\int_{-L}^{L} \int_{-\bar{y}_1}^{0} \lbrace \frac{1}{2}(\chi_x^{-,k} \psi_x + \chi_y^{-,k} \psi_y) + (\alpha x \chi_x^{-,k} + (\beta x + c_0 y - kY) \chi_y^{-,k} )\psi \rbrace \d x \d y=0,\\[3mm]
 \forall \psi \in H^1((-L,L)\times(-\bar{y}_1,0)) \mbox{ with } \psi(x,0)=\psi(x,-\bar{y}_1)=0.\\[3mm]
\end{array}
\label{chap6:p15}
\end{equation}
First, we study convergence of the sequence $\chi^{\pm,k}$ and we deduce PDEs satisfied by $\lim_{k \to 0 } \chi^{\pm,k}$.\\
In particular (\ref{chap6:rff}) implies  
\[
\begin{array}{c}
\chi^{+,k} \to \chi^{+} \mbox{ in } L^2((-L,L)\times(0,\bar{y}_1)) \mbox{ weakly, }\\[2mm]
\chi^{-,k} \to \chi^{-} \mbox{ in } L^2((-L,L)\times(-\bar{y}_1,0)) \mbox{ weakly. }\\[2mm]
\end{array}
\]
and equalities (\ref{chap6:p11}), (\ref{chap6:p15}) imply
\[
\begin{array}{c}
\|  \chi^{+,k} \pi \|_{H^1} \leq C \quad \mbox{ and } \quad  \chi^{+,k} \pi  \to  \chi^{+}\pi \mbox{ in } H^1((-L,L)\times(0,\bar{y}_1)) \mbox{ weakly, }\\[2mm]
\|  \chi^{-,k} \pi \|_{H^1} \leq C \quad \mbox{ and } \quad  \chi^{-,k} \pi \to  \chi^{-} \pi \mbox{ in } H^1((-L,L)\times(-\bar{y}_1,0)) \mbox{ weakly. }\\[2mm]
\end{array}
\]
Denote $\xi^{\pm, k}:= \chi^{\pm,k} y^2 $ and $g(\chi^{+,k},\chi^{+,k}_y):=  -\chi^{+,k} \lbrace 1 - 2y(\alpha x + c_0y +kY)  \rbrace - 2 y \chi^{+,k}_y $. From (\ref{chap6:p11}), we have $B_{Y} \xi^{+,k} = g(\chi^{+,k},\chi^{+,k}_y)$ in $(-L,L) \times (-\bar{y}_1,\bar{y}_1)$, in $H^{-1}((-L,L) \times (0,\bar{y}_1))$. The operator $B_+$ is strictly elliptic then $\xi^{+,k} \in H^2((-L,L)\times (0, \bar{y}_1))$. We have $\xi_x^{+,k}(\pm L, y)=0$ in $(H^{\frac{1}{2}}(0,\bar{y}_1))'$.
As $\xi^{+,k} \in H^2(\Delta_1)$ and $\pi B_+ \chi^{+,k} =0$, we have
\[
-B_+ \xi^{+,k}  = g(\chi^{+,k},\chi^{+,k}_y) \quad \mbox{ in a strong sense }
\]
We obtain that $\forall \psi \in H^1((-L,L)\times (0,\bar{y}_1)), \psi(x,0) = \psi(x,\bar{y}_1) =  0$, 
\begin{align*}
 \int_{-L}^{L} \int_{0}^{\bar{y}_1} \frac{1}{2} ( \xi^{+,k}_x \psi_x + \xi^{+,k}_y \psi_y)  & + \big ( \alpha x  \xi^{+,k}_x  + (\beta x + c_0y + kY)  \xi^{+,k}_y -y \xi^{+,k}_z \big ) \psi\\
 & =  \int_{-L}^{L} \int_{0}^{\bar{y}_1}  g(\chi^k,\chi_y^k) \psi.
\end{align*}
Now, when $k$ goes to $0$, we have
\[
 \int_{-L}^{L} \int_{0}^{\bar{y}_1} \frac{1}{2} ( \xi^{+}_x \psi_x + \xi^{+}_y \psi_y)  + \big ( \alpha x  \xi^{+}_x  + (\beta x + c_0y + kY)  \xi^{+}_y -y \xi^{+}_z \big ) \psi =  \int_{-L}^{L} \int_{0}^{\bar{y}_1}  g(\chi,\chi_y) \psi.
\]
We deduce that in $H^{-1}((-L,L)\times (0, \bar{y}_1))$ we have $-B_+  \xi^{+} = g( \chi^{+},\chi^+_y)$ which is equivalent to $y^2B_+ \chi^+ = 0$ and that choice of test function implies 
$y^2 \chi^+_x(\pm L,y)=0$ in $(H^{\frac{1}{2}}((0, \bar{y}_1)))'$. 
\item
Firstly, $\gamma(\pi \eta^k) \to \pi (\chi^{+} + \beta^+)$ in $H^1((-L,L)\times (0,\bar{y}_1))$ weakly. Secondly, weak convergence of $\pi \eta^k \to \pi \eta$ in $H^1(\Delta_1)$ implies weak convergence of $\gamma(\pi \eta^k) \to \gamma(\pi \eta)$ in $H^{\frac{1}{2}}( \partial \Delta_1)$. By uniqueness of the limit, we have $\gamma(\pi \eta) = \pi (\chi^{+} + \beta^+)$.
\end{enumerate}
\item
Using Green formula
\[
\forall \psi \in H_0^1(\Delta_1) \cap H^2(\Delta_1), \quad \int_{\Delta_1} \eta^k \psi_{yy} + \int_{\Delta_1} \eta_y^k \psi_y = \int_{\partial \Delta_1} \phi^k \psi_y \d \sigma. 
\] 
Now, when $k$ goes to $\infty$, we obtain
\[
\forall \psi \in H_0^1(\Delta_1) \cap H^2(\Delta_1), \quad \int_{\Delta_1} \eta \psi_{yy} + \int_{\Delta_1} \eta_y \psi_y = \int_{\partial \Delta_1} \phi \psi_y \d \sigma. 
\]  
\end{enumerate}
\end{proof}

\begin{lemma}We have
\[
 \eta_{xx} \pi^2 \in L^2(\Delta_1) ; \quad \eta_{xy}\pi^2 \in L^2(\Delta_1)  ; \quad  \eta_{yy}\pi^2 \in L^2(\Delta_1) 
\]
\end{lemma}
\begin{proof}
Denote $w:= v_x$. Deriving equation (\ref{chap6:regu1}) with respect to $x$, we obtain equality (\ref{chap6:regu2}),
\begin{equation}
\label{chap6:regu2}
-\frac{1}{2}w_{xx} -\frac{1}{2}w_{yy} + \Lambda(x,y,z) w_y + \alpha x w_x - y w_z = -\beta v_y - \alpha v_x + \rho_1 \eta_x + (\rho_1)_x \eta + \rho_2(y) \eta_{xy}
\end{equation}
Testing equality (\ref{chap6:regu2}) with $\pi^2$, we obtain
\begin{align*}
\frac{1}{2} \int_{\Delta_1} (w_x \pi )^2 + \frac{1}{2} \int_{\Delta_1} (w_y \pi)^2 & = - \int_{\Delta_1} w_y \pi \pi' w - \int_{\Delta_1} \Lambda(x,y,z) w_y w \pi^2 - \int_{\Delta_1} \alpha x w_x w \pi^2 \\
& \quad + \int_{\Delta_1} y w_z w \pi^2 - \beta \int_{\Delta_1} v_y w \pi^2 - \alpha \int_{\Delta_1} v_x w \pi^2 + \int_{\Delta_1} \rho_1 \eta_x w \pi^2\\ 
& \quad + \int_{\Delta_1} (\rho_1)_x \eta w \pi^2 + \int_{\Delta_1} \eta_{xy} \rho_2(y)w \pi^2
\end{align*}
which means
\begin{eqnarray*}
\frac{1}{2} \int_{\Delta_1} (w_x \pi )^2 + \frac{1}{2} \int_{\Delta_1} (w_y \pi)^2 & = & - \int_{\Delta_1} w_y \pi \pi' w - \int_{\Delta_1} \Lambda(x,y,z) w_y w \pi^2 - \int_{\Delta_1} \alpha x w_x w \pi^2 \\
& & + \int_{\partial \Delta_1} y \pi^2 w^2 \vec{n}(z) d \sigma - \beta \int_{\Delta_1} v_y w \pi^2 - \alpha \int_{\Delta_1} v_x w \pi^2\\
& & + \int_{\Delta_1} \rho_1  w^2 \pi + \int_{\Delta_1} (\rho_1)_x \eta w \pi^2 + \int_{\Delta_1} \eta_{xy} \rho_2(y)w \pi^2\\
& & + \int_{\Delta_1} (w_y \pi) \rho_2(y) w - \int_{\Delta_1} v_x \pi' \rho_2(y) w. 
\end{eqnarray*}
It is easy to verify that $v_{xx} \pi, v_{xy} \pi \in L^2(\Delta_1)$.\\[2mm]
Denote $\tilde{w}:= v_y$. Deriving equation (\ref{chap6:regu1}), with respect to $x$, we obtain equality (\ref{chap6:regu3}),
\begin{equation}
\label{chap6:regu3}
-\frac{1}{2}\tilde{w}_{xx} -\frac{1}{2}\tilde{w}_{yy} + c_0 \tilde{w} + \Lambda(x,y,z) \tilde{w}_y + \alpha x \tilde{w}_x - y \tilde{w}_z = v_z + \rho_1 \eta_y + (\rho_1)_y \eta + \eta_{yy} \rho_2(y) + \eta_{y} \rho_2(y)'.
\end{equation}
Then, we test equality (\ref{chap6:regu3}) with $\pi^2 \tilde{w}$, we obtain
\begin{eqnarray*}
\frac{1}{2} \int_{\Delta_1} (\tilde{w}_x \pi)^2 + \frac{1}{2} \int_{\Delta_1} (\tilde{w}_y \pi)^2 & = & - \int_{\Delta_1} \tilde{w}_y \pi \pi' \tilde{w} - \int_{\Delta_1} c_0 (w \pi)^2 - \int_{\Delta_1} \Lambda(x,y,z) \tilde{w}_y \tilde{w} \pi^2 - \int_{\Delta_1} \alpha x \tilde{w}_x \tilde{w} \pi^2\\
& & + \int_{\Delta_1} v_z \tilde{w} \pi^2 + \int_{\Delta_1} y \tilde{w}_z \tilde{w} \pi^2 + \int_{\Delta_1} \eta_y \rho_1 \tilde{w} \pi^2 +  \int_{\Delta_1} \eta (\rho_1)_y \tilde{w} \pi^2\\
& & + \int_{\Delta_1} \eta_{yy} \rho_2 \tilde{w} \pi^2 + \int_{\Delta_1} \eta_y (\rho_2)'\pi^2 \tilde{w}  
\end{eqnarray*}
which means
\begin{align*}
 \frac{1}{2} \int_{\Delta_1} (\tilde{w}_x \pi)^2 + \frac{1}{2} \int_{\Delta_1} (\tilde{w}_y \pi)^2 & = - \int_{\Delta_1} \tilde{w}_y \pi \pi' \tilde{w} - \int_{\Delta_1} c_0 (w \pi)^2 - \int_{\Delta_1} \Lambda(x,y,z) \tilde{w}_y \tilde{w} \pi^2 \\
& - \int_{\Delta_1} \alpha x \tilde{w}_x \tilde{w} \pi^2 + \int_{\Delta_1} v_z \tilde{w} \pi^2 + \int_{\partial \Delta_1} y (\tilde{w} \pi)^2 \vec{n}(z) d \sigma + \int_{\Delta_1} \eta_y \rho_1 \tilde{w} \pi^2 \\
& +  \int_{\Delta_1} \eta (\rho_1)_y \tilde{w} \pi^2 + \int_{\Delta_1} (\tilde{w}_y - 2 \eta_y \pi' + \eta \pi'')\tilde{w} \pi + \int_{\Delta_1} \eta_y (\rho_2)'\pi^2 \tilde{w}.  
\end{align*}
Again, it is easy to verify that $v_{yx} \pi, v_{yy} \pi \in L^2(\Delta_1)$.
\end{proof}
\begin{lemma}
We have
\[
 \eta_{xz} y^2 \pi^4 \in L^2(\Delta_1) ; \quad \eta_{yz} y^2 \pi^3 \in L^2(\Delta_1).
\]
\end{lemma}
\begin{proof}

We test equality (\ref{chap6:regu2}) with $y^{2p-1} \pi^{2q} w_z$, we obtain
\begin{align*}
 -\frac{1}{2} \int w_{xx} (y^{2p-1} \pi^{2q} w_z) & =  \frac{1}{2} \int_{\Delta_1} w_x w_{zx} y^{2p-1} \pi^{2q}\\
 & = \frac{1}{4} \iint w_x^2(x,y,Y) y^{2p-1} \pi^{2q} d xd y  - \frac{1}{4} \iint w_x^2(x,y,-Y) y^{2p-1} \pi^{2q} \d x \d y\\[2mm]
 -\frac{1}{2} \int_{\Delta_1} w_{yy} (y^{2p-1} \pi^{2q} w_z) & =  \frac{1}{2} \int_{\Delta_1} w_y w_{zy} y^{2p-1} \pi^{2q} + \int_{\Delta_1} w_y w_z \big{(} (2p-1)y^{2p-2} \pi^{2q} + y^{2p-1} 2q \pi^{2q-1} \pi'\big{)}\\
 & =  \frac{1}{4} \iint w_y^2(x,y,Y) y^{2p-1} \pi^{2q} d xd y  - \frac{1}{4} \iint w_y^2(x,y,-Y) y^{2p-1} \pi^{2q} \d x \d y\\
 & + (p-\frac{1}{2}) \int_{\Delta_1} (w_y \pi) (w_z y^p \pi^q) (y^{p-2} \pi^{q-1}) + q \int_{\Delta_1} (w_y \pi) (w_z y^p \pi^q) (y^{p-1} \pi^{q-1} \pi').  
\end{align*}
We have
\begin{align*}
\int_{\Delta_1} (y^p \pi^q w_z)^2 & \leq  \frac{1}{4} \iint \lbrace w_x^2(x,y,Y) + w_y^2(x,y,Y) \rbrace y^{2p-1} \pi^{2q} \d x \d y\\
&  - \frac{1}{4} \iint \lbrace w_x^2(x,y,-Y) + w_y^2(x,y,-Y) \rbrace y^{2p-1} \pi^{2q} \d x \d y\\
& + (p-\frac{1}{2}) \int_{\Delta_1} (w_y \pi) (w_z y^p \pi^q) (y^{p-2} \pi^{q-1}) + q \int_{\Delta_1} (w_y \pi) (w_z y^p \pi^q) (y^{p-1} \pi^{q-2} \pi')\\
&  + \int_{\Delta_1} \Lambda(x,y,z) (w_y \pi) (w_z y^p \pi^q) y^{p-1} \pi^{q-1} + \int_{\Delta_1} \alpha x (w_x \pi) (w_z y^p \pi^q) y^{p-1} \pi^{q-1}\\
& + \int_{\Delta_1} \beta v_y (y^p \pi^q w_z) y^{p-1} \pi^q + \int_{\Delta_1} \alpha v_x (y^p \pi^q w_z) y^{p-1} \pi^q\\
& - \int_{\Delta_1} \rho_1 v_x  (y^p \pi^q w_z) y^{p-1} \pi^{q-1} + \int_{\Delta_1} (\rho_1)_x \eta (y^p \pi^q w_z) y^{p-1} \pi^q\\
& - \int_{\Delta_1} \rho_2 (v_{xy} \pi -v_x \pi') (y^p \pi^q w_z) y^{p-1} \pi^{q-3} .  
\end{align*}
For $p \geq 2$ and $ q \geq 3$, some calculations give $v_{xz} y^2 \pi^3 \in L^2(\Delta_1)$.\\[2mm]
Now, we test equality (\ref{chap6:regu3}) with $y^{2p-1} \pi^{2q} \tilde{w}_z$, we obtain
\begin{align*}
\int_{\Delta_1} (y^p \pi^q \tilde{w}_z)^2 & \leq \frac{1}{4} \iint \lbrace \tilde{w}_x^2(x,y,Y) + \tilde{w}_y^2(x,y,Y) \rbrace y^{2p-1} \pi^{2q} \d x \d y\\
& \quad - \frac{1}{4} \iint \lbrace \tilde{w}_x^2(x,y,-Y) + \tilde{w}_y^2(x,y,-Y) \rbrace y^{2p-1} \pi^{2q} \d x \d y\\
& \quad + (p-\frac{1}{2}) \int_{\Delta_1} (\tilde{w}_y \pi) (\tilde{w}_z y^p \pi^q) (y^{p-2} \pi^{q-1}) + q \int_{\Delta_1} (\tilde{w}_y \pi) (\tilde{w}_z y^p \pi^q) (y^{p-1} \pi^{q-2} \pi')\\
& \quad + \int_{\Delta_1} c_0 \tilde{w} (y^p \pi^q \tilde{w}_z) y^{p-1} \pi^q\\
& \quad + \int_{\Delta_1} \Lambda(x,y,z) (\tilde{w}_y \pi) (\tilde{w}_z y^p \pi^q) y^{p-1} \pi^{q-1} + \int_{\Delta_1} \alpha x (\tilde{w}_x \pi) (\tilde{w}_z y^p \pi^q) y^{p-1} \pi^{q-1}\\
& \quad - \int_{\Delta_1} v_z (y^p \pi^q \tilde{w}_z) y^{p-1} \pi^q - \int_{\Delta_1} \rho_1 (v_y - \eta \pi') (y^p \pi^q \tilde{w}_z) y^{p-1} \pi^{q-1}\\
& \quad - \int_{\Delta_1} (\rho_1)_y \eta  (y^p \pi^q \tilde{w}_z) y^{p-1} \pi^q  - \int_{\Delta_1} (\rho_1)_x \eta (y^p \pi^q \tilde{w}_z) y^{p-1} \pi^q.\\
& \quad - \int_{\Delta_1} \rho_2 (y)\eta_{yy} \pi^2 (y^p \pi^q \tilde{w}_z) y^{p-1} \pi^{q-2}   
\end{align*}
For $p \geq 2$ and $ q \geq 2$, some calculations gives $v_{yz} y^2 \pi^2 \in L^2(\Delta_1)$.
\end{proof}
\begin{lemma}For $p$ and $q$ large enough, we have
\[
 \eta_{xyz} y^p \pi^q \in L^2(\Delta_1) ; \quad \eta_{yyz } y^p \pi^q \in L^2(\Delta_1); \quad  \rho(z) \eta_{xxz} y^p \pi^q \in L^2(\Delta_1) ; \quad \rho(z) \eta_{zz } y^p \pi^q \in L^2(\Delta_1).
\]
\end{lemma}



\begin{thebibliography}{10}
              
\bibitem{BenTuri1}
\newblock A. Bensoussan, J. Turi, 
\newblock \emph{Degenerate Dirichlet Problems Related to the Invariant Measure of Elasto-Plastic Oscillators},
\newblock Applied Mathematics and Optimization, \textbf{58(1)} (2008), 1--27.
     
\bibitem{BenTuri2} 
\newblock A. Bensoussan, J. Turi,
\newblock \emph{On a Class of Partial Differential Equations with Nonlocal Dirichlet Boundary Conditions}, 
\newblock Appl. Num. Par. Diff. Eq. {\bf 15}, pp. 9-23.

\bibitem{Doob}
\newblock J.L. Doob,
\newblock \emph{Stochastic Processes}, 
\newblock Wiley, New York,1953.

\bibitem{Feau1} 
\newblock C.Feau, 
\newblock \emph{Les m\'ethodes probabilistes en m\'ecanique sismique. Applications aux calculs de tuyauteries fissur\'ees},
\newblock Th\`ese CEA-Universit\'e d'Evry.

\bibitem{Feau2} 
\newblock C. Feau,
\newblock \emph{Probabilistic response of an elastic perfectly plastic oscillator under Gaussian white noise},
\newblock Probabilistic Engineering Mechanics, \textbf{23(1)} (2008),36--44.

\bibitem{KarSchar} 
\newblock D. Karnopp, T.D. Scharton, 
\newblock \emph{Plastic deformation in random vibration},
\newblock The Journal of the Acoustical Society of America,  \textbf{39} (1966), 1154-61.

\bibitem{Khasminskii} 
\newblock Khasminkii, R.Z.,
\newblock \emph{Stochastic Stability of Differential Equations},
\newblock  Sitjhoff and Noordhoff, The Netherlands (1980).    

\end{thebibliography}
\end{document}